\documentclass[PhD,binding=0.6cm]{sapthesis}

\usepackage{microtype}
\usepackage[utf8]{inputenx}

\usepackage{hyperref}
\hypersetup{pdftitle={Minimal Entropy of $3$-manifolds},pdfauthor={Erika Pieroni}}

\usepackage{lipsum}
\usepackage{curve2e}
\definecolor{gray}{gray}{0.4}

\title{Minimal Entropy of $3$-manifolds}
\author{Erika Pieroni}
\IDnumber{1695550}
\course[Mathematics]{Matematica}
\courseorganizer{Scuola di dottorato Vito Volterra}
\cycle{XXXI}
\submitdate{31 Ottobre 2018}
\copyyear{2019}
\advisor{Prof. Andrea Sambusetti}
\authoremail{pieroni@mat.uniroma1.it, erikapieroni24@gmail.com}

\examdate{18 January 2019}
\examiner{Prof. Roberto Frigerio}
\examiner{Prof. Alessandro Savo}
\examiner{Prof. Juan Souto}
\versiondate{\today}

\usepackage[italian,english]{babel}
\usepackage{comment}
\usepackage{mathtools}
\usepackage[dvipsnames]{xcolor}
\usepackage{verbatim}
\usepackage{amsmath}
\usepackage{amssymb}
\usepackage{amsfonts}
\usepackage{amsthm}
\usepackage{mathrsfs}

\usepackage[autostyle]{csquotes}
\usepackage{bbm }
\usepackage{cancel}
\usepackage{hyperref}
\usepackage{tikz}
\usetikzlibrary{matrix}
\usepackage{enumitem}
\usepackage{tikz-cd}

\excludecomment{definizioniEteoria}
\excludecomment{pezzotagliato}
\excludecomment{Souto}
\excludecomment{LeebConditions}

\newcommand{\R}{\mathbb{R}}
\newcommand{\C}{\mathbb{C}}

\newcommand{\f}{\rightarrow}

\newcommand{\JSJ}{\textit{JSJ}}

\newcommand{\psiz}{\psi_{0,t_0}}
\newcommand{\psitz}{\psi_{1,t_0}}
\newcommand{\psittz}{\psi_{2,t_0}}

\newcommand{\ent}{\operatorname{Ent}}
\newcommand{\minent}{\operatorname{MinEnt}}

\newcommand{\tr}{\operatorname{Tr}}

\newcommand{\inj}{\operatorname{inj}}
\newcommand{\diam}{\operatorname{diam}}

\newcommand{\Id}{\operatorname{Id}}
\newcommand{\jac}{\operatorname{Jac}}

\newcommand{\id}{\operatorname{id}}
\newcommand{\Vol}{\operatorname{Vol}}

\newcommand{\bary}{\operatorname{bar}}
\newcommand{\Isom}{\operatorname{Isom}}
\newcommand {\piorb}{\pi_1^{orb} }
\newcommand{\EntVol}{\operatorname{EntVol}}
\newcommand{\VolEnt}{\operatorname{VolEnt}}

\makeatletter

\newcommand{\Rmnum}[1]{\expandafter\@slowromancap\romannumeral #1@}
\makeatother

\newtheorem{thm}{Theorem}[chapter]
\newtheorem*{namedthm*}{\namedthmname}
\newenvironment{namedthm}[1]
{\newcommand\namedthmname{#1}\begin{namedthm*}}
	{\end{namedthm*}}

\newtheorem{prop}[thm]{Proposition}
\newtheorem{lem}[thm]{Lemma}
\newtheorem{cor}[thm]{Corollary}

\theoremstyle{remark}
\newtheorem*{claim}{Claim}

\newtheorem*{oss}{Observation}
\newtheorem{rmk}[thm]{Remark}

\theoremstyle{definition}
\newtheorem{defn}[thm]{Definition}

\begin{document}
	
	\frontmatter
	
	\maketitle
	\dedication{A Babbo e Mamma}

	\begin{abstract}
		\noindent
		We compute the Minimal Entropy for every closed, orientable $3$-manifold, showing that its cube  equals the sum of the cubes of the minimal entropies of each hyperbolic component arising from the $JSJ$ decomposition of each prime summand.
		As a consequence we show that the cube of the Minimal Entropy is additive with respect to both the prime and the $JSJ$ decomposition.  This answers a conjecture asked by Anderson and Paternain for irreducible manifolds in \cite{anderson2003minimal}.
	\end{abstract}

	\begin{acknowledgments}[Ringraziamenti]
		\begin{otherlanguage*}{italian}

	\noindent	Ringrazio il mio relatore, Andrea Sambusetti, per avermi proposto questo appassionante problema di tesi, un ottimo ponte verso la Geometria Riemanniana per una persona con una formazione topologica come ero io all’inizio del Dottorato, e per le discussioni di questi anni.\\

	\noindent	Ringrazio Sylvain Gallot, per avermi offerto una prospettiva privilegiata sul metodo del baricentro da lui sviluppato, e per l’attenzione e la pazienza con cui ha ascoltato la prima parte del mio lavoro.\\

		\noindent Grazie a Juan Souto, per aver accettato di fare da referee per questa tesi, per la generosità con cui condivide le sue idee, e per l’entusiasmo ed il trasporto con cui discute di Matematica: uno degli esempi più sinceri che mi sia mai capitato di incontrare.\\

	\noindent	Grazie a Roberto Frigerio, per la cura e l’attenzione con cui ha revisionato e valorizzato il mio lavoro e per i preziosi consigli. E grazie, ancora, per avermi saputo motivare durante i miei primi anni universitari.\\
		
	\noindent	Un grazie affettuoso a Marco Isopi, per la sua grande gentilezza e profonda umanità. Sono davvero felice di averlo incontrato lungo questo percorso.\\
		
	\noindent	Un grazie di cuore al collega Filippo Cerocchi. Per l’aiuto che mai mi ha fatto mancare in questi anni, per avermi trasmesso il suo entusiasmo per la Matematica, per la sconfinata generosità che ha dimostrato nei miei confronti, per la sua competenza e sensibilità, per avermi saputo motivare nei momenti di sconforto. Perché studiare insieme le 3-varietà è stata la parte più bella della mia carriera accademica.\\
		
	\noindent	Ringrazio tutti i colleghi del Dottorato, che con la loro presenza hanno reso più briosa questa esperienza. Una menzione speciale va a Veronica, per il suo affetto e la sua generosità. E per tutti i tè sorseggiati insieme.\\
		
	\noindent	Grazie ai NoKoB, per tutte le risate che ci siamo fatti ai margini di questa esperienza (o meglio al bordo).\\
		
	\noindent	Vorrei ringraziare Silvia, per la nostra bella amicizia nata durante il Dottorato. Sei davvero un’Amica con la A maiuscola.\\
		
	\noindent	Un grazie speciale ai miei genitori, che pur non comprendendo appieno le mie scelte, e seppur a volte soffrendone, non mi hanno mai fatto mancare il loro supporto. Grazie per tutto ciò che mi hanno insegnato, lezioni e valori che vanno ben oltre qualsiasi insegnamento accademico. \\
		
	\noindent	Infine, un ulteriore grazie a Filippo. Per tutto ciò che di Matematica e di Vita abbiamo condiviso in questi anni.\\
	\noindent	“Galeotto fu il libro e chi lo scrisse, quel giorno più non vi leggemmo avante”.\\
	\noindent (see \cite{scott1983geometries} and \cite{bonahon2002geometric} for further details).
		\end{otherlanguage*}
	\end{acknowledgments}

	\tableofcontents
	\listoffigures
	\newpage
	\pagestyle{empty}
	\chapter*{Notation}
	We summarize here the notation used in the thesis.
	
	\section*{Chapter \ref{section_a_conjectural_minimizing_sequence_irreducible}}
	
\noindent 	--- $Y$  orientable, closed, irreducible $3$-manifold, with at least one hyperbolic $JSJ$-component and non-trivial $JSJ$ decomposition.\\
\noindent	--- $X_1, \dots, X_n$ the $JSJ$-components of $Y$.\\
\noindent --- $X_1, \dots, X_k$ the hyperbolic $JSJ$-components, and $(hyp_i)$ the complete, finite volume hyperbolic metric on $X_i$.\\
\noindent --- $\{g_{\delta}\}$ a family of $C^2$ Riemannian metrics defined on $Y$.\\
\noindent --- $\{k_{\delta, \zeta}^j\}$ a family of $C^2$ Riemannian metrics on the Seifert component $X_j$ of $Y$. \\
\noindent --- $\{hyp_{i, \delta}^0\}$ a family of $C^2$ Riemannian metrics on the hyperbolic component $X_i$ of $Y$.\\

\section*{Chapter \ref{section_a_conjectural_minimizing_sequence_reducible}}

\noindent --- $Y$ orientable, closed, reducible $3$-manifold. \\
\noindent --- $Y_1, \dots, Y_m$ the prime summands of $Y$, \textit{i.e.} $Y= Y_1 \# \dots, \# Y_m$.\\
\noindent --- $X_i^1, \dots, X_i^{n_i}$ the hyperbolic components of $Y_i$, and $hyp_i^j$ the corresponding complete, finite volume, hyperbolic metrics on them. \\
\noindent --- $\{g_{L,r}\}$ a family of $C^{\infty}$ Riemannian metrics defined on $Y$.

\section*{Chapter \ref{section_a_lower_bound_irreducible}} 
\noindent 	--- $Y$  orientable, closed, irreducible $3$-manifold, with at least one hyperbolic $JSJ$-component and non-trivial $JSJ$ decomposition.\\
\noindent	--- $X_1, \dots, X_n$ the $JSJ$-components of $Y$.\\
\noindent  --- $g$ a Riemannian metric on $Y$.\\
\noindent --- $(\widetilde Y, \tilde g)$ the Riemannian universal cover of $(Y,g)$.\\
\noindent --- $ \tilde d$ the Riemannian distance on $(\widetilde Y, \tilde g )$.\\
\noindent --- $(X,g_0)$ a non-positively curved Riemannian manifold, and $( \widetilde X, \tilde g_0)$ its Riemannian universal covering, which is a  complete, simply connected Riemannian manifold with non-positive curvature.\\
\noindent --- $\tilde d_0$ the Riemannian distance on $\widetilde X, \tilde g_0$.\\
\noindent --- $(\tilde \rho_0)_z(\cdot)= \tilde d_0(z, \cdot)$, the distance function from a fixed point $z$.\\
\noindent --- $\{(Y, g_{\delta})\}$ the family of target spaces for the barycentre method.\\
\noindent --- $f:(Y, g) \rightarrow (Y, g_{\delta})$ the initial map, and $\tilde f: (\widetilde Y, \tilde g) \rightarrow (\widetilde Y, \tilde g_{\delta})$ its lift.\\
\noindent --- $\mu_{c,y}$ a family of measures on $ \widetilde Y$, namely $d\mu_{c, y}(y')= e^{-c \tilde d(y,y')}dv_{\tilde g}(y')$.\\
\noindent --- $  F_{c, \delta}: (Y,g) \rightarrow (Y, g_{\delta})$ the family of maps constructed via the barycentre method.

\section*{Chapter \ref{section_a_lower_bound_reducible}}

\noindent --- $Y$ orientable, closed, reducible $3$-manifold. \\
\noindent --- $Y_1, \dots, Y_m$ the prime summands of $Y$, \textit{i.e.} $Y= Y_1 \# \dots, \# Y_m$.\\
\noindent --- $X_i^1, \dots, X_i^{n_i}$ the hyperbolic components of $Y_i$, and $hyp_i^j$ the corresponding complete, finite volume, hyperbolic metrics on them. \\
\noindent --- $Y_1, \dots, Y_k$ the prime summands having at least a hyperbolic $JSJ$ component in their decomposition.\\
\noindent --- $Y_{k+1}, \dots, Y_m$ the prime summands which are graph manifolds.\\
\noindent ---$X=\bigvee_{y_1=\cdots=y_k} Y_i$ is
 the space obtained by shrinking $Y\smallsetminus \bigcup_{i=1}^k int(\bar Y_i)$ to a single point $x_0$.\\
 \noindent ---$(X, d_{\delta})$ the family of (locally $CAT(0)$) singular target spaces for the barycentre method.\\
 \noindent --- $(\widetilde X, \tilde d_{\delta})$ the ($CAT(0)$) Riemannian  universal covering of $(X, d_{\delta})$.\\
 \noindent ---$f:(Y,g) \rightarrow (X, d_{\delta})$ the initial map, and $\tilde f: (\widetilde Y, \tilde g) \rightarrow (\widetilde X, \tilde d_{\delta})$ its lift.\\
 \noindent --- $d\mu_{c, y}(y')= e^{-c\, \tilde d(y,y')}dv_{\tilde g}(y')$ a family of measures on $\widetilde Y$.\\
 \noindent ---$F_{c, \delta}: (Y, g) \rightarrow (X, d_{\delta})$ the family of maps constructed via the barycentre method.

	\mainmatter
	
	\noindent

	\bigskip

\chapter{Introduction}

Let $(Y,g)$ be a closed or finite volume Riemannian manifold of dimension $n$.
The \textit{volume entropy} of $Y$ is defined as:
\small
$$\ent(Y,g)= \lim_{R \to \infty} \dfrac{\log (\Vol (B_{\widetilde{Y}}(\widetilde{y},R)))}{R} ,$$
\normalsize
where $\widetilde{y}$ is any point in the Riemannian universal cover $\widetilde{Y}$ of $Y$, and where $\Vol (B_{\widetilde{Y}}(\widetilde{y},R))$ is the volume of the ball of radius $R$ centered at $\widetilde{y}$. It is well known from the work of Manning (see \cite{manning1979topological}) that this quantity does not depend on $\widetilde{y}$, and in the compact case is related to the \textit{topological entropy} of the geodesic flow of $(Y,g)$, denoted $\ent_{top}(Y,g)$; namely, $\ent(Y,g) \le \ent_{top}(Y,g)$, and the equality holds if $(Y,g)$ is non-positively curved.
On the other hand, $\ent(Y,g)$ can even be greater then $\ent_{top}(Y,g)$ for negatively curved, finite volume manifolds, see \cite{dal2009growth}. \\
It is readily seen that the volume entropy scales as the inverse of $\Vol(Y,g)^n$, hence one obtains a scale invariant by setting:
$$\VolEnt(Y,g)=  \Vol(Y,g)^{1/n} \ent(Y,g).$$
\noindent The \textit{minimal entropy} $\minent$ is thus defined as:
$$ \minent(Y)=\inf_{g}
\VolEnt(Y,g)=\inf_{\Vol(Y,g)=1}  \ent(Y,g).$$
where $g$ varies among all Riemannian metrics on $Y$. It turns out that this is a homotopy invariant, as shown in \cite{babenko1993asymptotic}.

Notice that if $Y$ is not compact it makes sense to define  the  Minimal Entropy in the same way as for closed manifolds, except that we shall take the infimum over all the complete, Riemannian metrics of volume $1$ on $M$.

The computation of $\minent$, when nonzero, is in general hard, and is known only for very few classes of manifolds.
In the middle seventies, Katok computed the $\minent$ of every closed surface $S$ of genus $g \ge 2$, namely, $\minent (S)= 2 \pi  \lvert \chi(S) \rvert $, see \cite{katok1982entropy} (when $g \le 1, \minent=0$ trivially, the fundamental group being of sub-exponential growth). Moreover, he proved that, in this case, the infimum is indeed a minimum, and it is realized precisely by the hyperbolic metrics. Katok's main interest was computing the \textit{topological entropy} (see \cite{manning1979topological}) and the proof uses techniques of dynamical systems. Moreover, its proof strongly relies on the Uniformization Theorem for surfaces, so cannot be generalized to higher dimensions.


The following major step was made by Besson, Courtois and Gallot, who in the middle '90s proved the following theorem, showing in particular that locally symmetric metrics of strictly negative curvature are characterized by being minima of the volume-entropy functional on any closed manifold $X$ (provided that such a metric exists on $X$), in any dimension, as conjectured by Katok and Gromov, see \cite{gromov1983filling}. 

\begin{thm} [\cite{besson1995entropies}] \label{theorem_BCG}
	Let $Y$ and $X$ be closed, connected, oriented manifolds of the same dimension $n$, and let $f: Y \rightarrow X$ be a continuous map of non-zero degree. Suppose $X$ endowed with a locally symmetric metric $g_0$ with negative curvature. Then for every Riemannian metric on $Y$ the following inequality holds:
	\begin{equation} \label{equation_BCG}
\VolEnt(Y,g) ^n \ge |\deg f|  \VolEnt(X, g_0)^n.
	  	\end{equation}
	Furthermore, if $n \ge 3$ the equality holds if and only if $(Y,g)$ is
	homothetic to a Riemannan covering $f_0:Y \rightarrow X$.
\end{thm}


%

 The main tool behind the result of Besson, Courtois and Gallot is the so called \textit{barycenter method}, a technique that was broadly used in other works by several authors, see \cite{besson1996minimal}, \cite{besson1999lemme}, \cite{boland2005volume}, \cite{cerocchientropyandconvergence}, \cite{merlin2014minimal},  \cite{peigne2017entropy}, \cite{sambusetti1999minimal},  \cite{storm2006minimal}.
 Recently, Merlin computed the minimal entropy of the compact quotients of the polydisc $(\mathbb{H}^2)^n$, using the original idea, already present in \cite{besson1995entropies}, of explicitly constructing a \textit{calibrating form} (see \cite{merlin2014minimal}).\\\\
\indent In this thesis we will be concerned with the case where $Y$ is a closed, orientable $3$-manifold. Every closed, orientable $3$-manifold admits a $2$-step topological decomposition: namely, it can be firstly decomposed as the connected sum of \textit{prime} manifolds (\textit{i.e.}, not decomposable as non trivial connected sum), and then each of these summands can be further decomposed, along suitable embedded incompressible\footnote{An embedded surface $S \subset M$ in a closed, orientable $3$-manifold  is called \textit{incompressible} if, for every disc $D \subset M$ such that $\partial D= S \cap D$, then $\partial D$ bounds a disc in $S$.} tori, in \textit{$JSJ$ pieces}. Furthermore, each of the pieces (manifolds with boundary) arising from this decomposition is either a \textit{Seifert manifold}, or is \textit{homotopically atoroidal} (see Section \ref{section_topology_and_geomety}) and its interior can be endowed with a complete, hyperbolic metric of finite volume: this is the content of Thurston's Hyperbolization Theorem in the Haken case (see  \cite{thurston82}, \cite{thurstonhyperbolic1}, \cite{thurstonhyperbolic2}, \cite{thurstonhyperbolic3}, \cite{otalhyperbolization}, \cite{otalhyperbolization1} and \cite{kapovichhyperbolicmanifolds}), now fully proved in all generality by the work of Perelman, see \cite{perelman2002entropy}, \cite{perelman2003finite}, \cite{perelman2003ricci}; see also the book \cite{bessieres2010geometrisation} for a more detailed discussion of Perelman's results. 
In this $3$-dimensional context, the work by Anderson and Paternain  \cite{anderson2003minimal}  implies that $\minent$ vanishes for every Seifert manifold ---and, more generally, for every \textit{graph manifold}, \textit{i.e.}, a prime $3$-manifold admitting no hyperbolic pieces in its $JSJ$ decomposition. However, the general behaviour of $\minent$ with respect to the prime and the $JSJ$ decomposition was not known. In \cite{anderson2003minimal}, Anderson and Paternain ---probably inspired by the behaviour of the topological entropy--- conjectured that, if $Y$ is a closed, irreducible and orientable $3$-manifold, its $\minent$ equals the maximum of the entropies of the complete, possibly non compact hyperbolic pieces. While this turns out to be true in the case where there is only one $JSJ$-component of hyperbolic type ---as a consequence of our Theorem \ref{thm_minimal_entropy}--- we will show that the cube of $\minent$ is actually \emph{additive} in the prime summands and in the $JSJ$ pieces.
\newpage
 
 \begin{thm} 
	\label{thm_minimal_entropy}
	Let $Y$ be a closed, orientable, connected $3$-manifold:
	\begin{itemize}
		\item [a)]
			If $Y$ is irreducible, let $Y=  X_1 \cup \dots \cup X_k \cup X_{k+1} \cup \dots X_n$ be its decomposition in $JSJ$ components, where for $i=1, \dots, k$ the $X_i$'s are the hyperbolic $JSJ$ components, whose interior can be endowed with the complete, finite volume, hyperbolic metric $hyp_i$ (unique up to isometry), while for $i= k+1, \dots, n$ the component $X_i$ is Seifert.
	Then:
	\begin{equation} \label{eq_min_ent_irreducible}
	\minent(Y)^3= \sum_{i=1}^k \VolEnt(int(X_i), hyp_i)^3= 2^3 \cdot \sum_{i=1}^{k} \left(\Vol(int(X_i),hyp_i)\right).
	\end{equation}
	\item [b)]
	If Y is reducible, namely $Y= Y_1  \#\dots \#Y_m$ then:
	\begin{equation} \label{eq_min_ent_reducible}
 \minent(Y)^3= \sum_{j=1}^{m} \minent(Y_j)^3.
 \end{equation}
 	\end{itemize}
\end{thm}

For open manifolds which admit complete Riemannian metrics of finite volume it makes sense to define the Volume Entropy along the same line of the closed case restricting to the metrics of finite volume. It is worth noting that, with this definition, the Minimal Entropy is allowed to be equal to $+\infty$.
 In \cite{storm2006minimal} Storm proved that when $M$ is a manifold which admits a complete, finite volume, locally symmetric metric of rank $1$, then such a metric realizes the Minimal Entropy of $M$. Furthermore, if an open manifold admits volume-collapsing with bounded curvature (as it is the case of Seifert fibred manifolds fibering over hyperbolic orbifolds) its Minimal Entropy vanishes. Moreover, for those open manifolds admitting a finite volume metric and whose fundamental group has polynomial growth it is readily seen that the Minimal Entropy is equal to $0$, a volume-collapsing sequence being provided by rescaling the finite volume metric by smaller and smaller constants (it is the case of $K\widetilde{\times}I$).  Hence, in view of the definition of Minimal Entropy for open manifolds and thanks to  Storm's work we see that part a) of Theorem \ref{thm_minimal_entropy} can be read as an Additivity Theorem.\\
 
 \noindent{\bf Additivity Theorem.} {\it Let $Y$ be an irreducible, closed, orientable $3$-manifold and let $X_1,..., X_n$ be the $JSJ$ components of $Y$. Then:}
 	$$\minent(Y)^3=\sum_{i=1}^n\minent(int(X_i))^3.$$
 	

 It is worth observing that our computation allow us to state that, for every connected, orientable, closed $3$-manifold $Y$, the simplicial volume $\lVert Y \rVert$ and  $\minent(Y)^3$ are thus proportional invariants, the constant of proportionality depending only on the dimension $n= 3$. 
 \footnote{Recall that for every closed, orientable $n$-manifold $M$, the {\it simplicial volume} is defined as:
 $\lVert M\rVert=\inf \sum_{i=1}^\infty |a_i|$
 where $[M] =\left[\sum_{i=1}^\infty a_i\sigma_i\right]\in H_n(M,\mathbb R)$ is the fundamental class of $M$ and the infimum is taken over all real cycles $\sum_{i=1}^\infty a_i\sigma_i$ representing $[M]$.} Actually Gromov's Additivity Theorem (see \cite{frigerio2017bounded}, Theorem 7.6 and references in there), and Gromov's computation of the simplicial volume of finite volume hyperbolic manifolds (\cite{gromov1982volume}, \S0.3), together with the vanishing of the (relative) simplicial volume on Seifert fibred manifolds and the additivity of the simplicial volume with respect to the connected sum (see \cite{gromov1982volume}, \S3.5) yield the well known formula:
 \bigskip
 
\textbf{Gromov's Formula} \label{thm_simplicial_vol_3mfd}
 \textit{For every closed, orientable, connected $3$-manifold $Y$, whose prime decomposition is $Y=Y_1\#\cdots \# Y_m$, for every $j=1,..., m$  denote the hyperbolic $JSJ$ components of $Y_j$ by $X_1^j,...,X_{k_j}^j$, endowed with the complete, finite volume metric $hyp_i^j$. Then:
 	$$\lVert Y\rVert=v_3\cdot\sum_{j=1}^m\sum_{i=1}^{k_j}\Vol(int(X_{i}^j), hyp_i^j).$$
 	where $v_3$ is the volume of an ideal  regular geodesic simplex\footnote{ A geodesic simplex in $\mathbb H^n$ is said to be {\it ideal} if all its vertices lie in $\partial H^n$, and {\it regular} if every permutation of its vertices is induced by an isometry of $\mathbb H^n$.} in $\mathbb H^3$. }
 \bigskip

 Hence, from Theorem \ref{thm_minimal_entropy} and Formula \ref{thm_simplicial_vol_3mfd} we obtain the following:
 \begin{cor}
 	For every closed, orientable $3$-manifold $Y$ we have:
 	$$ \lVert Y \rVert= \frac{v_3}{8}  \minent(Y)^3.$$
 \end{cor}
 
\bigskip

Some of the ideas to prove part $1$ of Theorem \ref{thm_minimal_entropy} have been developed in the PhD thesis of Souto (see \cite{soutophd}), with whom we are in debt for explaining us the overall strategy. However Souto's work relies on a result (presented as a  consequence of Leeb's argument \cite{leeb19953}), which is given without proof (see Proposition 8, \cite{soutophd}).
As we shall explain later in Section \ref{section_the_extension_problem},  Leeb's construction of non-positively curved metrics on irreducible $3$-manifolds with non-trivial $JSJ$-decomposition and at least one hyperbolic $JSJ$-component does not ensure, alone, a sufficiently precise control on the geometric invariants 
which are needed for the estimate of $\minent$ in the irreducible case; we shall give all the necessary details for the construction of these metrics in Chapter \ref{section_a_conjectural_minimizing_sequence_irreducible},  where a precise estimate on the sectional curvatures and the volumes of each $JSJ$ component is given. To prove $(a)$, we endow $Y$ with a sequence of $C^2$, non-positively curved Riemannian metrics, which look more and more like hyperbolic metrics with cusps on the $JSJ$ components of hyperbolic type, and  shrink the volumes of the non-hyperbolic $JSJ$-components. Then we adapt a slightly different version of the barycentre method, using the square of the distance function instead of the Busemann functions (as developed in \cite{sambusetti1999minimal}) which avoids the analythical difficulty of considering measures on the ideal boundary, and obtaining the sharp lower estimate of $\minent$ given by Equations (\ref{eq_min_ent_irreducible}) and (\ref{eq_min_ent_reducible}).\\

The case $(b)$ of reducible manifolds $Y=Y_1 \# \cdots  \#Y_m$ is achieved through a direct computation applied to a suitable, geometrically well-chosen sequence of metrics which expand the horospherical collars to long and arbitrarily thin tubes: we can then decompose optimally the Poincaré series of $(Y_i,g_i)$ as a geometric series with common ratio given by the Poincaré series of the summands $Y_i$.\\

The paper is organized as follows. In Chapter \ref{chapter_basics} we recall some basics on the  topology of  Seifert manifolds and the construction of geometric metrics on them. In particular, in Section \ref{section_JSJ_decomposition} we focus on the $JSJ$-decomposition Theorem for irreducible manifolds, and the geometries carried by the $JSJ$ components. 
Chapter \ref{section_a_conjectural_minimizing_sequence_irreducible} is devoted to obtain the upper estimate of $\minent$ in the case where $Y$ is irreducible, by explicitly exhibiting a (conjecturally) minimizing sequence of metrics. 
In Chapter \ref{section_a_conjectural_minimizing_sequence_reducible} we will get the (sharp) upper estimate for reducible manifolds.
In Chapters \ref{section_a_lower_bound_irreducible} and \ref{section_a_lower_bound_reducible} we shall prove the lower estimate of $\minent$ in the irreducible and reducible case, respectively. In both cases, the cornerstone of the proof is the \textit{barycentre method} developed by Besson, Courtois and Gallot (see \cite{besson1995entropies}, \cite{besson1996minimal}, \cite{besson1999lemme}). In particular, in Chapter \ref{section_a_lower_bound_reducible} we will prove that the barycentre is also defined in the case where the target space is a  $CAT(0)$ metric space ---and not a non-positively curved manifold, as in the classical setting--- and apply it to a specific $CAT(0)$ space constructed starting from the manifold $Y$. Appendix \ref{Appendix_orbifolds}  is devoted to some standard notions about smooth and Riemannian orbifolds, which will be needed in the construction of the minimizing sequence of metrics. Appendix \ref{appendix_proposition} contains the proof of some technical lemmata and propositions exploited in Chapter \ref{section_a_conjectural_minimizing_sequence_irreducible}.

\noindent

\bigskip

\chapter{Basics of Geometry and Topology of $3$-manifolds}\label{chapter_basics}
\section{Topology of Seifert fibred manifolds}\label{section_seifert_fibered} \label{section_topology_and_geomety}
This section is devoted to introduce some basic facts and definitions about the topology of  Seifert manifolds. A complete treatise of this topic is far from the aim of this paper: for more details, we recommend the book by Thurston \cite{thurston1997three}, the surveys by Scott \cite{scott1983geometries}, Bonahon \cite{bonahon2002geometric}, and the book by Martelli \cite{martelligeometric}.\\
 \subsection{Fibered tori}
In this section we shall introduce Seifert fibred manifolds; our exposition is based on the work by Scott, see \cite{scott1983geometries} for full details.
Informally, a Seifert fibred manifold is a $3$-manifold which can be described as the disjoint union of circles, called fibres, satisfying some additional properties. In order to make this notion precise, we need to give some preliminary definitions.
A \textit{trivial fibred solid torus} is $T=S^1 \times D^2$ endowed with the standard foliation by circles, \textit{i.e.}, the foliation whose fibres are of type $S^1 \times \{y\}$ for $y \in D^2$.
A \textit{fibred solid torus} is a solid torus with a foliation by circles which is finitely covered (as a circle foliation) by a trivial fibred solid torus. This object can be realized as follows: take a trivial fibred solid torus, cut it along a meridian disk $\{x\} \times D^2$ for some point $x \in S^1$, rotate through $q/p$ of a full turn, where $p$ and $q$ are coprime integers, one of the two discs obtained via this procedure, and glue these discs back together, see Figure \ref{fig:ibredtorus32}. The resulting fibred solid torus $T(p,q)$ has a $p$-fold covering which is a trivial fibred solid torus.

\begin{figure}
	\centering
	\includegraphics[width=0.2\linewidth]{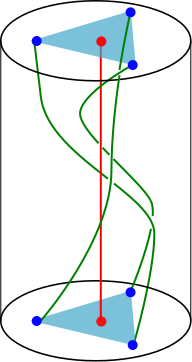}
	\caption[A fibred solid torus]{A fibred solid torus $T(2,3)$.}
	\label{fig:ibredtorus32}
\end{figure}

\subsection{Seifert fibered manifolds}
Now we recall some basic facts about Seifert manifolds. 
We can define a \textit{Seifert manifold} as a $3$-manifold $S$ with a decomposition into disjoint circles, called \textit{fibers}, such that every fibre has a neighbourhood in $S$ which is a union of fibres, and is fiberwise isomorphic to a fibred solid torus.
Clearly any circle bundle over a surface is a Seifert manifold. It follows by the definition that a Seifert manifold is foliated by circles. Furthermore, Epstein (see \cite{epstein1972periodic}) proved that the converse statement holds for compact, orientable $3$-manifolds; namely, any foliation by circles of a compact, orientable $3$-manifold is a Seifert fibration. \footnote{Epstein actually proved this result for compact $3$-manifolds, possibly not orientable. However, this requires to consider a more general notion of Seifert fibrations, admitting that the neighbourhood in $S$ of a fibre is a \textit{fibred Klein bottle}. For our purposes we are interested only in the orientable case.}
To start with, we want to determine when two fibered solid tori $T(p,q)$ and $T(p',q')$ are isomorphic (\textit{i.e.}, they are homeomorphic through a homeomorphism which preserves the fibration).
It follows from the previous description that in a fibred solid torus $T(p,q)$ all the fibres but the central one represent $p$ times the generator of $\pi_1(T(p,q))$, and wind $q$ times around the central fibre.
It is not hard to prove that if $T(p,q)$ and $T(p',q')$ are isomorphic, then $p=p'$ and $q \equiv q' \, \mod(p)$.
We observe that the operation of cutting $T(p,q)$ along a meridian disc and gluing it back with a full $2 \pi$ twist affects $q$ adding an integer multiple of $p$. Hence $p$ is an invariant of the fibred solid torus $T(p,q)$, while $q$ can be made an invariant of $T(p,q)$ provided that we take $0 \le q \le \frac{1}{2}p$.
Furthermore, if we orient the solid tori and care about other orientations, we need only to normalise $q$ so that $0 \le q <p$ in order to obtain an invariant of oriented fibred solid tori.   The invariants $(p,q)$ normalised in this way are called \textit{orbital invariants} of the central fibre. We use the following orientation convention. Given an oriented manifold $S$, its boundary $\partial S$ is endowed with the orientation such that the orientation of the boundary, together with the inward normal coincides with the orientation of $S$. Then, we assume that the oriented fibred solid torus $T(p,q)$ is constructed from the standard  oriented trivial fibred solid torus $S^1 \times D^2$ by performing a $2 \pi q/p$ anticlockwise twist on the $D^2$ factor.\footnote{We remark that reversing the orientation of both factors $S^1$ and $D^2$ provides an object which is orientation preservingly isomorphic to $T(p,q)$. }
Now we call a fibre \textit{regular} if it has a neighbourhood isomorphic to a trivial fibred solid torus, \textit{critical} otherwise. Since a fibred solid torus has at most one critical fibre, it turns out that the critical fibers are isolated. 
\subsection{Topological construction of Seifert fibered manifolds}
All Seifert manifolds can be constructed via the following procedure. Let $\Sigma_{g,h}$ be the genus $g$ surface with $h$ boundary components $d_1, \dots, d_h$ (with the convention that $g \in \mathbb{Z}$ is positive if $\Sigma_{g,h}$ is orientable and negative otherwise). Remove a collection of $k$ disks $D_1, \dots D_k$ from $\Sigma_{g,h}$, thus obtaining a surface $\Sigma_{g,h}^{\ast}$ having $h+k$ boundary components. Consider the (unique) orientable $S^1$-bundle over $\Sigma_{g,h}^{\ast}$, which is isomorphic to $\Sigma_{g,h}^{\ast} \times S^1$ or to $\Sigma_{g,h}^{\ast} \widetilde{\times} S^1$, where the product is twisted or not depending on $g$ being negative or not. In this way we get a $3$-manifold having $k+h$ toroidal boundary components. 
 Denote by $f$ the homotopy class of the fiber of this $S^1$-bundle --- which we will refer to as the 	\textit{regular fiber}--- and call $\gamma_i$ the boundary loop of the disk $D_i$ for $i=1, \dots, k$. Denote by $T_1, \dots, T_k$ the boundary tori of $\Sigma_{g,l}$ fibering on the $\gamma_i$.
We recall that a \emph{slope} on a torus $\mathbb{T}$ is the isotopy class of an unoriented, homotopically non trivial, simple closed curve. 
For every $i=1, \dots, k$ fix the basis $\langle f, \gamma_i \rangle$ of $H_1 \left(\mathbb{T}, \mathbb{Z}\right)\simeq \pi_1(\mathbb{T})$; then every slope can be written as $p \gamma_i+qf$ for a suitable pair of integer coefficients $(p,q)$.
Let $M$ be a $3$-manifold, and $\mathbb{T} \subseteq \partial M$ be a boundary torus. A \emph{Dehn filling} of $M$ along $\mathbb{T}$ is the operation of gluing a solid torus $D \times S^1$ to $M$ via a diffeomorphism. We call $M^{fill}$ the resulting manifold, which has a boundary component less then $M$.

\begin{prop} [10.1.2 , \cite{martelligeometric}]
	\label{prop_dehn_filling}
	Let $M$ be a $3$-manifold and $T \subset \partial M$ be a boundary component. 
	Perform a Dehn filling along $T$ such that the meridian $\partial D \times \left \lbrace x \right \rbrace $ is glued to some simple closed curve $\gamma \subset T$. Then the resulting manifold $M^{fill}$ only depends on the isotopy class of the unoriented curve $\gamma$.
	We shall say that the Dehn filling \emph{kills} the curve $\gamma$.
\end{prop}

Proposition \ref{prop_dehn_filling} implies that every pair of coprime integers $(p,q)$ determines a Dehn filling killing the slope $p \gamma_i+ q f$. Now for every boundary torus $\mathbb{T}_i$, where $i=1, \dots, k$ we choose a pair of coprime integers $(p_i, q_i)$, and perform the Dehn filling along $T_i$ killing the slope $p_i \gamma_i+ q f$.
 Carrying out this construction, we obtain a Seiert manifold having $l$ toroidal boundary components; we shall denote it by $S=S(g,h; (p_1,q_1), \dots, (p_k, q_k))$.
Calling $\bar{a}_i,\bar{b}_i,\bar{d}_i$ the canonical generators of $\pi_1(\Sigma_{g,h})$ when $g >0$ and $\bar{a}_i,\bar{d}_i$ the standard generators of $\Sigma_{g,h}$ when $g<0$, with in both cases $d_i$ being the loop representing the boundary components, the computation of $\pi_1(S)$ is straightforward by the Van Kampen Theorem:
\begin{equation}  \label{equation_presentation_orientable}
 g \ge 0  
\end{equation}
$$\small
\pi_1(S)=\left\langle a_1, b_1,..,a_g, b_g, c_1,..,c_k,d_1,..,d_h, f\left |\begin{array}{l}
\prod_{i=1}^{g}[a_i, b_i]\prod_{j=1}^{k} c_j\prod_{l=1}^{h} d_l=1\\ 
c_i^{p_i}=f^{q_i}\\ 
\left[ f, a_i\right] \!=\! \left[ f, b_i\right] \!=\! \left[ f, c_i\right] \!=\! \left[ f, d_i\right] \!=\!1\\
\end{array}
\right. \!\!\right\rangle$$

\normalsize

\begin{equation} \label{equation_presentation_not_orientable}
 g<0
 \end{equation}
\small
$$\pi_1(S)=\left\langle a_1,...\,a_g,c_1,...,c_k, d_1,...,d_h, f\,\left |\,\begin{array}{l}
\prod_{i=1}^{g}a_i^2\cdot\prod_{j=1}^{k} c_j\cdot\prod_{l=1}^{h} d_l=1\\ 
c_i^{p_i}=f^{q_i}\\ 
a_i f a_i^{-1}=f^{-1}
\end{array}
\right.\right\rangle$$

\normalsize
	where we used the notation $f$ for the class of the regular fibre, and $a_i, b_i, c_i, d_i$ in order to distinguish the classes $\bar{a}_i, \bar{b}_i, \bar{c}_i, \bar{d}_i$ in $\pi_1(S(g,h; (p_1, q_1), \dots, (p_k,q_k))$ from the corresponding projected classes in $\pi_1(\Sigma_{g,h})$.

Given any Seifert fibration $S(g,h; (p_1,q_1), \dots, (p_k,q_k))$, its \textit{Euler number} is $e(S)= \sum_{i=1}^k \frac{q_i}{p_i}$, where we take the sum $\mbox{mod } \mathbb{Z}$ when $S$ is not closed. 
 Two Seifert manifolds are \textit{isomorphic} if there exists a fiber-preserving homeomorphism between them. We recall two results on Seifert manifolds which classify them up to isomorphism and homeomorphism.

\begin{prop} [Corollary 10.3.13 in \cite{martelligeometric}]
	Let $S=S(g,h; (p_1, q_1), \dots, (p_k,q_k))$ and $S'=S(g',h'; (p'_1, q'_1 ), \dots (p'_{k'}, q'_{k'}))$ be two Seifert fibrations with $p_i,p'_i \ge 2$. Then $S$ and $S'$ are isomorphic if and only if $g=g', h=h',k=k'$ (up to reordering), $p_i=p'_i$ and $q_i=q'_i \mod p_i$, for all $1 \le i \le k$ and, if the manifolds are closed, $e(S)=e(S')$.
\end{prop}
Moreover, it turns out that every Seifert manifold admits a unique Seifert fibration, but for few exceptional cases. Namely, the following proposition holds (see \cite{jaco1980lectures}, Theorems VI.17 and VI.18 or Propositions 10.4.16 and 10.4.17 in \cite{martelligeometric}).

\begin{prop} [Classification of Seifert fibred manifolds up to homeomorphism] $\phantom{a}$
	\vspace{-5mm}
		\begin{enumerate}
	\item	Every Seifert fibred manifold with non-empty boundary admits a unique Seifert fibration up to isomorphism, except in the following cases:
	\begin{itemize}
		\item $D^2 \times S^1$, which admits the fibrations $S=S(0,1 ; (p,q))$ for any pair of coprime integers $(p,q)$.
		\item $K \widetilde{ \times} I$, the twisted, orientable interval bundle over the Klein bottle, which admits two non-isomorphic Seifert fibrations:\\
		$S_1=S(-1,1; -), S_2=S(1,1;(2,1), (2,1))$.
	\end{itemize}
\item	Every closed Seifert fibred manifold $S$ which is not covered by $S^3$ or $S^2 \times S^1$ admits a single Seifert fibration up to isomorphism, except for $K \widetilde{ \times}S^1$, which admits the non-isomorphic fibrations:
$S_3= S(0,0; (2,1), (2,1), (2,-1),(2,-1))$,
$S_4=S(-2,0; -)$.

\end{enumerate}
\end{prop}
\subsection{Seifert manifolds as $S^1$-bundles over $2$-orbifolds}
We can think about a Seifert fibred  manifold as a kind of surface bundle, whose fibres are the circles of the foliation. In this view, define $\mathcal{O}_S$ as the quotient space of $S$ obtained  by identifying every circle of the foliation to a point, and call $p: S \rightarrow \mathcal{O}_S$ the projection to the quotient. Let us focus on the local structure of this projection.
If $S$ is a trivial fibred solid torus, then $\mathcal{O}_S$ is clearly a $2$-disc and the projection $p: S \rightarrow \mathcal{O}_S$ is a bundle projection. If $S$ is a fibred solid torus $T(p,q)$, then $S$ is $p$-covered by a trivial fibred solid torus. The corresponding action of $\mathbb{Z}_p$ on $T=S^1 \times D^2$ is generated by a homeomorphism which is the product of a rotation through $2 \pi/p$ on the $S^1$-factor with a rotation through $2 \pi q/p$ on the $D^2$ factor. This action of $\mathbb{Z}_p$ on $S^1 \times D^2$ induces an action of $\mathbb{Z}_p$ on the base space $D^2$, generated by a rotation through $2 \pi q/p$.
Hence in this particular case the space $\mathcal{O}_S$ obtained by contracting each fibre to a point can be identified with the quotient of the disc $D^2$ by an action of $\mathbb{Z}_p$ by rotations, \textit{i.e.}, an orbifold having $D^2$ as underlying topological space and a conical singularity of angle $2 \pi/p$. We shall think about the projection map $p: S \rightarrow \mathcal{O}_S$ as a fibre bundle whit orbifold base, in a generalized sense. Actually, this is a bundle in the orbifold category, see  \cite{thurston1997three} and \cite{boileau2003three}, \S 2.4.
In the general case, the quotient space $\mathcal{O}_S$ of a Seifert fibred space $S$ obtained identifying each fibre to a point is (topologically) a surface, and is naturally equipped with an orbifold structure, in which cone points correspond to the projections of the critical fibres. Furthermore, if $S$ is a manifold with boundary then $\mathcal{O}_S$ is an orbifold with boundary, and $\partial S$ is the preimage of $\partial \mathcal{O}_S$ under the projection $p: S \rightarrow \mathcal{O}_S$.

Thus, Seifert fibred manifolds can be viewed as true $S^1$-bundles over $2$-dimensional orbifolds. In particular, orientable Seifert manifolds can be viewed as $S^1$-bundles over $2$-orbifolds with only conical singularities. We remark that in hte case where $\mathcal{O}_S$ is a surface then $S$ is a honest $S^1$-bundle on $\mathcal{O}_S$, and the Euler number of the Seifert fibration coincides with the Euler number of this $S^1$-bundle. Despite the fact that the definition via Dehn fillings provides a deeper topological insight, this approach through orbifolds is more informative about their geometry. Indeed, this perspective originally due to Thurston has been used, for example, in \cite{scott1983geometries}, \cite{bonahon2002geometric} and \cite{ohshika1987teichmuller} to describe the complete, locally homogeneous metrics admitted by  Seifert manifolds. 
We shall explain how to switch from the description by Dehn fillings to the description as fibration over an orbifold, which will be useful for what follows.
Let $S=S(g,h; (p_1,q_1), \dots, (p_k,q_k ) )$ be a Seifert fibered manifold. We can associate to $S$ its \textit{base orbifold} $\mathcal{O}_S$, which has $|\mathcal{O}_S|=\Sigma_{g,h}$ as underlying topological surface, and has $k$ conical points of order $p_1,\dots, p_k$.  In the case where the base orbifold is compact, it is possible to define its \textit{orbifold Euler characteristic}, generalizing the idea of Euler characteristic for surfaces.
$$ \chi^{orb}(\mathcal{O}_S)= \chi(|\mathcal{O}_S|)- \left(1- \sum_{i=1}^{k}\frac{1}{p_i}\right):$$
where $\chi(|\mathcal{O}_S|)$ denotes the usual Euler characteristic of the surface $|\mathcal{O}_S|$.
The orbifold characteristic encodes many information, both topological and geometric. First of all, since the Euler characteristic is multiplicative with respect to finite coverings, it allows to distinguish between orbifolds admitting a covering which is a manifold (\textit{good orbifolds}) and those that do not admit any manifold covering which is a manifold (\textit{bad orbifolds})\footnote{Here we are considering coverings in the orbifold sense, see \cite{thurston1997three}}. Bad orbifolds are completely characterized: see \cite{scott1983geometries} for a complete list. It turns out that Seifert manifolds fibering on bad orbifolds are \textit{lens spaces}, a well-understood class of quotients of the $3$-sphere. According to Thurston \cite{thurston1997three}, Theorem 13.3.6, good orbifolds carry geometric structures (\textit{i.e.}, they can be realized as quotients of $\mathbb{E}^2, \mathbb{S}^2$ or $\mathbb{H}^2$ by a discrete group of isometries acting possibly non freely), and the type of the supported structure can be detected by computing the orbifold characteristic. Namely, good orbifolds are distinguished in \textit{hyperbolic, Euclidean} and \textit{elliptic} depending whether $\chi^{orb}<0$,$\chi^{orb}=0$ or $\chi^{orb}>0$ (equivalently, if and only if they can be given a structure modelled on $\mathbb{H}^2$, $\mathbb{E}^2$ or $\mathbb{S}^2$ respectively).

In this paper we will be mainly concerned with orbifolds appearing as bases of orientable Seifert manifolds; moreover, we are mainly interested in Seifert fibred manifolds with boundary, which obviously fiber over $2$-orbifolds with boundary.
Two-dimensional orbifolds with boundary are either disks with a unique singular point (and in this case the Seifert fibration is a solid torus $D^2 \times S^1$), or are Euclidean or hyperbolic orbifolds.
Moreover, the only Euclidean orbifolds with boundary are either the M\"{o}bius band $Mb$ without singular points, or $D^2(2,2)$, the disk with two conical points of the same order $2$, or the cylinder.
This can be summarized in the following:

\begin{prop}
	The $2$-dimensional orbifolds with boundary appearing as bases of orientable Seifert manifolds are:
	\begin{enumerate}
		\item The disk $D^2(p)$ with a unique singular point of order $p$;
		\item The disk $D^2(2,2)$ with two conical points of order $2$, the M\"{o}bius band $Mb(-)$ without singular points and the cylinder, of Euclidean type;
		\item All orbifolds $\Sigma(g,h; p_1, \dots, p_k)$ of hyperbolic type.
	\end{enumerate}
Moreover, the only Seifert manifold which fibers over $D^2(p)$ is the solid torus, the only one fibering over $D^2(2,2)$ and $Mb(-)$ is $K \widetilde{ \times} I$ and over the cylinder is $T^2 \times I$.
Hence, all the Seifert manifolds with boundary except $K \widetilde {\times}I$, $T^2 \times I$ and $D^2 \times S^1$ fiber over hyperbolic orbifolds.
\end{prop}


\subsection{The short exact sequence} \label{section_short_exact_sequence}
The structure of Seifert manifolds as $S^1$-bundles over $2$-orbifolds gives the following exact sequence of groups (see \cite{scott1983geometries})
$$ 1 \rightarrow \langle f \rangle \overset{\psi}{\hookrightarrow}\pi_1(S) \overset{p_{\ast}}{\twoheadrightarrow} \piorb(\mathcal{O}_S) \rightarrow 1$$
where $\langle f \rangle$ is the subgroup generated by the regular fiber, the injective morphism $\psi$ is induced by the immersion, $\piorb(\mathcal{O}_S)$ is the \textit{orbifold fundamental group}, (\textit{i.e.}, the group of deck transformations of the orbifold universal cover, see Appendix \ref{Appendix_orbifolds}), and the surjective morphism $p_{\ast}$ is induced on orbifolds groups by the projection $p: S \rightarrow \mathcal{O}_S$.

We recall the presentation of the orbifold fundamental group of a $2$-dimensional orbifold, where the loops $\bar{a}_1, \bar{b}_1, \dots, \bar{a}_g, \bar{b}_g, \bar{d}_1, \dots, \bar{d}_k$   are the standard generators for the fundamental group of the underlying surface $|\mathcal{O}_S|$:

\footnotesize
$$\piorb(\mathcal{O})=\left\langle \bar{a}_1,\bar{b}_1,..,\bar{a}_g, \bar{b}_g, \bar{c}_1,..,\bar{c}_k, \bar{d}_1,..,\bar{d}_h\left |\!\begin{array}{l}
\prod_{i=1}^{g}[\bar{a}_i,\bar{b}_i]\cdot\prod_{j=1}^{k} \bar{c}_j\cdot\prod_{l=1}^{h} \bar{d}_l=1,
\bar{c}_i^{p_i}=1\\ 

\end{array}\!\!\!\!
\right.\right\rangle $$
\normalsize
if $g\ge0$,
\footnotesize
$$\piorb(\mathcal{O})=\left\langle \bar{a}_1,..,\bar{a}_{|g|}, \bar{c}_1,..,\bar{c}_k, \bar{d}_1,..,\bar{d}_h\left |\,\begin{array}{l}
\prod_{i=1}^{g}\bar{a}_i^2\cdot\prod_{j=1}^{k} \bar{c}_j\cdot\prod_{l=1}^{h} \bar{d}_l=1,
\bar{c}_i^{p_i}=1\\ 

\end{array}
\right.\right\rangle $$
\normalsize

\noindent if $g<0$.\\

Let us consider $p: S \rightarrow \mathcal{O}_S$, and let $i: \left(\Sigma_{g,h} \setminus \cup_{i=1}^k D_i \right) \times S^1 \hookrightarrow S$ be the inclusion. Then the elements $a_i= i_{\ast}(\bar{a}_i)$, $b_i= i_{\ast}(\bar{b}_i)$, $c_i= i_{\ast}(\bar{c}_i)$ and $d_i= i_{\ast}(\bar{d}_i)$ satisfy $p_{\ast}(a_i)=\bar{a}_i$, $p_{\ast}(b_i)=\bar{b}_i$, $p_{\ast}(c_i)=\bar{c}_i$ and $p_{\ast}(d_i)=\bar{d}_i$.
Hence, the short exact sequence gives us a more significant information concerning the relation between $\pi_1(S)$ and $\piorb(\mathcal O_S)$. Indeed it provides a correspondence between the canonical presentation of $\pi_1(S)$ and the canonical  presentation of $\piorb(\mathcal O_S)$.

\section{The geometry of irreducible $3$-manifolds} \label{section_JSJ_decomposition}

\subsection{The JSJ decomposition and the Hyperbolization Theorem}Every compact, orientable irreducible $3$-manifold admits a decomposition along suitable embedded tori. This decomposition is called \lq\lq \textit{JSJ decomposition}'' after the names of Jaco--Shalen and Johansson, who obtained it independently in the late 1970s (see \cite{jaco1978new}, \cite{johannson1979homotopy}).\\
\medskip

 Before stating the main results about this decomposition, we introduce some notation; for full details, we refer the reader to \cite{martelligeometric}.
Let $M$ be a compact, orientable $3$-manifold, and $S \subset M$ a \emph{properly embedded orientable surface}, \textit{i.e.} an orientable, embedded surface such that $\partial S \subset \partial M$. A disk $D \subset M$ is called a \textit{compressing disc for $S$} if $\partial D= D \cap S$ and $\partial D$ does not bound a disc in $S$. A properly embedded, connected, orientable surface $S \subset M$ with $\chi(S) \le 0$ is \textit{compressible} if it has a compressing disc, \textit{incompressible} otherwise, see Figure \ref{fig:compressiondisc}.
\begin{figure}[h!]
	\centering
	\includegraphics[width=0.7\linewidth]{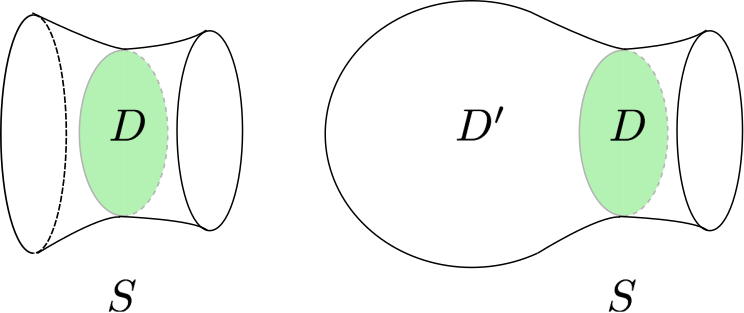}
	\caption[Incompressible surface]{If the surface $S \subset M$ is incompressible, then for every disc $D \subset M$ such that $\partial D \subset S$ (on the left), there exists a disc $D' \subset S$ such that $\partial D'= \partial D$ (on the right). }
	\label{fig:compressiondisc}
\end{figure}

\noindent A properly embedded surface $S \subset M$ is \textit{ $\partial$-parallel} if the inclusion $ i :S \hookrightarrow M$ is isotopic to a map into the boundary $i_0: S' \hookrightarrow \partial M$. 
A surface is \textit{essential} if it is incompressible and not $\partial$-parallel.\\

 In \cite{thurston82} Thurston introduced the following definition.
\begin{defn}
 An irreducible $3$-manifold is \textit{homotopically atoroidal} if every $\pi_1$-injective map from the torus to the irreducible manifold $X$ is homotopic to a map into the boundary.
 \end{defn}
\noindent It turns out that homotopically atoroidal manifolds are, together with Seifert manifolds, the \lq \lq \textit{building blocks}''  of the closed, orientable, irreducible $3$-manifolds. This is made precise in the following statement.

\begin{namedthm}{JSJ Decomposition Theorem}
	Let $Y$ be a closed, orientable, irreducible $3$-manifold. There exist a collection $\mathcal{C}=\{T_1, \dots , T_r\}$ of disjoint, essential tori in the interior of $Y$ enjoing the following properties:
	\begin{enumerate}[label=\roman{*}., ref=(\roman{*})]
		\item every connected component of $Y \setminus \mathcal{C}$ is either a Seifert manifold, or \textit{homotopically atoroidal};
		\item the collection $\mathcal{C}$ is minimal.
	\end{enumerate}
	Furthermore, the collection $\mathcal{C}$ is unique up to isotopy.
\end{namedthm}

Notice that statement \textit{i.} does not provide a dichotomy between Seifert and homotopically atoroidal components of the decomposition. However, if we focus on orientable Seifert fibred manifolds with boundary there are only three manifolds which are both homotopically atoroidal and Seifert:  $K \widetilde{ \times}I, T^2 \times I, D^2 \times S^1$. Moreover, the only one that can possibly appear as a $JSJ$ component of a closed, orientable, irreducible $3$-manifolds is $K \widetilde \times I$.\footnote{Notice that neither $T^2 \times I$ nor $D^2 \times S^1$ can appear as $JSJ$ component, (the first one by minimality of the collection, and the second one because any torus abutting it would not be incompressible).}\\

In \cite{thurston82}, Thurston stated his celebrated Geometrization Conjecture and announced some papers (see \cite{thurstonhyperbolic1}, \cite{thurstonhyperbolic2}, \cite{thurstonhyperbolic3}) proving his conjecture in restriction to Haken manifolds\footnote{A \textit{Haken manifold} is an irreducible $3$-manifold containing a properly embedded, incompressible, $2$-sided surface.}. One of the major steps in the proof is the so called Hyperbolization Theorem, which claims the existence of a complete Riemannian metric, locally isometric to $\mathbb H^3$ in the interior of non-Seifert fibered irreducible  $3$-manifolds.
We refer to \cite{otalhyperbolization} and \cite{otalhyperbolization1} for a complete proof of this theorem. Later on, thanks to the work of Perelman (see \cite{perelman2002entropy}, \cite{perelman2003finite} and \cite{perelman2003ricci}), it was possible to replace the hypothesis \lq\lq Haken manifold'' with \lq\lq with infinite fundamental group'' in the statement of the Hyperbolization Theorem. We state here the more recent version of the Thoerem:

\begin{namedthm}{Hyperbolization Theorem} 
	The interior of a compact, orientable, irreducible, $3$-manifold $X$ with infinite fundamental group admits a complete, hyperbolic metric if and only if the manifold is homotopically atoroidal and not homeomorphic to $K \widetilde{ \times}I$. Moreover, the hyperbolic metric has finite volume if and only if $\partial X$ is a union of tori and $X\not\simeq T^2\times I, D^2 \times S^1$.
\end{namedthm}

In the light of the Hyperbolization Theorem, in the following we will speak about a \textit{JSJ component of hyperbolic type} meaning a $JSJ$ component of a given closed, orientable irreducible $3$-manifold which is homotopically atoroidal and other than $K \widetilde{ \times} I$, and therefore admitting a complete, hyperbolic metric of finite volume.
Hence, given a closed, orientable, irreducible $3$-manifold $Y$, every $JSJ$ component $X$ of $Y$ satisfies one of these three, mutually exclusive conditions:
\begin{enumerate}
	\item $X$ is of hyperbolic type;
	\item $X$ is a Seifert manifold with hyperbolic base orbifold;
	\item $X$ is homeomorphic to $K \widetilde{\times}I$, and has Euclidean base orbifold.
\end{enumerate}
In the next section we shall describe geometric metrics on Seifert fibred JSJ components.

\subsection{Geometric metrics on the interior of  Seifert fibered manifolds}\label{section_geometric_metrics_on_orientable_Seifert_manifolds_with_boundary}$\phantom{aaa}$

\noindent Let $S=S(g,h; (p_1,q_1), \dots (p_k,q_k))$ be an orientable Seifert manifold such that $\chi^{orb}(\mathcal{O}_S) <0$ and having non-empty boundary.
It is possible to endow $int(S)$ with a Riemannian metric locally isometric to $\mathbb{H}^2 \times \mathbb{E}^1$ as follows.\\\\
\noindent Since $\chi^{orb}(\mathcal{O}_S) <0$, we know by Thurston's work that we can consider a discrete, faithful representation $$\bar \rho: \piorb(\mathcal{O}_S) \rightarrow PSL_2(\mathbb{R})$$ Moreover, we can assume (see \cite{liebeck2004fuchsian}) that $\bar \rho(\bar {d_i})$ is a parabolic isometry of $\mathbb{H}^2$. The choice of such a representation is equivalent to endow the orbifold $\mathcal{O}_S$ with a Riemannian metric of finite volume and modelled on $\mathbb{H}^2$ (see Appendix \ref{Appendix_orbifolds} for the definition).  In order to obtain a complete metric of finite volume locally isometric to $\mathbb{H}^2 \times \mathbb{E}^1$ of $\pi_1(S)$ into $\Isom^+ (\mathbb{H}^2 \times \mathbb{E}^1)$ it is sufficient to define a homomorphism $$\rho_v: \pi_1(S) \rightarrow \Isom(\mathbb{E}^1)$$ satisfying the following relations:

\noindent $$[g \ge 0]  \left \lbrace \begin{array} {c}
p_j \cdot \rho_v(c_j)= q_j \cdot\rho_v(f), \mbox{ for all } 1\le j \le k; \\ \sum_{j=1}^{k} \rho_v(c_j)+ \sum_{l=1}^{h} \rho_v(d_l)=0\\ \end{array} \right.$$
\noindent $$[g <0]  \left \lbrace \begin{array}{c}
p_j \cdot \rho_v(c_j)= q_j \cdot\rho_v(f), \mbox{ for all } 1 \le j \le k; \\ 
\sum_{i=1}^{|g|} \rho_v(a_i^2)+ \sum_{j=1}^{k} \rho_v(c_j)+ \sum_{l=1}^{h} \rho_v(d_l)=0.\\
\rho_v(a_i) \cdot \rho_v(f) \cdot \rho_v(a_i^{-1})= - \rho_v(f)  \mbox{, for all } 1 \le i \le |g|\end{array} \right.  $$

\noindent and put $$\rho= (\bar \rho \circ \varphi, \rho_v): \pi_1(S) \rightarrow \Isom^+(\mathbb{H}^2 \times \mathbb{E}^1)$$ Notice that the first set of conditions in the case $g \ge 0$ is satisfied when we put $\rho_v(c_j)= \frac{q_j}{p_j} \rho_v(f)$. The second relation implies that we can choose $\rho_v(d_l)$ for all the boundary loops but one, say $d_{h}$.
In the case where $g <0$, once again the first set of conditions is verified once we put $\rho_v(c_j)= \frac{q_j}{p_j} \rho_v(f)$. 
On the other hand, since $\rho_v(a_i)$ is a reflection with respect to a given point of $\mathbb{E}^1$ (because of the third set of relations), it follows that $\rho_v(a_i^2)= \id$. Hence, the second relation provides the same restrictions on $\rho_v(d_i)$ as in the case $g \ge0$.
We shall now verify injectivity of $\rho$, and the discreteness and freedom of the action of $\rho(\pi_1(S))$ on $\mathbb{H}^2 \times \mathbb{E}^1$.\\\\

\noindent \textit{Injectivity}:
Let us assume that $\gamma_1, \gamma_2 \in \pi_1(S)$ are such that $\rho(\gamma_1)= \rho(\gamma_2)$. Then by the short exact sequence, and by the fact that $\bar{ \rho}$ is faithful, $\gamma_1 \in \gamma_2  \langle f \rangle$, \textit{i.e.}, $\gamma_1= \gamma_2 f^k$. Furthermore, since $\rho_v(\gamma_1)= \rho_v(\gamma_2)$, we have that $k \rho_v(f)=0$, but $\rho_v(f) \not =0 $, thus we conclude that $k=0$, and $\gamma_1= \gamma_2$.\\\\
\noindent \textit{Discreteness}: Let us assume that $\rho(\gamma_n) .o$ converges to a point, say $p$. By discreteness of $\bar{\rho}$, we may assume $\rho(\gamma_n)$ for $n>>0$ only differ for the action on the second factor. This means that $\gamma_n= \gamma_1 \cdot f^{k_n}$, for every $n \in \mathbb{N}$, and the discreteness follows from the discreteness of the action of $\langle \rho_v( f ) \rangle$ --- which acts as a translation. \\\\
\noindent \textit{Freedom}: Assume $\rho(\gamma). o=o$; then, $(\bar{\rho}\circ \varphi (\gamma). \bar{o}, \rho_v(\gamma)(0))= (\bar{o},0)$. The condition on the first factor implies that $\bar{\rho}\circ \varphi (\gamma)$ is conjugate for a suitable $k$ to one the $c_i^k$ (which are sent in elliptic elements of $PSL_2(\R)$ ), while the condition on the second factor implies that $\rho_v(\gamma)= \Id$, which is a contradiction because $\rho_v(c_i^k)$ is a translation by $k\frac{p_i}{q_i}\rho_v(f)$, thus $k=0$.\\

\noindent Let now $S$ be the twisted interval bundle on the Klein bottle.
This is the only case where the base orbifold has not negative Euler orbifold characteristic, namely $\chi^{orb}(\mathcal{O}_S)=0$.
In this case, its interior can be endowed with a complete, Euclidean metric. Indeed, recalling that $K \widetilde{ \times} I$ can be fibred as $Mb\widetilde{\times} S^1$, we can take a representation $\bar{ \rho}:\pi_1(Mb) \simeq \mathbb{Z}\rightarrow \Isom(\mathbb{E}^2)$, by choosing $\bar{\rho}(\bar a)$ to be a glide-reflection. By the presentation recalled in \ref{section_seifert_fibered}, where $ \bar a$ is the generator of $\pi_1(Mb)$, we define the representation of $\pi_1(Mb\widetilde{ \times} S^1)$ as follows:

$$ \begin{array}{cccc}
\rho: & \pi_1(Mb \widetilde{ \times }S^1) & \longrightarrow & \Isom(\mathbb{E}^2) \times \Isom(\mathbb{E}) \\
& a & \longmapsto & (\bar{\rho}( \bar a), \sigma(a))\\
& f& \longmapsto & (\mbox{ Id }\;, \rho_v(f))
\end{array}$$

\noindent where we have defined $\sigma(a)$ as the reflection of $\mathbb{E}^1 \simeq \langle z \rangle $ with respect to the origin ---in particular, an orientation-reversing isometry--- so that the pair $(\bar{ \rho}( \bar a), \sigma(a)) \in \Isom^+(\R^3)$ and $\rho_v(f)$ is a translation.
Faithfulness, discreteness and freedom of the representation are readily seen.

\subsection{Non-positively curved metrics with totally geodesic boundary on Seifert fibred manifolds}\label{section_Eberlein}

Let $S$ be a  Seifert fibred manifold with  hyperbolic base orbifold $\mathcal O_S$. Let $\bar{g}_0$ be a non-positively curved metric with totally geodesic boundary on $\mathcal O_S$. Such a metric provides an injective homomorphism  $ \rho_{o}:\piorb(\mathcal O_S) \rightarrow \mathrm{Isom}(\widetilde{\mathcal O_S},\tilde{\bar g}_0)$, whose image acts discretely and cocompactly on $(\widetilde{\mathcal O_S},\tilde{\bar g}_0)$.
 As explained in Section \ref{section_geometric_metrics_on_orientable_Seifert_manifolds_with_boundary}, such a representation can be extended to an injective homomorphism  $\rho:\pi_1(S)\f \mathrm{Isom}^+((\widetilde{\mathcal O}_S, \tilde{\bar g}_0 )\times \mathbb E^1)$, whose image acts freely, discretely and cocompactly on $(\widetilde{\mathcal O_S},\tilde{\bar g}_0)\times\mathbb E^1$. 
 This can be done similarly to what we did to define the geometric metric on $int(S)$ in \S\ref{section_geometric_metrics_on_orientable_Seifert_manifolds_with_boundary}. Namely, we consider a homomorphism $\rho_v$ satisfying the conditions given in Section \ref{section_geometric_metrics_on_orientable_Seifert_manifolds_with_boundary} and check that the proof of the properties of the resulting $\rho=(\rho_o, \rho_v)$ work as well in this case. Moreover, notice that the non-positively curved Riemannian metric on $S$ induced by $\rho$ has totally geodesic boundary: indeed, by construction $\partial ((\widetilde{\mathcal{O}_S}, \tilde{\bar g}_0) \times \mathbb{E}^1)= \partial(\widetilde{\mathcal{O}_S}, \tilde{\bar g}_0) \times \mathbb{E}^1$.
 On the other hand, since $(\mathcal{O}_S, g_o)$ has totally geodesic boundary we know that $(\partial \widetilde{O}_S, \tilde{\bar g}_0) \times \mathbb{E}^1 $ is a totally geodesic submanifold, and $\rho_o(\pi_1(S))$ globally preserves $\partial (\widetilde{\mathcal{O}_S}, \tilde{\bar g}_0)$. Hence, $\rho (\pi_1(S))=(\rho_o, \rho_v)(\pi_1(S))$ globally preserves $(\partial(\widetilde{O}_S, \bar g_0 )) \times \mathbb{E}^1$; so  the $\partial(S, g_0)$ is a totally geodesic submanifold of $(S, g_0)$.
   The resulting non-positively curved metric with totally geodesic boundary $g_0$ on $S$ is such that the map $p:S\f \mathcal O_S$ obtained by collapsing the fibres is a Riemannian  orbifold submersion from $(S,g_0)$ to $(\mathcal O_S, \bar g_0)$. 
     It is worth to notice that, in the previous procedure is not necessary to ask that $(\mathcal O_S,  g_o)$ is non-positively curved.\\

Conversely let us consider any non-positively curved metric $g_0$ with totally geodesic boundary on a compact, orientable, Seifert fibred manifold S with hyperbolic base orbifold, {\it i.e.} an injective homomorphism $\rho:\pi_1(S)\rightarrow \mathrm{Isom}^+(\widetilde S, \widetilde{g_0})$, whose image acts cocompactly, discretely and freely.  
Let us consider the action $\rho(\pi_1(S))\curvearrowright(\widetilde S, \widetilde{g_0})$. By the explicit presentation of $\pi_1(S)$ provided in \S\ref{section_seifert_fibered}, it is readily seen that the infinite cyclic group $\langle f\rangle$ generated by the regular fibre is the unique maximal normal abelian subgroup of maximal rank. Hence, by the work of Eberlein (\cite{eberlein1983euclidean}) we know that $\langle f\rangle$ is the subgroup of Clifford translations in $\rho(\pi_1(S))$. Hence, the Riemannian universal covering $(\widetilde{S}, \widetilde{g_0})$ splits as a Riemannian product $(\tilde X, \tilde g_X)\times\mathbb E^1$, where the dimension of the Euclidean factor equals the maximal rank of the normal abelian subgroup $\langle g \rangle$, and $(\tilde X, \tilde g_X)$ is a simply connected, non-positively curved surface with totally geodesic boundary. Moreover being of Clifford translations, the elements  of $\langle f \rangle$ are precisely the elements of $\pi_1(S)$ acting on 
$(\tilde X, \tilde g_X)\times\mathbb E^1$ as the identity on the first factor and as translations on the second factor. 
Because of the splitting $(\widetilde{S},\widetilde{g_0})\cong(\tilde X, \tilde g_X)\times\mathbb E^1$  any isometry of $(\widetilde{S},\widetilde{g_0})$ can be written as a pair of isometries, see \cite{bridson2013metric}, Theorem 7.1. In particular, the homomorphism $\rho$ can be written as $(\rho_h, \rho_v)$ where $\rho_h:\pi_1(S)\rightarrow\mathrm{Isom}(\widetilde X, \tilde{g}_X)$ and $\rho_v:\pi_1(S)\rightarrow\mathrm{Isom}(\mathbb E^1)$.  We shall now verify that there exists an injective homomorphism $\bar\rho_h:\piorb(\mathcal O_S)\rightarrow\mathrm{Isom}(\widetilde X, \tilde g_X)$, whose image acts on $(\widetilde X,\tilde g_X)$ discretely and cocompactly with quotient $(X,g_X)$, and a map $q:(S,g_0)\f (X,g_X)$, which makes the following diagram commute\\

\begin{center}
\begin{tikzcd}
	\rho(\pi_1(S))\curvearrowright (\widetilde S, \tilde g_0) \arrow[r, "\tilde q"] \arrow[d]
	& (\widetilde{X}, \tilde{g}_X)\curvearrowleft {\bar\rho}_h(\piorb(\mathcal O))\arrow[d] \\
	(S,g_0) \arrow[r, "q", dashed]
	&( X, g_X) \end{tikzcd}
\end{center}

\noindent where $\tilde q$ is the projection on the first direct factor of $(\widetilde S, \tilde g_0)$.\\
First of all we remark that $\rho_h$ is constant on the fibres of the surjective homomorphism $\varphi: \pi_1(S) \rightarrow \piorb(\mathcal O_S)$ provided by the short exact sequence (see \S \ref{section_short_exact_sequence}), because $f$ and its powers act as the identity on the first factor. Thus, there exists a map $\bar \rho _h$ such that the following diagram commutes:
\begin{center}
	\begin{tikzcd}

	\pi_1(S) \arrow[r, "\rho_h"] \arrow[d, "\varphi"]
		& \mathrm{Isom}(\widetilde X, \tilde g_X)\\
		\piorb(\mathcal{O}_S )\arrow[ru, "\bar \rho_h", dashed] &
	\end{tikzcd}
\end{center}
Notice that $\bar \rho_h$ is a homomorphism, because both $\rho_h$ and $\varphi$ are.
In order to check that the action of $\bar \rho_h(\piorb(\mathcal{O}_S))$
 is discrete, we argue by contradiction. If this was not the case, there would exists a sequence of elements $\{\bar a_n\}_{n \in \mathbb N} \subset \piorb(\mathcal{O}_S)$ such that $\bar\rho_h(\bar a_n)(\bar x_0) \rightarrow \bar x_{\infty}$. For every $n$, let $a_n$ be an element such that $\varphi(a_n)=\bar a_n$. Let $x_0\in \tilde q^{-1}(\bar x_0)$; for every $n \in \mathbb N$, up to compose with a suitable power $k_n$ of $\rho(f)$ we may assume that the projection on the second factor of $\rho(a_n \cdot f^{k_n})(x_0)$ is contained in the interval $[- \rho_v(f), \rho_v(f)]$. Since $\tilde q(\rho(a_nf^{k_n}). x_0)=\bar\rho_h(\bar a_n).\bar x_0\f \bar x_\infty$, the sequence $\rho(a_n \cdot f^{k_n})(x_0)$ would contain a convergent sequence and this is a contradiction because the action of $\rho(\pi_1(S))$ on $(\widetilde{S}, \tilde g_0)$ is discrete. \newline
 Now we remark that by construction the map $\tilde q: (\widetilde S, \tilde g_0)\rightarrow (\widetilde X, \tilde g_X)$ is equivariant with respect to the surjective homomorphism $\bar \rho_h\circ\varphi\circ\rho^{-1}:\rho(\pi_1(S))\rightarrow \bar\rho_h(\piorb(\mathcal O_S))$. Hence there exists the required map $q:(S,g_0)\rightarrow(X,g_X)$.\newline
 Since $q$ is surjective and $(S,g_0)$ is compact it is readily seen that the action  of $\bar\rho_h(\mathcal O_S)$ on $(\widetilde X,\tilde g_X)$ is cocompact.\\\\
 
Hence, the action $\bar\rho_h(\piorb(\mathcal O_S))\curvearrowright (\widetilde X, \widetilde{g}_X)$ defines a non-positively curved metric with totally geodesic boundary $\bar g_0$ on the base orbifold $\mathcal O_S$. It is also clear that $\rho=(\bar\rho_h\circ\varphi, \rho_v)$. In particular the map which collapses the fibres of $S$ to points is a Riemannian orbifold submersion (see Appendix \ref{Appendix_orbifolds} for the definition) between $(S,g_0)$ and $(\mathcal O_S, \bar g_0)$.


\chapter[Bounding MinEnt from above: irreducible manifolds]{Bounding MinEnt from above:\\ irreducible $3$-manifolds} \label{section_a_conjectural_minimizing_sequence_irreducible}
In this chapter we shall prove the following theorem.

\begin{thm} \label{thm_construction_sequence_on_irreducible_Y}
	Let $Y$ be a connected, closed, orientable, irreducible $3$-manifold, denote by $X_1, \dots, X_k$ , $k \ge 1$, its hyperbolic $JSJ$ components and with $hyp_i$ the complete, finite-volume hyperbolic metric on $int(X_i)$ ---unique up to isometry. Then, there exists a $\bar\delta>0$ and a family of $C^2$ Riemannian metrics $\{g_{\delta}\}_{\delta\in(0,\bar\delta]}$ such that:
	\begin{enumerate}[label=\roman{*}., ref=(\roman{*})]
		\item $-1-\delta \le \sigma_{\delta} \le 0,$ where we denoted by $\sigma_\delta$ the sectional curvatures of the metrics $g_{\delta}$;
		\item $\Vol(Y, g_\delta)\overset{\delta\rightarrow0}\longrightarrow \sum_{i=1}^k\Vol(int(X_i), hyp_i)$.
	\end{enumerate}	
	As a consequence,
	\small
	$$ \minent(Y)\le 2  \left( \sum_{i=1}^{k} \Vol(int(X_i),hyp_i)\right)^{1/3}.$$
	\normalsize
\end{thm}

\begin{rmk}[About graph manifolds]
	Notice that we shall deal only with (closed orientable) irreducible $3$-manifolds having at least one $JSJ$ component of hyperbolic type. Indeed, the case of Seifert and graph manifolds has already been considered by Anderson and Paternain, building on the work by Cheeger and Gromov (see \cite{cheeger1985collapsing}).
	Actually, for graph manifolds the existence of a volume-collapsing sequence of metric with uniformly bounded sectional curvatures is sufficient to conclude that the Minimal Entropy of these manifolds is zero. 
\end{rmk}

Our construction relies on a celebrated work by Leeb, see \cite{leeb19953}. In that paper, the author addresses the question: which $3$-manifolds admit a non-positively curved metric? To answer this question Leeb first investigates when, on each fixed component $X$ of the $JSJ$ decomposition of $Y$, any prescribed flat metric on the boundary tori of $X$ can be extended to a global non-positively curved metric on $X$ with totally geodesic boundary.
The answer depends on the kind of $JSJ$ component we are dealing with: namely, the author proves that, while the hyperbolic components are \textit{flexible} ---\textit{i.e.}, any flat metric on the boundary can be extended to a global non-positively curved metric on $X$--- this is not true for the Seifert components (cf. Lemmata 2.3, 2.4, 2.5 in \cite{leeb19953}). However, it turns out that every \emph{graph manifold} \footnote{A \emph{graph manifold} is  a $3$-manifold such that every component in its $JSJ$ decomposition is a Seifert manifold.} with boundary admits a non-positively curved metric (cf. Theorem 3.2 in \cite{leeb19953}), and such a metric can be obtained by suitably gluing chosen metrics on each Seifert component.\\

In this chapter, we shall suitably modify Leeb's construction. In order to control the sectional curvatures and the convergence of the volumes, (two properties that Leeb does not take care of), particular attention will be paid in performing the gluings. We shall obtain this simultaneous control on the sectional curvatures and volumes at the price of loosing smoothness of the metrics $g_\delta$. Notice that, in order to compute the invariant $\minent(Y)$ we are allowed to consider more general Riemannian metrics ({\it pl-metrics} see \cite{babenko1993asymptotic}, Lemma 2.3); in particular, a $C^2$-family of Riemannian metrics will be fine for our purposes.
 It also worth noticing that the diameters of the metrics $g_\delta$ diverge as $\delta\rightarrow 0$.
 
 \section[The Extension Problem]{The Extension Problem} \label{section_the_extension_problem}
 
 In this section we shall briefly present the results proved by Leeb in his paper (\cite{leeb19953}), to get some acquaintance with the subject. In his work \cite{leeb19953} Leeb raises the following question:\\ \\
 \noindent \textit{Given an irreducible 3-manifold, can we endow it with a non-positively curved metric?}\\\\
Leeb's paper provides an affirmative answer to the question when the irreducible $3$-manifold considered either admits at least one hyperbolic $JSJ$ component or  is a graph-manifold with at least one boundary component. \footnote{
The problem of finding a complete criterion to detect those closed graph manifolds admitting a non-positively curved metric has been successively solved by Buyalo and Svetlov in \cite{buyalo2002homological}.}
In fact, he rather works on an equivalent problem, the so-called \lq \lq \textit{Extension Problem}''. Namely:\\
 
 \noindent \textit{Extension Problem: Given an irreducible $3$-manifold whose boundary components are tori, and a collection of flat metrics on the boundary tori, can we extend it to a non-positively curved metric on the whole manifold?}\\
 
 Following Leeb, we shall say that a collection of flat metrics $\{h_1,...,h_l\}$ defined on the connected components $T_1,...,T_l\subset \partial X$ of a given irreducible $3$-manifold $X$ with non-empty toroidal boundary is  {\it compatible} if the collection can be extended to a non-positively curved metric on $X$ with totally geodesic boundary. Hence, solving the \lq\lq Extension Problem" relative to a given collection of flat metrics on the boundary of an irreducible $3$-manifold is actually equivalent to understanding whether this collection  is compatible or not.
 In order to solve the {\it Extension Problem} he first investigates the problem in restriction to those manifolds which can appear as $JSJ$ components of an irreducible $3$-manifold with empty or toroidal boundary. He found out a rather different behaviour of the hyperbolic $JSJ$ components and of the Seifert fibred components. Namely, whereas the hyperbolic $JSJ$ components exhibit a strong {\it flexibility}, Seifert fibred $JSJ$ components show a certain {\it rigidity}, partly due to the structure of circle bundle. In detail, he proved the following result for hyperbolic $JSJ$ components:
 
 \begin{prop} [Proposition 2.3 \cite{leeb19953}]
 	Let $X$ be a hyperbolic $JSJ$ component of an irreducible $3$-manifold $Y$. Then any collection of flat metrics $\{h_1,.., h_l\}$ on the connected components of $\partial X$ is compatible.
\end{prop}

We shall state now Leeb's result concerning the {\it Extension Problem}  for Seifert fibred manifolds (\cite{leeb19953}, Lemma 2.5) using the notation introduced in Chapter \ref{chapter_basics}.
\begin{prop}
	Let $S=S(g,l;(p_1,q_1),...,(p_k,q_k))$ be a Seifert fibred manifold with non-empty boundary $\partial S= T_1 \cup \dots, T_l$. Let $\pi_1(T_i)=\langle d_i, f\rangle$. A collection of metrics $h_1$,.., $h_l$ on the boundary components of $S$ define scalar products $\sigma_i$ in restiction to $\pi_1(T_i)$. The collection $\{h_1,..,h_l\}$  is compatible if and only if the following conditions hold:
	\begin{equation}
	 \sigma_1(f,f)=\cdots=\sigma_l(f,f)=\| f\|^2
	\end{equation}
	\vspace{-5mm}
	\begin{equation}
	\sum_{i=1}^l \sigma_i(d_i, f)=-\left(\sum_{j=1}^k\frac{q_j}{p_j}\right)\cdot\| f\|^2
	\end{equation}
\end{prop}


Once established these results, Leeb proves the existence of non-positively curved metrics on irreducible $3$-manifolds with at least one hyperbolic $JSJ$ component, as well as the existence of non-positively curved metrics with totally geodesic boundary for graph-manifolds with non-empty boundary. Moreover, he gives examples of graph-manifolds which do not admit metrics of non-positive curvature. Here there are its affirmative answers to the existence of non-positively curved metrics:

\begin{thm}[\cite{leeb19953}, Theorem 3.2]\label{thm_leeb_graph}
	Let $Y$ be a graph-manifold with non-empty boundary. Then there exists a Riemannian metric with non-positive curvature and totally geodesic boundary on $Y$.
\end{thm}

\begin{thm}[\cite{leeb19953}, Theorem 3.3]\label{thm_leeb_ator}
	Let $Y$ be an irreducible $3$-manifold with (possibly empty) toroidal boundary. Suppose that $Y$ has at least one hyperbolic JSJ component, then $Y$ admits a Riemannian metric of non-positive curvature and totally geodesic boundary.
\end{thm}

\begin{rmk}
	Notice that what Leeb actually proves is that you can choose the non-positively curved metric on $Y$ to be flat on suitable collars of the $JSJ$ tori and of the boundary tori.
\end{rmk}

Let $Y$ be a closed, irreducible $3$-manifold with at least one hyperbolic $JSJ$ component, and let $X_1,...,X_n$ be its $JSJ$ components. Let $$I=\left\{\begin{array}{c}
      \phantom{a}\\
	(i,j,k,l)\\
	\phantom{a}
	\end{array}
	\right.\left|\begin{array}{c}
	1\le i\le j\le n,\\
	k\in\{1,..,l_i\}, l\in\{1,...,l_j \}\\
	\partial_k X_i, \partial_l X_j\mbox{ are identified in }Y
	\end{array}
	\right\}$$
 Let $F_{i,j}^{k,l}:\partial_k X_i\rightarrow\partial_lX_j$, $(i,j,k,l)\in I$ be the homeomorphisms between the boundary tori of the $X_i$'s which produce $Y$ as a result of the gluings. We shall say that a collection of flat metrics $\{\{h_{i,j}\}_{j=1}^{l_i}\}_{i=1}^n$ on $\{\partial X_i \}_{i=1}^n$ is \textit{$JSJ$ compatible} if the following hold:
\begin{itemize}
	\item[i.] for any $i=1,.., n$ the family $\{ h_{i,j}\}_{j=1}^{l_i}$ is a collection of compatible metrics on $X_i$;
	\item[ii.] whenever $\partial_k X_i$ and $\partial_l X_j$ are identified via $F_{i,j}^{k,l}$, then $F_{i,j}^{k,l}$ is isotopic to an isometry from $(\partial_k X_i, h_{i,k})$ to $(\partial_l X_j, h_{j,l})$.
\end{itemize}
Using this terminology, a restatement of Theorem \ref{thm_leeb_ator} is the following:

\begin{cor}
	Let $Y$ be an irreducible $3$-manifold with at least one hyperbolic $JSJ$ component. Then there exists a collection $\{\{ h_{i,j}\}_{j=1}^{l_i}\}_{i=1}^n$ of $JSJ$ compatible metrics on $\{\partial X_i\}_{i=1}^n$.
\end{cor}

\begin{rmk}
	Clearly, given $Y$ a closed, irreducible $3$-manifold with at least one hyperbolic $JSJ$ component, the set of $JSJ$ compatibe metrics on $\{\partial X_i\}_{i=1}^n$ is a cone: any uniform rescaling by a positive real number of a $JSJ$ compatible collection is again a $JSJ$ compatible collection.
 \end{rmk}

\section[Non-positively curved $ C^2$ metrics on Seifert manifolds]{Non-positively curved $C^2$ metrics on Seifert manifolds }

We observed in Chapter \ref{chapter_basics} that Seifert components of irreducible, orientable, closed $3$-manifolds having at leat a hyperbolic piece in their $JSJ$ decomposition are either of hyperbolic type or homeomorphic to  $K \tilde{\times }I$, as $D^2 \times S^1$ and $T^2 \times I$ cannot appear as $JSJ$ components of closed, irreducible $3$-manifolds. In \S\ref{section_Eberlein} we explained how a non-positively curved metric with totally geodesic boundary $g_0$ on a Seifert fibred manifold $S$ with  base orbifold $\mathcal O_S$ provides us with a non-positively curved metric with totally geodesic boundary $\bar g_0$ on the base orbifold. Moreover, the quotient map collapsing the fibres to points is a Riemannian orbifold submersion $p: (S, g_0) \rightarrow (\mathcal{O}_S, \bar{g_0})$.
This section is devoted to the proof of the following proposition:
\begin{prop}
		 \label{proposition_metrics:on:Seifert:components}
		 For every orientable Seifert manifold with boundary and hyperbolic base orbifold $S=S(g,l; (p_1, q_1), \dots, (p_k,q_k) )$, given any compatible collection of metrics $h_1, \dots, h_l$ on the boundary tori $T_1, \dots, T_l$ of $S$, there exist $\bar\delta_S , \bar\zeta_S(\delta)>0$ with $\bar\zeta_S(\delta)\rightarrow 0$ as $\delta\rightarrow 0$ and a $2$-parameter family of $C^2$ Riemannian metrics  $\left \lbrace k_{\delta, \zeta} \right \rbrace_{\delta \in(0, \bar \delta_S], \zeta\in(0, \bar \zeta_S(\delta)]}$, such that:
		 
		 \begin{itemize}
		 	\item[(i)]     $-1-\delta \le \sigma_{\delta,\zeta} \le 0$, where $\sigma_{\delta, \zeta}$ is the sectional curvature of the metric $k_{\delta, \zeta}$;
		 	\item[(ii)] $k_{\delta,\zeta}|_{T_i}= \zeta^2 h_i$;
		 	\item[(iii)] the metrics are flat in a collar of the boundary;
		 	\item[(iv)] $\Vol \left(S, k_{\delta,\zeta}\right) \overset{\delta \rightarrow 0}{\longrightarrow} 0$.
		 \end{itemize}
	 Any compatible metric $h_1$ on $\partial(K\widetilde\times I)$ is the restriction to the boundary of a flat metric $g$. Then the collection $\zeta^2g$ for $\zeta\in(0,1]$ is a continuous family of homothetic flat metrics verifying in particular $(i)$---$(iv)$.
\end{prop}

\begin{rmk}
	Few words about this result:
	\begin{itemize}
		\item [i.] with these requirements we are not able to extend any compatible collection of flat metrics on $\partial S$, but any compatible collection has a rescaling which can be extended to a $C^2$ Riemannian metric as above; 
		\item[ii.] in the case where $S$ is a Seifert fibred manifold with hyperbolic base orbifold the bounds on $\sigma_{\delta,\zeta}$ are \lq\lq $\delta$-almost optimal" in the sense that the upper bound coincides with the upper bound of a geometric metric of finite  volume on $int(S)$ and the lower bound is smaller than the lower bound of such a metric by an amount of $\delta$;
		\item[iii.] the almost optimal lower bound on the sectional curvature, turns out to be a strong constraint: indeed, it is possible to put smooth, non-positively curved Riemannian metric satisfying conditions  $(ii)$---$(iv)$ up to relaxing the condition $(i)$ to $\sigma_{\delta, \zeta}\ge -K^2$ for a suitable $K$, but we are not able to find a smooth family of metrics satisfying the bounds in $(i)$;
		\item[iv.] if we ask a $\delta$-almost optimal upper bound on the sectional curvature, instead of non-positivity, that is $-1-\delta\le\sigma_{\delta,\zeta}\le \delta$, we can improve $C^2$ to $C^\infty$, using a density argument.
	\end{itemize}
\end{rmk}



\begin{proof} First of all, we shall treat separately the case where $\mathcal O_S$ is a hyperbolic orbifold and  the case of $K\widetilde\times I$.\\

	 \noindent{\it Case $\chi_{orb}(\mathcal O_S)<0$.} Let $\{h_1,...,h_l\}$ be a collection of compatible metrics on $S$. By definition, there exists on $S$ a metric $g_0$ of non-positive curvature such that $(T_i, g_0|_{T_i})$ is isometric to $(T_i, h_i)$. As we have seen in Section \S\ref{section_Eberlein} this gives us a non-positively curved metric $\bar g_0$ on the base orbifold $\mathcal{O}_S$ and a Riemannian orbifold submersion $p:(S,g_0)\rightarrow(\mathcal O_S,\bar{g}_0)$. Equivalently, we have a pair of morphisms $\rho_h:\pi_1(S)\rightarrow\mathrm{Isom}(\widetilde{\mathcal O}_S,\tilde{\bar g}_0)$ and $\rho_v:\pi_1(S)\rightarrow\mathrm{Isom}(\mathbb E^1)$ such that the pair $(\rho_h,\rho_v)$ is an injective morphism of $\pi_1(S)$ into $\mathrm{Isom}^+((\widetilde{\mathcal O}_S,\tilde{\bar g}_0)\times \mathbb E^1)$, whose action on $(\widetilde{\mathcal O}_S,\tilde{\bar g}_0)\times\mathbb E^1$ is free and discrete. We recall that, by construction, we have :\\ 
	 
	 --- if $g\ge 0$, then $\rho_v(a_i)$, $\rho_v(b_i)$, $\rho_v(c_j)$, $\rho_v(d_r)$ are orientation preserving isometries of $\mathbb E^1$ ---thus translations of $\mathbb E^1$---, then Leeb's compatibility condition can be expressed in additive notation (with a slight abuse of notation, since we are identifying the translation thought as an isometry with the amount of the translation itself):
	 \small
	 \begin{align*}
	 &\sum_{r=1}^l\rho_v(d_r)+\sum_{j=1}^k\rho_v(c_j)=0,\\
	 &\frac{p_j}{q_j}\rho_v(c_j)=\rho_v(f)\mbox{ for }j=1,...,k;
	 \end{align*}
	 \normalsize
	 
	 --- if $g<0$ we know that the generators $a_1,...,a_{|g|}$ of $\pi_1(S)$ act on each of the two factors of the Riemannian product $(\widetilde{\mathcal O}_S,\tilde{\bar g}_0)\times\mathbb E^1$ as orientation reversing isometries, this means that $\rho_v(a_i)$, $i=1,..,|g|$ are inversions of $\mathbb E^1$. Consequently $\rho_v(a_i^2)=\rho_v(a_i)^2=\id_{\mathbb E^1}$ for any $i=1,..., |g|$. As  the $d_r$'s and the $c_j$'s act as orientation preserving isometries on both factors we conclude that, also in this case, we can use the additive notation to express the compatibility conditions:
	 \small
	 \begin{align*}
	 &\sum_{r=1}^l\rho_v(d_r)+\sum_{j=1}^k\rho_v(c_j)=0,\\
	 &\frac{p_j}{q_j}\rho_v(c_j)=\rho_v(f)\mbox{ for }j=1,...,k;
	 \end{align*}
	 \normalsize
	 Observe that, if $(\rho_h,\rho_v):\pi_1(S)\rightarrow\mathrm{Isom}^+((\widetilde{\mathcal O}_S,\tilde{\bar g}_0)\times\mathbb E^1)$ is an injective morphism whose image acts freely and discretely on $(\widetilde{\mathcal O}_S,\tilde{\bar g}_0)\times\mathbb E^1$, the same holds for $(\rho_h,\rho_v^{\varepsilon}):\pi_1(S)\rightarrow\mathrm{Isom}^+((\widetilde{\mathcal O}_S,\tilde{\bar g}_0)\times\mathbb E^1)$, where we denoted $\rho_v^\varepsilon$ the morphism which coincides with $\rho_v$ on the generators $a_i$'s, $b_i$'s if $g\ge0$ and $a_i$'s if $g<0$ and which is a rescaling by a factor $\varepsilon$ of the translations $\rho_v(d_r)$'s, $\rho_v(c_j)$'s and $\rho_v(f)$.\\
	 For any $r=1,...,l$ we shall denote by $\ell_r$ the translation length of $\rho_h(d_r)$, that is, the length of the corresponding boundary component of $(\mathcal O_S,\bar g_0)$. Notice that the translation length of $(\rho_h(d_r), \rho_v(d_r))$ is equal to $\sqrt{\ell_r^2+(\rho_v(d_r))^2}$ and it is attained for all the points of one boundary component of $(\widetilde S,\tilde g_0)$ which is one of the lifts of the boundary torus $T_r$.\\

Now we shall modify a geometric metric $g$ on $int(S)$ locally isometric to $\mathbb H^2\times\mathbb E^1$ in order to obtain the desired collection of $\mathcal{C}^2$ metrics $\lbrace k_{\delta, \zeta} \rbrace$ whose curvatures are bounded by $-1-\delta \le \sigma \le 0$, with $\delta \longrightarrow 0$, flat near the boundary, such that the restriction $k_{\delta, \zeta}|_{T_i}= \zeta^2 h_i$, where $\zeta<\bar\zeta_S(\delta)$ with $\bar\zeta_S(\delta)\rightarrow0$ for $\delta\rightarrow 0$, and whose volumes collapse as $\delta\rightarrow 0$. Once chosen $g$ a complete geometric metric of finite volume on $int(S)$, locally isometric to $\mathbb H^2\times\mathbb E^1$, call $\bar g$ the corresponding complete hyperbolic metric on $int(\mathcal O_S)$, and $\rho_g:\pi_1(S)\rightarrow\mathrm{Isom}^+(\mathbb H^2\times\mathbb E^1)$ the injective morphism such that $\mathbb H^2\times\mathbb E^1/\rho_g(\pi_1(S))$ is isometric to $(int(S),g)$. \newline
	Let us denote by $\mathcal{C}_1, \dots, \mathcal{C}_{l}$ the connected components of a  horoball neighbourhood of the cusps in $\mathcal{O}_S$; observe that by definition it does not contain any conical point (see Appendix \ref{Appendix_orbifolds}). Fix the parametrization in horospherical coordinates; then $\mathcal{C}_r \simeq S^1\times [0, + \infty) $.  Let $m_r $ be the length of the section $\{0\} \times S^1$ with respect to the hyperbolic metric.
	Then the hyperbolic metric on each cusp $\mathcal{C}_r$ can be locally written as a warped product:
	$$\bar{ g}= e^{-2t} m_r^2 \,dx^2 \oplus dt^2,$$
	where $dx^2$ is an Euclidean metric on the horosphere.\\

Now, let $T_{\delta}=\frac{1}{\delta}$. We shall modify the hyperbolic metric on the cusps starting from time $T_\delta$, thus producing the required sequence. Before that, let us state the following proposition whose proof is given in Appendix \ref{appendix_proposition}.

\begin{prop}\label{prop_interpolation}
	For every fixed $\ell, \delta >0$ there exist $\varepsilon, \varepsilon' >0$ arbitrarily small and $\varphi_{\varepsilon,\varepsilon'} \in C^2(\mathbb{R})$ such that:
	\begin{enumerate}
		\item $\varphi_{\varepsilon, \varepsilon'}$ is not increasing and convex, such that
		
		$$ \left \lbrace \begin{array}{l}
		\varphi_{\varepsilon, \varepsilon'}(t)= \varphi_0(t)= \ell e^{-t}\mbox{ for }t\le -\varepsilon,\\
		\varphi_{\varepsilon, \varepsilon'}(t)= \ell'\mbox{ for }t\ge t_\delta+\varepsilon',
		\end{array} \right.$$
		where
		\small
		$$ \ell \frac{\sqrt{\delta (1 + \delta)}}{1+ 2 \delta} \le \ell'(\delta, \varepsilon, \varepsilon') \le 4\ell \frac{\sqrt{\delta(1+ \delta)}}{1+2\delta}$$
		\normalsize
		\item $\left(\dfrac{\varphi_{\varepsilon, \varepsilon'}'}{\varphi_{\varepsilon, \varepsilon'}}\right)^2 < (1+ 2 \delta)^2\left(\dfrac{\varphi_{0 \phantom{,'}}'}{\varphi_{0 \phantom{,'}}}\right)^2$ for every $t \in \mathbb{R}$;
		\item $\dfrac{\varphi_{\varepsilon, \varepsilon'}''}{\varphi_{\varepsilon, \varepsilon'}} \le (1+2\delta)^2$ for every $t \in (-\varepsilon, t_{\delta}+ \varepsilon')$.
	\end{enumerate}
	where $t_{\delta}= \frac{1}{2(1+2 \delta)} \ln \left(1+ \frac{1}{\delta}\right)$.
\end{prop}

Now, let $\varphi_{\delta, r}$ be the $C^2$ function on $[0,T_\delta+t_\delta+2\varepsilon']$ which coincides with $e^{-t}m_r$ for $t\in[0, T_\delta-\varepsilon]$ and with $\varphi_{\varepsilon, \varepsilon'}$ given by Proposition \ref{prop_interpolation} for $\ell=e^{-T_\delta}m_r$ for $t\in[T_\delta-\varepsilon, T_\delta+t_\delta+2\varepsilon']$. Then, substitute the warping function $m_re^{-t}$ on $S^1\times[0, T_\delta+t_\delta+2\varepsilon']$ with the new function $\varphi_{\delta, r}$, thus obtaining the metric
$$ \bar g_{\delta}= \varphi_{\delta, r}^2 dx^2 \oplus dt^2.$$ For any fixed $\delta$ this gives us a $C^2$ Riemannian metric $\bar g_{\delta}$ satisfying the required bounds on the sectional curvature, which is hyperbolic on the complement of the horoball neighbourhoods $\mathcal C_r(T_\delta)\simeq S^1\times[T_\delta,+\infty)$ and flat near the boundary. Nevertheless, condition $ii$ is not  verified, hence, we shall adjust the metric $g_\delta$ along the cusps, and we shall give an explicit  value of $\varepsilon(\delta)$ of point $ii$. Let us remark that $\ell'$ of Proposition \ref{prop_interpolation}  continuously depends on $(\delta, \varepsilon, \varepsilon')$, and goes to $0$ as $(\delta, \varepsilon, \varepsilon')\rightarrow(0,0,0)$. In particular, for every $$\zeta\le\bar{\zeta}(\delta, S, \{h_i\})=\frac{1}{2}\cdot\min_{i=1,..,l}\left\{\frac{\ell_r'(\delta,\varepsilon,\varepsilon')}{\ell_r} \right\}$$ we can find a suitable $3$-tuple $(\delta_r(\zeta), \varepsilon_r(\zeta), \varepsilon_r'(\zeta))$ so that for every cusp we have $\ell_r'=\zeta\cdot\ell_r$. Let us call $\bar k_{\delta, \zeta}$ the previous metric and let us denote by $\bar \rho_{\delta,\zeta}:\pi_1(\mathcal O_S)\rightarrow \mathrm{Isom}(\widetilde O_S,\tilde{\bar k}_{\delta, \zeta})$ the corresponding injective homomorphism. It is worth to notice that, integrating along the cusps, we get
\small
$$\Vol(\mathcal O_S, \bar k_{\delta, \zeta})\le \Vol(int(\mathcal O_S), \bar g)+\sum_{r=1}^l\int_{T_{\delta_r}-\varepsilon_r}^{T_{\delta_r}+t_{\delta_r}+2\varepsilon_r'}\varphi_{\delta_r, r}dt\le$$
$$\le\Vol(int(\mathcal O_S), \bar g)+\sum_{r=1}^l e^{-(\frac{1}{\delta_r}-1)}\left(\frac{\ln(1+\frac{1}{\delta_r})}{2(1+2\delta_r)}+1\right)\longrightarrow \Vol(int(\mathcal O_S), \bar g)$$
\normalsize
when $\delta_1,\dots,\delta_l\rightarrow 0$.\\
 To conclude, consider $(\bar{ \rho}_{\delta,\zeta}\circ p_*, \rho_v^\zeta):\pi_1(S)\rightarrow\mathrm{Isom}^+((\widetilde O_S,\tilde{\bar k}_{\delta, \zeta})\times\mathbb E^1)$ and observe that the induced Riemannian metric $k_{\delta, \zeta}$ clearly verifies conditions $i$---$iii$, while condition $iv$ follows from:
 \small
 $$\Vol(S,k_{\delta,\zeta})=\Vol(\mathcal O_S,\bar k_{\delta,\zeta})\cdot \zeta\cdot|\rho_v(f)|\overset{\delta\rightarrow0}{\longrightarrow} 0$$
 \normalsize
\vspace{2mm}

	\noindent {\it Case $S=K\widetilde\times I$.} Now suppose that $\chi_{orb}(\mathcal{O}_S)=0$. Then the Seifert manifold $S$ is $K \widetilde{ \times}I$, and can be viewed as $Mb \widetilde\times S^1$, having the M\"{o}bius band as base orbifold (surface) and no singular points. We write down presentations for the (orbifold) fundamental groups:
	
	$$\piorb(\mathcal{O}_S)= \pi_1(Mb)= \left \langle \bar a,\bar d \left \lvert  \right.\bar a^2=\bar d \right \rangle \simeq \mathbb{Z}$$
	$$ \pi_1(K \widetilde{ \times}I)= \left \langle a, d,f \left \lvert \begin{array}{l}
	{a}^2={d}\\
	{a}f{a}^{-1}= f^{-1}
	\end{array} \right. \right\rangle $$
	
\noindent Observe that a flat metric on the boundary torus of $K\widetilde{ \times}I$ is  compatible if and only if  is the restriction to the boundary of a flat metric $g$  on $K\widetilde\times I$. Indeed, a non-flat, non positively curved metric $g$ on $K\widetilde\times I$ flat on a collar of the boundary would give a non-positively curved metric on $(D(K\widetilde{ \times }I), g')$, the Riemannian manifold obtained by doubling $(K\widetilde{ \times } I, g)$ along the boundary torus. But $D(K\widetilde{ \times }I)$ is finitely covered by a $3$-dimensional torus $T^3$, and this would give a contradiction since $\hat g'$, the lift of the metric $g'$ would be a non-positively curved, non-flat metric on $T^3$. The flat metric $g$ on $K\widetilde{ \times }I$ is obtained as the quotient of $([-a,a]\times\mathbb E^1)\times\mathbb E^1$ (for a suitable choice of $a$) by the isometry subgroup given by the image of the following injective homomorphism:

	

	$$ \begin{array}{cccc}
	\rho=(\rho_h,\rho_v): & \pi_1(Mb \times S^1) & \longrightarrow & \Isom^+(([-a,a]\times\mathbb E^1)\times\mathbb E^1) \\
	& {a} & \longmapsto & (\rho_h(a), \rho_v({a}))\\
	& f& \longmapsto & (\mbox{ Id }\;, \rho_v(f))
	\end{array}$$
	where $\rho_h( a)$ is the glide-reflection $(x, y)\mapsto(-x, k y)$, for $k\in\mathbb R$ and $\rho_v(a)$ is a the symmetry with respect to a point in $\mathbb E^1$. As usual $\rho(f)$ acts as the identity on the first component $([-a,a]\times \mathbb E^1)$ and as a translation $\rho_v(f)$ on $\mathbb E^1$. Observe that the isometries $\rho(a^2)$ and $\rho(f)$ act fixing both boundary components of $([-a,a]\times\mathbb E^1)\times\mathbb E^1$ and that $\rho(a^2)=(\rho_h(a^2), \Id_{\mathbb E^1})$ with $\rho_h(a^2)(x,y)=(x,2k\,y)$. In other words, given the previous presentation of $\pi_1(K\widetilde{ \times}I)$ the elements $a^2$ and $f$ act as translations along orthogonal directions on the universal covering: hence a compatible metric $h_1$ on $\partial(K\widetilde{ \times }I)$	is necessarily a flat metric generated by an orthogonal lattice (given by restricting $\langle\rho(a^2), \rho(f)\rangle$ to an isometry of $\mathbb E^1\times\mathbb E^1$). Rescalings of $g$ provides $K\widetilde{ \times }I$ with the required family of Riemannian metrics.
\end{proof}

\section[Metrics of non-positive curvature on atoroidal $3$-manifolds]{Metrics of non-positive curvature\\ on atoroidal $3$-manifolds} \label{section_metrics_hyperbolic_components}
The main result of this section is the following:

\begin{thm} \label{thm_metrics_hyperbolic_pieces}
	Let $X$ be a hyperbolic component and let $\{h_r\}_{r=1}^l$ be a flat metric on the boundary of $X$. There exists a sequence of metrics $hyp_\delta^0$ on $X$ such that:
	\begin{enumerate}
		\item For every $\delta$ the sectional curvature $\sigma_{\delta}^0$ is pinched:
		$$ -(1+2\delta)^2\le \sigma_{\delta}^0 \le 0$$
		\item $\Vol(X, hyp_{\delta}^0) \overset{\delta \rightarrow 0}{\longrightarrow} \Vol(int(X), hyp)$, where $hyp$ it the unique ---up to isometry--- complete, finite volume hyperbolic metric on $int(X)$.
		\item There exists a family $\{\mathcal U_\delta\}_{\delta\in(0,1]}$ of nested horoball neighbourhoods of the cusps of $(int(X),hyp)$, whose complements for $\delta\rightarrow 0$ progressively exhaust $(int(X),hyp)$, such that $(int(X)\smallsetminus\mathcal U_\delta, hyp)$ is isometric to an open set $V_\delta$ in $(X,hyp_\delta^0)$ and such that $\frac{\Vol (X,hyp_\delta^0)}{\Vol(V_\delta, hyp_\delta^0)}\rightarrow 1$.
		\item  The metric $hyp_{\delta}^0$ is flat in a collar of the boundary and there exists $\ell'=\ell'(X,\delta,\{h_i\})$ such that $(T_i,hyp_\delta^0|_{T_i})$ is isometric to $(T_i,(\ell')^2h_i)$.
	
	\end{enumerate}
\end{thm}

\noindent In order to achieve this result we will prove a the following lemma, which is a refinement of part of the argument given by Leeb to prove \cite{leeb19953}, Proposition 2.3.

\begin{lem}\label{lemma_change_of_conformal_type}
	Let $(X', hyp)$ be an open $3$-manifold endowed with a complete Riemannian metric of finite volume, locally isometric to $\mathbb H^3$. Let $$\mathcal U=\bigcup_{i=1}^r\mathcal C_i=\bigcup_{i=1}^r T_i\times[0,\infty)$$ be a horoball neighbourhood of the cusps of $(X', hyp)$. Let $h_1,\dots, h_r$ be a collection of flat metrics on $T_1,\dots, T_r$. There exists a $1$-parameter family of complete Riemannian metrics $\{hyp_\delta\}_{\delta\in(0,1]}$ with the following properties:
	\begin{enumerate}
		\item $-1-\delta\le\sigma_{\delta}\le-1$;
		\item $hyp_{\delta}=hyp$ on $X'\smallsetminus \mathcal U_\delta$, where $\mathcal U_\delta=\bigcup_{i=1}^r T_i\times{[T_\delta,+\infty)}\subset\mathcal U$, for $T_\delta=\frac{C}{\delta}$ and where the constant $C$ is determined by $(X', hyp)$, $\{h_i \}$ and $\mathcal U$;
		\item $\sigma_{\delta}=-1$ in $(\mathcal U_{\delta}', hyp_\delta)$ where $\mathcal U_\delta'=\bigcup_{i=1}^r T_i\times{[2T_\delta,+\infty)}$;
		\item for $t\ge 2T_\delta$ the Riemannian manifold $(T_i\times\{t\}, hyp_\delta|_{T_i\times\{t\}})$ is isometric to $(T_i, \zeta^2e^{-2t}h_i)$ where $\zeta$ is determined by $(X',hyp)$, by the  collection $\{h_i\}$ and the choice of $\mathcal U$;
		\item $\Vol(X',hyp_\delta)\overset{\delta\rightarrow 0}{\longrightarrow}\Vol(X', hyp)$.
	\end{enumerate}
\end{lem}

\begin{proof}
	Let us start with a complete hyperbolic metric of finite volume $hyp$ on $X'$. Let $\mathcal U$ be the horoball neighbourhood of the cusps of $(X',hyp)$ defined in the statement. We know that $hyp$ is given in horospherical coordinates on $\mathcal C_i$ by:
	\small
	$$\mathcal C_i=T_i\times[0,+\infty), \quad hyp_{(x,t)}=e^{-2t}(k_i)_x\oplus dt^2$$
	\normalsize
	where $k_i$ is determined up to rescaling solely by the metric $hyp$ on $X'$, which by Mostow's rigidity theorem is unique, and is thus independent of the choice of the horoball neighbourhood of the cusps $\mathcal U$.\\ Now we shall change the warping function describing the hyperbolic metric on $\mathcal U$ following Leeb's argument. Let us consider the flat metrics $k_i$, $h_i$ on $T_i$; by the Spectral Theorem it is possible to find a coordinate system $(x_i,y_i)$ for $T_i$ such that  $k_i=dx_i^2\oplus dy_i^2$ and $h_i=a_i\,dx_i^2+b_i\,dy_i^2$ where $a_i, b_i>0$ are two positive real constants. Up to rescaling the metrics $h_i$ by the constant $\zeta^2=\frac{1}{2}\max_{i=1,..,r}\{\frac{1}{a_i}, \frac{1}{b_i} \}$ we have $0<\zeta^2 a_i,\, \zeta^2b_i<1$ for every $i=1,...,r$. \\
	
\noindent Let $\phi: \mathbb{R} \rightarrow \left[ 0, \infty\right] $ be the following function 
\small
$$\phi(t)=\left\{\begin{array}{cl}
  1&\mbox{ for }t\le 0\\
  \dfrac{e^{-\frac{1}{1-t}}}{e^{-\frac{1}{1-t}}+e^{-\frac{1}{t}}} &\mbox{ for }t\in(0,1)\\ 
  0 &\mbox{ for } t\ge 1
\end{array}\right.
$$
\normalsize
Now we define the following functions, which will be the new warping functions for the metric along the cusps:
\small
$$\eta_{a_i}^{\delta}(t)= e^{-t} \left(\phi\left(\dfrac{t-T_\delta}{T_\delta}\right)+\left(1- \phi\left(\dfrac{t-T_\delta}{T_\delta}\right) \right) \zeta^2a_i  \right),$$
$$\eta_{b_i}^{\delta}(t)= e^{-t} \left(\phi\left(\dfrac{t-T_\delta}{T_\delta}\right)+\left(1- \phi\left(\dfrac{t-T_\delta}{T_\delta}\right) \right) \zeta^2b_i  \right)$$
\normalsize
where $T_\delta$ will be suitably chosen later on. Then consider the following metrics on the $\mathcal C_i$'s:
\small
$$ g_i^{\delta}(x_i,y_i,t)= (\eta_{a_i}^{\delta}(t))^2dx_i^2 \oplus (\eta_{b_i}^{\delta}(t))^2 dy_i^2 \oplus dt^2$$
\normalsize
Notice that the metrics $g_i^{\delta}$ coincide with $hyp$ on $T_i\times[0, T_\delta]$ and are equal to $e^{-2t}\zeta^2h_i\oplus dt^2$ on $T_i\times[2T_\delta,+\infty)$. Now we shall compute the sectional curvature of the metric $hyp_{\delta}$ obtained by gluing in the obvious way the metrics $g_i^{\delta}$ along the cusps. By construction the sectional curvature is constantly equal to $-1$  except on the portions  $T_i\times[T_\delta, 2T_\delta]$ of the cusps. Hence we will need to compute the sectional curvature of the metrics $g_{i}^{\delta}$ only in restriction to $T_i\times[T_\delta, 2T_\delta]$. An explicit computation of the Christoffel's symbols for $g_i^{\delta}$ gives:
\small
\begin{eqnarray}
\nonumber \Gamma^t_{x_ix_i}&=& \frac{1}{2} \left( -\frac{\partial}{\partial t} (\eta_{a_i}^{\delta})^2\right)\\
\nonumber \Gamma^t_{y_iy_i}&=& \frac{1}{2} \left(  - \frac{\partial}{\partial t} (\eta_{b_i}^{\delta})^2\right) \\
\nonumber \Gamma^{x_i}_{tx_i}= \Gamma^{x_i}_{x_it} &= &\frac{1}{2} \frac{1}{(\eta_{a_i}^{\delta})^2} \left( \frac{\partial}{\partial t} (\eta_{a_i}^{\delta})^2\right) \\
\nonumber \Gamma^{y_i}_{ty_i}= \Gamma^{y_i}_{y_it}&= &\frac{1}{2} \frac{1}{(\eta_{b_i}^{\delta})^2} \left( \frac{\partial}{\partial t} (\eta_{b_i}^{\delta})^2\right) 
\end{eqnarray}
\normalsize
and zero for all the remaining Christoffel symbols. 
Thus the sectional curvatures of $g_i^{\delta}$ are given by the following formulas:
\small
$$
 \sigma_{x_iy_i}^{\delta}= - \frac{({\eta^{\delta}_{a_i})}' \cdot {(\eta_{b_i}^{\delta})}'}{{(\eta_{a_i}^{\delta})} \cdot {(\eta_{b_i}^{\delta})}},\quad
\sigma_{x_it}^{\delta}= - \frac{({\eta_{a_i}^{\delta}})''}{({\eta_{a_i}^{\delta}})},\quad
	\sigma_{y_it}^{\delta}= - \frac{({\eta_{b_i}^{\delta}})''}{{(\eta_{b_i}^{\delta})}}
$$
\normalsize

Let us compute the derivatives of $\eta_{a_i}^{\delta}$ up to the second order (those of $\eta_{b_i}^{\delta}$ are  analogous):

\small
$$
	({\eta_{a_i}^{\delta}})'(t)= e^{-t} \left( \left(1-\zeta^2a_i\right) \left(\frac{\phi' \left( \frac{t-T_\delta}{T_\delta}\right)}{T_\delta} -  \phi\left(\frac{t-T_\delta}{T_\delta}\right)\right)-\zeta^2a_i\right)$$
$$
 ({\eta_{a_i}^{\delta}})''(t)= e^{-t} \left(
(1-\zeta^2a_i) 
\left( \frac{\phi'' \left( \frac{t-T_\delta}{T_\delta}\right)}{T_\delta^2}- \frac{2 \phi' \left( \frac{t-T_\delta}{T_\delta}\right)}{T_\delta}+\phi \left(\frac{t-T_\delta}{T_\delta}\right) \right) +\zeta^2a_i
\right)
$$
\normalsize
Notice that $(\eta_{a_i}^{\delta})'$, $(\eta_{b_i}^{\delta})'\le 0$. We shall now give conditions on $T_\delta$ which ensure that $\sigma_{x_i t}^{\delta}, \sigma_{y_it}^{\delta}, \sigma_{x_iy_i}^{\delta}$ are between $-1-\delta$ and $-1$.\\

\noindent Let us start with the curvature $\sigma_{x_iy_i}^{\delta}$. We want to give conditions on $T_\delta$ so that:
$$-1-\delta\le \sigma_{x_iy_i}^{\delta}\le-1$$

\noindent The second inequality is equivalent to:

$$ \label{K_xy_ge}
\frac{\left[ \left( 1-\zeta^2a_i\right) \left(\phi  - \frac{\phi'}{T_{\delta}} \right)+\zeta^2a_i\right] \left[ \left( 1-\zeta^2b_i\right) \left(\phi - \frac{\phi'}{T_{\delta}}\right)+\zeta^2b_i\right] }
{\left[  \left( 1-\zeta^2a_i\right) \phi +\zeta^2a_i \right] \left[  \left( 1-\zeta^2b_i\right) \phi +\zeta^2b_i \right] } \ge 1
$$

\noindent where $\phi$, $\phi'$ are evaluated at $(t-T_\delta)/T_\delta$. We notice that, since $\phi$ is non-increasing, the previous ratio is always greater then or equal to $1$, independently on the choice of $T_\delta$.
Thus inequality $ \sigma_{x_iy_i}^{\delta} \le -1$ actually holds. Now we shall deal with  the first inequality $\sigma_{x_iy_i}^{\delta}\ge -1-\delta$. 
Remark that:
$$\left( 1-\zeta^2a_i\right) \left(\phi  - \frac{\phi'}{T_{\delta}} \right)+\zeta^2a_i\overset{T_\delta\rightarrow+\infty}{\longrightarrow}\left( 1-\zeta^2a_i\right) \phi +\zeta^2a_i $$
$$\left( 1-\zeta^2b_i\right) \left(\phi - \frac{\phi'}{T_{\delta}}\right)+\zeta^2b_i\overset{T_\delta\rightarrow+\infty}{\longrightarrow} \left( 1-\zeta^2b_i\right) \phi +\zeta^2b_i $$
Moreover, an easy calculation shows that condition $$T_\delta\ge\frac{4}{\delta}\cdot\max_{i=1,..,r}\left\{\frac{(1-\zeta^2a_i)}{\zeta^2a_i}, \frac{(1-\zeta^2b_i)}{\zeta^2b_i}\right\}$$
ensures that $\sigma_{x_iy_i}^{\delta}\ge-1-\delta$ for every $i=1,...,r$.\\

\noindent Now we shall prescribe the bounds on $\sigma_{x_it}^{\delta}$, to derive the conditions on $T_\delta$. We ask:
$$-1-\delta\le\sigma_{x_it}^{\delta}\le -1$$
	The first inequality is equivalent to:
$$ \frac{\left(1-\zeta^2a_i\right) \left[\frac{1}{T_\delta^2} \phi'' - \frac{2}{T_\delta} \phi'  + \phi  \right]+\zeta^2a_i}
{\left( 1-\zeta^2a_i\right) \phi+\zeta^2a_i} \le 1 + \delta$$
where $\phi, \phi', \phi''$ are computed in $(t-T_\delta)/T_\delta$. 
Denote by $M''= \max_{t \in \left[0,1 \right] } \phi''$; then a straightforward computation shows that  inequality $-1-\delta\le\sigma_{x_it}^{\delta}$ holds provided that 
$$ T_{\delta} \ge \frac{1}{\delta} \left(\frac{M''+4}{\zeta^2 a_i}\right)$$

\noindent For the second inequality, notice that, since $\phi''\ge 0$ on $[\frac{1}{2},+\infty)$ and $\phi'\le 0$ we only need to look at the ratio defining $\sigma_{x_it}^{\delta}$ on $T_i\times[T_\delta, 2T_\delta]$. Let us denote by $m''=\min \phi''$ then, it is sufficient to choose $T_\delta$ such that:
$$\frac{m''}{T_\delta^2}+\frac{4}{T_\delta}\ge 0 \iff\frac{4}{T_\delta}\ge\frac{-m''}{T_\delta^2}\iff T_\delta\ge\sqrt{\frac{-m''}{4}}$$

\noindent Thus, the sectional curvature $\sigma_{x_it}^{\delta}$ is pinched as required provided that $$ T_{\delta} \ge \max  \left\{ \frac{1}{\delta} \left(\frac{M''+4}{\zeta^2 a_i}\right),\sqrt{\frac{-m''}{4}}\right\}  $$

Analogously, the condition
$$ T_{\delta} \ge \max  \left \{ \frac{1}{\delta} \left(\frac{M''+4}{\zeta^2 b_i}\right),\sqrt{\frac{-m''}{4}}\right \} $$
ensures that inequalities $-1-\delta \le \sigma_{y_it}^{\delta} \le \delta$ hold.\\

It follows by  the discussion about curvatures that, for $\delta<1 $, choosing
\small
$$T_{\delta}\ge\frac{1}{\delta}\cdot\max_{i=1,..,r}\left\{\frac{4(1-\zeta^2a_i)}{\zeta^2a_i}, \frac{4(1-\zeta^2b_i)}{\zeta^2b_i}, \frac{M''+4}{\zeta^2a_i}\frac{M''+4}{\zeta^2b_i}, \sqrt{\frac{-m''}{4}}\right\}=\frac{C}{\delta}$$
\normalsize
where the constant $C$ only depends on $(X',hyp)$, $\{h_i \}_{i=1}^r$ and on the choice of the original horoball neighbourhood $\mathcal U$, the sectional curvatures are pinched as required in the statement. From now on we shall consider $T_\delta=\frac{C}{\delta}$. With this proviso, the metric $hyp_\delta$ verifies properties $(1)$---$(4)$.\\

Finally, we have to check Property $(5)$, concerning the convergence of the volumes. 
Let us denote by
$\overline{V}_{T_{\delta}}$  the volume of $(X'\smallsetminus\mathcal U_\delta, hyp)$, by 
$V_1(\delta)$ the sum of the volumes respect of the warped toroidal cylinders $(T_i\times[T_\delta, 2T_\delta], hyp_{\delta})$ and by 
$V_2(\delta)$ the sum of the volumes of $(T_i\times[2T_\delta,+\infty], hyp_\delta)$. 
With this notation:
$$ \Vol(X', hyp_{\delta})= \overline{V}_{g_0, T_{\delta}}+ V_1(\delta)+V_2(\varepsilon).$$
By construction $\overline{V}_{T_{\delta}} \overset{\delta \rightarrow 0}{\longrightarrow} \Vol(X',hyp)$. For any $i=1,\dots,r$ and any $t\ge T_\delta$ we denote by  $\mathcal A_i(\delta,t)=\mathcal{A}( T_i\times\{t\}, hyp_{\delta})$ the area of $T_i \times \{t\}$ with respect to the restriction of the metric $hyp_{\delta}$.
Then:
$$ V_1(\delta)= \sum_{i=1}^r \int_{T_{\delta}}^{2T_{\delta}}e^{-2t} \cdot \mathcal{A}_i(\delta,t)dt.$$
Now observe that for $t\in[T_\delta,2T_\delta]$,
$$\mathcal{A}_i(\delta,t)= \mathcal{A}\left(T_i, \phi\left(\frac{t-T_\delta}{T_\delta}\right) k_i+\left(1-\phi\left(\frac{t-T_\delta}{T_\delta}\right)\right)\zeta^2 h_i\right).$$
We define: $$A=\max_i \max_{t \in \left[0,1 \right] }\mathcal{A}\left(T_i,\phi(t)k_i+(1-\phi(t))\zeta^2h_i\right) .$$ Then:
\small
$$
	V_1(\delta)=\sum_{i=1}^r \int_{ T_{\delta}}^{2T_{\delta}}e^{-2t} \cdot \mathcal{A}_i(\delta,t)dt \le r A \cdot \int_{T_{\delta}}^{2T_{\delta}}e^{-2t}=  r A\huge \left[-\frac{e^{-2t}}{2} \huge \right]_{T_{\delta}}^{2T_{\delta}} =$$$$
	 = r A e^{-T_{\delta}} \left(\frac{1}{2}-\frac{e^{-T_\delta}}{2}\right)\overset{\varepsilon \rightarrow 0}{\longrightarrow}0
$$

\normalsize

\noindent Finally, we give an estimate for $V_{2}(\delta)$. Let $V_{max}=\max_{i=1,..,r} \mathcal{A}(T_i,\zeta^2h_i)$, then:
$$V_2(\delta)\le\int_{2T_\delta}^{+\infty} e^{-2t}V_{max}dt=V_{max}\cdot\left[-\frac{e^{-2t}}{2}\right]_{2T_\delta}^{+\infty}=V_{max}\,\frac{e^{-4 T_\delta}}{2}\overset{\delta\rightarrow0}{\longrightarrow} 0$$
Since for $\delta \rightarrow 0$ we have $\overline{V}_{T_{\delta}} \rightarrow \Vol(X',hyp)$, $V_1(\delta) \rightarrow 0$ and $V_2(\delta) \rightarrow 0$, we conclude:
\small
$$  \Vol(X', hyp_{\delta}) \overset{\delta \rightarrow 0}{\longrightarrow} \Vol(X',hyp)$$
\normalsize
\end{proof}

\begin{proof}[Proof of Theorem \ref{thm_metrics_hyperbolic_pieces}]
	Let $X$ be a compact, irreducible, orientable, $3$-manifold with non-empty toroidal boundary $\partial X=T_1,..,T_r$, whose interior $int(X)$ admits a complete hyperbolic metric $hyp$ of finite volume. Fix a horoball neighbourhood $\mathcal U$ of the cusps of $(int(X),hyp)$ (as in the statement of Lemma \ref{lemma_change_of_conformal_type}). Hence, Lemma \ref{lemma_change_of_conformal_type} provides us with a continuous $1$-parameter family of complete Riemannian metrics $hyp_\delta$, satisfying $(1)$---$(5)$. Let the horoball neighbourhood $\mathcal U$ be parametrized as:
	$$\mathcal U=\bigcup_{i=1}^r T_i\times[0,+\infty)$$
	Then, the metrics $hyp_\delta$ in restriction to $T_i\times[2T_\delta,+\infty)$ can be written as  $e^{-2t}\zeta^2h_i\oplus dt^2$. We shall now apply Proposition \ref{prop_interpolation} for $\ell=\zeta e^{-3T_\delta}$ and $\delta$: there exists a $C^2$ function $\varphi_{\varepsilon,\varepsilon'}$ (where $\varepsilon, \varepsilon'$ are sufficiently small positive constants) such that
		\begin{enumerate}
		\item $\varphi_{\varepsilon, \varepsilon'}$ is not increasing and convex;
		\item $\varphi_{\varepsilon, \varepsilon'}$ verifies $$\left \lbrace \begin{array}{l}
		\varphi_{\varepsilon, \varepsilon'}(t)= \varphi_0(t)= \ell e^{-t}\mbox{ for }t\le -\varepsilon\\
		\varphi_{\varepsilon, \varepsilon'}(t)= \ell'\mbox{ for }t\ge t_\delta+\varepsilon'
		\end{array} \right.$$
		where:
		$ \ell \frac{\sqrt{\delta (1 + \delta)}}{1+ 2 \delta} \le\ell'= \ell'(\delta, \varepsilon, \varepsilon') \le 4\ell \frac{\sqrt{\delta(1+ \delta)}}{1+2\delta}$.
		\item $\left(\dfrac{\varphi_{\varepsilon, \varepsilon'}'}{\varphi_{\varepsilon, \varepsilon'}}\right)^2 \le (1+ 2 \delta)^2$ for every $t \in \mathbb{R}$;
		\item $\dfrac{\varphi_{\varepsilon, \varepsilon'}''}{\varphi_{\varepsilon, \varepsilon'}} \le (1+2\delta)^2$ for every $t \in (-\varepsilon, t_{\delta}+ \varepsilon')$.
	\end{enumerate}
	where $t_{\delta}= \frac{1}{2(1+2 \delta)} \ln \left(1+ \frac{1}{\delta}\right)$. Now, let $\varphi_{i,\delta}(t)=\zeta e^{-t}$ for $t\in[2T_\delta, \frac{5T_\delta}{2}]$ and $\varphi_{i, \delta}(t)=\varphi_{\varepsilon,\varepsilon'}(t)$ for $t\in[\frac{5T_\delta}{2},3T_\delta+2t_\delta]$. 
	Then glue the Riemannian metric $hyp_\delta$ in restriction to $T_i\times[0,\frac{5T_\delta}{2}]$ to the Riemannian  $C^2$ metric given by $\varphi_{i,\delta}(t)^2h_i\oplus dt^2$ on $[2T_\delta, 3T_\delta+2t_\delta]$; remark that both metrics on $T_i\times[T_\delta,2T_\delta]$ have the form $e^{-2t}\zeta^2h_i\oplus dt^2$ by construction. The result of these gluings for any $i=1,...,r$ produces a $C^2$ Riemmanian metric $hyp_{\delta}^0$ on $X$ satisfying properties $(1)$---$(4)$. Properties $(1)$, $(3)$, $(4)$ are straightforward from the construction; to check property $(2)$ one needs to adapt the argument for the volume convergence from Proposition \ref{proposition_metrics:on:Seifert:components} and use it in combination with Lemma \ref{lemma_change_of_conformal_type}, $(5)$.
\end{proof}

\section[An upper bound to MinEnt]{An upper bound to MinEnt}\label{section_upper_estimate_irreducible}

In this last section of Chapter 3 we shall give a proof of Theorem \ref{thm_construction_sequence_on_irreducible_Y}, thus exhibiting a sequence of $C^2$ Riemannian metrics of non-positive sectional curvatures converging to the conjectural value of the Minimal Entropy of $Y$.

\begin{proof}[Proof of Theorem \ref{thm_construction_sequence_on_irreducible_Y}]
	Since $Y$ is a closed irreducible, orientable $3$-manifold, with non-trivial $JSJ$ decomposition and at least one hyperbolic $JSJ$-component, we know that there exists a $JSJ$ compatible collection of flat metrics $\{\{h_{i,j}\}_{j=1,..,l_i}\}_{i=1,...,n}$ on the boundary components $\partial X_i$'s of the $JSJ$ components of $Y$. By Proposition \ref{proposition_metrics:on:Seifert:components} and Theorem \ref{thm_metrics_hyperbolic_pieces} we know that for any choice of $\delta$ there exists a suitable $\bar\zeta(\delta)$ depending on $Y$ and the $JSJ$ compatible collection $\{h_{i,j}\}$ such that:
	\begin{itemize}
	\item on every hyperbolic $JSJ$ component $X_i$, $i=1,..,k$ of $Y$ there exists a metric $hyp_{i,\delta}^0$ verifying $(1)$--$(3)$ and $(4)$ for the constant $\ell'=\bar\zeta(\delta)$;
	\item on every Seifert fibred $JSJ$ component $X_i$, $i=k+1,...,n$ of $Y$ there exists a metric $k_{\delta,\bar\zeta(\delta)}^i$ satisfying properties $(i)$---$(iv)$;
	\end{itemize}
     By construction  and $JSJ$ compatibility of the $\{h_{i,j}\}$ collection these metrics on the $JSJ$ components fit together to produce a non-positively curved $C^2$ Riemannian metric $g_{\delta}$ on $Y$ with the following properties:
     \begin{enumerate}
     	\item $-(1+2\delta)^2\le \sigma_{\delta}\le 0$;
     	\item $\Vol(Y, g_\delta)=\sum_{i=1}^k\Vol(X_i,hyp_{i,\delta}^0)+\sum_{i=k+1}^n\Vol(X_i, k_{\delta,\bar\zeta(\delta)}^i)$ so that, by what has been proved in Proposition \ref{proposition_metrics:on:Seifert:components} and Theorem \ref{thm_leeb_ator} we have:
     	$$\Vol(Y,g_\delta)\overset{\delta\rightarrow 0}{\longrightarrow}\sum_{i=1}^k\Vol(int(X_i), hyp_i)$$
     \end{enumerate}
     Using Bishop-Gunther Theorem we get from (1) that $\ent(Y,g_\delta)\le 2(1+2\delta)$, hence for any $\eta$ we can find a sufficiently small $\delta$ such that
     $$\VolEnt(Y,g_\delta)\le (1+\eta)\cdot 2\cdot\left(\sum_{i=1}^k\Vol(X_i, hyp_i)\right) ^{1/3}$$
     This implies the estimate:  $$\minent(Y)\le 2\cdot \left(\sum_{i=1}^k\Vol(X_i,hyp_i)\right)^{1/3}$$.
\end{proof}

\chapter[Bounding MinEnt from above: reducible manifolds]{Bounding MinEnt from above: \\ reducible $3$-manifolds} \label{section_a_conjectural_minimizing_sequence_reducible}

 In this Chapter, we shall provide the ({\it a posteriori}) optimal estimate from above for the Minimal Entropy in the case where $Y$ is a closed, orientable, reducible $3$-manifold. Namely, we shall prove the following statement.
	\begin{thm} \label{thm_upper_estimate_reducible}
		Let $Y$ be a closed, orientable, reducible $3$-manifold, and suppose that $Y=Y_1 \# \dots \# Y_m$ is the decomposition of $Y$ in prime summands. For every $i$, denote by $X_i^1, \dots, X_i^{n_i}$ the hyperbolic $JSJ$ components of $Y_i$, and by $hyp_i^j$ the corresponding complete, finite volume metric on $int(X_i^j)$.
		Then:
		\small
		$$ \minent(Y)^3 \le  2^3 \cdot \left(  \sum_{i=1}^{m} \sum_{j=1}^{n_i} \Vol \left( int(X_i^j),hyp_i^j\right) \right)$$
		\normalsize
	\end{thm}
	
The proof of such an estimate relies on the existence of sequences of metrics converging to the Minimal Entropy on the single prime summands. Indeed, also in this case we shall construct explicitly a family of metrics whose volume-entropies are smaller than any constant strictly greater than the previous upper bound.
%
%
\begin{rmk}
	It is worth to stress that Theorem \ref{thm_upper_estimate_reducible} will be presented as a consequence of a more general statement concerning the Minimal Entropy of connected sums of closed, orientable, $n$-dimensional manifolds (for $n\ge 3$), see Theorem \ref{thm_upper_estimate}.
\end{rmk}

Let us briefly introduce the \textit{Poincaré series} which will be the key tool to achieve this result.

\section{Entropy and the Poincaré series}
\label{section_integral_characterization}
It is well known that, given any Riemannian manifold $(Y,g)$ with Riemannian distance $d$, its entropy can be characterized as:
$$(\star)\quad\ent(Y,g)= \inf \left\lbrace  c \in \mathbb{R} \Big \arrowvert \normalsize \int_{\widetilde{Y}} e^{-c d(y,z)} dv_{\tilde g}(z)   < \infty\right\rbrace$$
where $\tilde g$ is the usual pull-back metric on the universal cover.

Let us fix any base point $y$ for the fundamental group of $Y$. The Poincaré series of $(\pi_1(Y, y),g)$  computed at $s$ is : $$ (\star\star)\hspace{1.25cm} \mathcal{P}_s(\pi_1(Y,y),g)= \sum _{\gamma \in \pi_1(Y,y)} e^{-s d(\widetilde{y}, \gamma. \widetilde{y})} \hspace{1.25cm}$$
It is straightforward that the abscisse of convergence of $(\star)$ coincides with the abscisse of convergence of $(\star\star)$, also known as {\it critical exponent of $\mathcal P_s(\pi_1(Y,y), g)$}. 
Indeed, let $\mathcal{D}$ denote  a fixed fundamental domain for the action of $\pi_1(Y, y)$ on $\widetilde{Y}$. Then, $$\widetilde{Y}=\smashoperator{\bigsqcup_{\gamma \in \pi_1(Y, y)} } \gamma . \mathcal{D}\,,$$ and if we denote by $D= \diam (\mathcal{D}, \widetilde{g})$, the following inequalities hold:

\begin{align*}
&e^{-cD} \Vol(\mathcal{D}, \widetilde{g})\sum_{\gamma \in \pi_1(Y, y)} e^{-c d(y, \gamma . y)} \le  \int_{\widetilde Y} e^{-c d(y, z)} dv_g(z) \le\\
\le & e^{+cD} \Vol(\mathcal{D}, \widetilde{g})\sum_{\gamma \in \pi_1(Y, y)} e^{-c d(y, \gamma . y)}
\end{align*}

\noindent This shows that the integral converges at $s$ if and only if $\mathcal{P}_s(\pi_1(Y,y),g) < \infty$.\\

\section{A Riemannian metric on the connected sum}\label{section_metric_connected_sum}

Let $(Y', g')$, $(Y'', g'')$ be two closed, oriented, connected Riemannian $n$-manifolds for $n\ge 3$. We shall now construct a family of metrics on $Y=Y'\#Y''$. Let us assume to have fixed $y'\in Y'$, $y''\in Y''$ and two positive real numbers $L, r>0$ with $r<\frac{1}{2}\cdot\min\{\inj(Y',g'), \inj(Y'',g'')\}$. Excide the geodesic balls of radius $r$ around $y'$, $y''$ and glue the Riemannian tube $(S^{n-1}\times[-L-r, L+r], k_{L,r})$ to $(Y'\smallsetminus B_{g'}(y',r), g')\cup(Y'\smallsetminus B_{g''}(y'',r), g'')$ where $k_{L,r}$ is defined as follows. Let $\varphi:[-L-r, L+r]\rightarrow [0,1]$ be a smooth function which is identically equal to $1$ on an $\frac{r}{3}$-neighbourhood of $\{-L-r, L+r\}$, identically equal to zero on $(-L-\frac{r}{3}, L+\frac{r}{3})$, which is non-increasing on $[-L-r,0]$ and non-decreasing on $[0, L+r]$.
 Then, let us denote by $g_{S_r'}$ and $g_{S_r''}$ the restrictions of $g'$ and $g''$ respectively to $S_{r}(y')$ and $S_r(y'')$ (the spheres of radius $r$ centered at $y'$, $y''$ respectively). We shall define $k_{L,r}$ as follows:
<
\noindent where we denoted by $\chi_I$ the characteristic function of the interval $I$ (see Figure \ref{fig_long_tube}). In the sequel we shall assume to have chosen $r$ such that $\pi r^2/2\le r/3$ or equivalently that $r\le 2/(3\pi)$. We shall denote by $(Y, g_{L,r})$ the Riemannian manifold which is the result of the gluing, see Figure \ref{fig:connectedsumnew}.

\begin{figure}
	\centering
	\includegraphics[width=0.9\linewidth]{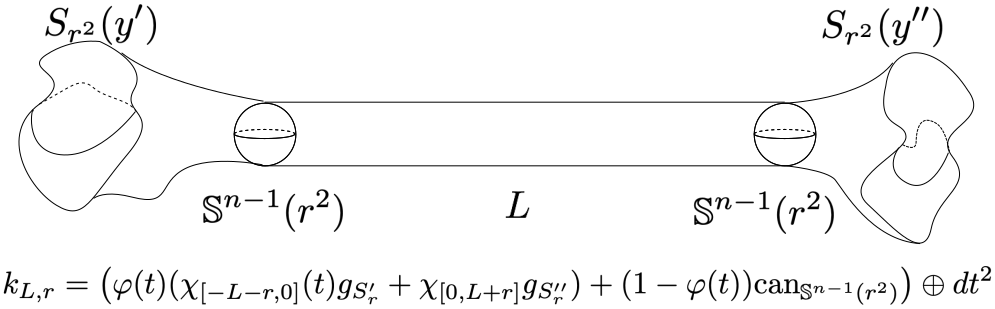}
	\caption[Metric on the gluing tube]{ The metric $k_{L,r}$ on the \lq \lq tube '' of length $R$ connecting the two geodesic balls of radius $r$ $S_r(y')$ and $S_r(y'')$. }
	\label{fig_long_tube}
\end{figure}

The main result of the next two sections is the following theorem:
\begin{thm} \label{thm_upper_estimate}
	Let $(Y', g')$ and $(Y'', g')$ be two Riemannian manifolds. Let $L, r>0$ be two positive real constants and $(Y, g_{L,r})$ be the Riemannian manifold constructed as above. There exists a constant $\bar L>0$ depending on $(Y', g')$, $(Y'', g'')$ such that for any $L>\bar L$ the following holds:
	\begin{itemize}
		\item[(1)] the critical exponent  of the series \small $\mathcal{P}_s(\pi_1(Y,y), g_{L,r})$ \normalsize  is smaller than or equal to  \small  $\max \{ \ent(Y', (g')^{\delta}), \ent(Y'',(g'')^{\delta})\}$. \normalsize
	\end{itemize}
	\noindent	Moreover, chosen $L>\bar L$, for every $\varepsilon>0$ there exists an $\bar r(L,\varepsilon)$ such that for every $r<\bar r(L,\varepsilon)$:
	\begin{itemize}
		\item[(2)] \small$|\Vol(Y,g_{L,r}) - \Vol(Y', g')- \Vol(Y'', g'')|< \varepsilon$. \normalsize
	\end{itemize}
\end{thm}

\begin{figure} [h]
	\centering
	\includegraphics[width=0.9\linewidth]{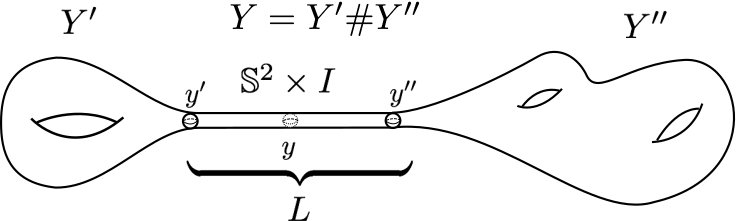}
	\caption[Metrics on reducible $3$-manifolds]{The Riemannian manifold $(Y,g_{L,r})$ resulting from the gluing.}
	\label{fig:connectedsumnew}
\end{figure}

As a consequence, we have the following estimate which is valid for any dimension $n\ge 3$:

\begin{cor}\label{cor_minent_subadditivity}
	Let $Y'$, $Y''$ be two closed, orientable, differentiable $n$-manifolds for $n\ge 3$. Then:
	$$\minent(Y'\# Y'')^3\le \minent(Y')^3+\minent(Y'')^3,$$
	where $Y'\# Y''$ denotes the connected sum of $Y'$, $Y''$.
 \end{cor}

\begin{proof}
	Let $g_k'$ and $g_k''$ two sequences of metrics converging to the Minimal Entropy of $Y'$ and $Y''$ respectively. Up to normalization, we can assume that $\mbox{Ent}(Y',g_k')= \mbox{Ent}(Y'',g_k'')$; this does not affect the minimality of the sequence. 
	The proof is a straightforward application of Theorem \ref{thm_upper_estimate} to these two sequences of metrics  on $Y'$, $Y''$ minimizing the invariant $\EntVol$.
\end{proof}

\section{The proof of Theorem \ref{thm_upper_estimate}}

Before proving Theorem \ref{thm_upper_estimate}, we shall fix some notation.
Denote by $d_{L,r}$ the Riemannian distance associated to $g_{L,r}$.
Recall that for any $h \in \pi_1(Y,y)$ the value $d_{L,r}(\widetilde{y}, h. \widetilde{y})$ does not depend on the lift chosen for $y$. From now on we shall denote by $|h|_{y}$ the length, with respect to the metric $g_{L,r}$ of the shortest geodesic loop representing $h$. Namely, $|h|_{y}= d_{L,r}(\widetilde{y}, h . \widetilde{y})$, where $\widetilde{y}$ is any lift of $y$.
Similarly, for any choice of $y'\in Y'$, $y''\in Y''$ we define $|\cdot|_{y'}: \pi_1(Y', y') \longrightarrow \mathbb{R}$ and $|\cdot|_{y''}: \pi_1(Y'', y'') \longrightarrow \mathbb{R}$ the function mapping any class in $\pi_1(Y', y')$, $\pi_1(Y'',y'')$ into the length of the shortest geodesic representative with respect to $g'$, $g''$.\\

	Let $y\in Y$ be a point contained in the set $S^{n-1} \times \{0\}$ and $y', y''$ be the centers of the balls removed from $Y', Y''$ respectively when constructing the metric $g_{L,r}$. We shall denote by $$\Phi: \pi_1 (Y, y) \rightarrow \pi_1(Y', y') \ast \pi_1(Y'', y'')$$ the isomorphism obtained in two steps:
	\begin{enumerate}
		\item Let $(\widehat Y, \hat d_{L})$ be the metric space obtained as the junction of $(Y', g')$ and $(Y'', g'')$ via an interval $[-L,L]$ (with the standard metric) joining $y'$ to $y''$. Let $p:(Y,g_{L,r})\rightarrow (\widehat{Y}, \hat d_{L})$ be the map obtained by collapsing the tube $S^{n-1}\times[-L,L]$ to the interval $[-L,L]$, sending $(S^{n-1}\times[-L-r,-L], g_{L,r})$ onto $B_{g'}(y',r)$, via the map $S^{n-1}\times\{t\}\mapsto \partial B_{g'}(y',r)$ and $(S^{n-1}\times[L,L+r], g_{L,r})$ onto $B_{g''}(y'',r)$ via a similar map. Finally $(Y'\smallsetminus B_{g'}(y',r), g_{L,r})$ is sent identically onto itself (and similarly $(Y''\smallsetminus B_{g''}(r), g_{L,r})$).
		\item Call $\hat y\in\hat Y$ the point $p(S^{n-1}\times\{0\})=p(y)$. By construction we know that $p_*:\pi_1(Y,y)\rightarrow \pi_1(\widehat Y, \hat y)$  is an isomorphism. On the other hand it is obviously true that  $\pi_1(Y', y')*\pi_1(Y'',y'')$ is isomorphic to $\pi_1(\widehat Y, \hat y)$, where the isomorphism is given by the following map: fix $\gamma_{\hat y,y'}$ and $\gamma_{\hat y,y''}$ the unique unit speed geodesics in $(\widehat Y, \hat d_{L})$ joining $\hat y$ to $y'$ and $y''$ respectively. Then, the isomorphism is given by:
		\small
		\begin{equation*}
		g=[\gamma_1']\cdot[\gamma_1'']\cdots [\gamma_n']\cdot[\gamma_n'']\overset{\Psi}{\mapsto} [\gamma_{\hat y,y'}^{-1}*\gamma_1'*\gamma_{\hat y,y'}]\cdots[\gamma_{\hat y,y''}^{-1}*\gamma_n''*\gamma_{\hat y,y''}]
		\end{equation*}
		\normalsize
		for any $g\in\pi_1(Y', y')*\pi_1(Y'', y'')$.
	\end{enumerate}
Hence we define $\Phi=(p_*)^{-1}\circ\Psi$, and $\Phi': \pi_1(Y', y') \hookrightarrow \pi_1(Y, y)$ and $\Phi'': \pi_1(Y'', y'') \hookrightarrow \pi_1(Y, y)$ the monomorphisms obtained restricting $\Phi$, which embed $\pi_1(Y',y')$ and $\pi_1(Y'',y'')$ into $\pi_1(Y,y)$.
	Moreover, we denote by $G_1= \Phi' \left( \pi_1(Y', y')\right)$ and $G_2= \Phi'' \left( \pi_1(Y'', y'')\right)$.\\
	By definition of $\pi_1(Y', y') \ast \pi_1(Y'', y'')$ every $g \in \pi_1(Y,y)$ can be (uniquely) written as: $g= h_1'h_1'' \cdots h_n' h_n''$,
	where $h_1'$ and $h_n''$ are possibly equal to the identity, $h_i' \in \pi_1(Y', y')\smallsetminus\{\id\}$ for $i=2, \dots, n$ and $h''_i \in \pi_1 (Y'', y'')\smallsetminus\{\id\}$ for $i=1, \dots, n-1$. Notice that we are making a little abuse of notation since  we are omitting the identification of the elements of $\pi_1(Y',y')$ and $\pi_1(Y'',y'')$ with the elements of $\pi_1(Y,y)$ which correspond via the isomorphisms $\Phi'$, $\Phi''$.
	\begin{lem}\label{lem_length_estimate}
	With the above notation, assume $L \ge \frac{\pi r}{2}$. For any element  $g \in \pi_1(Y, y)$ let $g=h_1'h_1'' \dots h_n'h_n''$ be a reduced syllabic form for $g$.
	Then:
		$$ \lvert g\rvert_{y} \ge\sum_{i=1}^n(|h_i'|_{y'}+2L) +\sum_{i=1}^n(|h_i''|_{y''}+2L)$$
	\end{lem}
\begin{proof} First of all we observe that by construction of the metric the sets $S^{n-1}\times \{t\}$ where $t \in [-L- \frac{r}{3}, L + \frac{r}{3}]$ are totally geodesic hypersurfaces (because the metric is a Riemannian product). Moreover the set $S^{n-1}\times \{t\}$ for $t\in[-L,L]$ is such that every minimizing geodesic $\alpha_{PQ}$ joining pair of points $P$, $Q \in S^{n-1}\times\{t\}$ is entirely contained in $ \subset S^{n-1}\times \{t\}$. Indeed, any curve $\beta_{PQ}$ joining the two points and contained in the tube $(S^{n-1}\times[-L-\frac{r}{3},L+\frac{r}{3}], k_{L,r})$ has length greater than or equal to the length of its projection on $S^{n-1} \times \{t\}$ with equality holding only for those curves contained into $S^{n-1}\times\{t\}$, because the metric in the tube is a Riemannian product. Furthermore,  since $r\le 2/(3\pi)$, every curve joining $P$ and $Q$ and escaping the tube $(S^{n-1}\times[-L-\frac{r}{3}, L+\frac{r}{3}], k_{L,r})$ has length greater than $\pi r$, which is equal to the diameter of $(S^{n-1} \times \{t\}, k_{L,r}|_{S^{n-1}\times\{t\}})$.
	Let $\gamma$ be a shortest representative in the homotopy class of $g$. Denote by $\gamma_i', \gamma_i''$ the connected subsegments of $\gamma$ entirely contained in one of the two sides separated by $S^{n-1} \times \{0\}$. Observe that: 
	\begin{equation}\label{cardinality}
	|\{\gamma_i'\}|=|\{h_i'\,:\,h_i'\neq 1\}|,\quad |\{\gamma_i''\}|=|\{h_i''\,:\,h_i'' \neq 1\}|
	\end{equation}
	Indeed, for any $i=1,...,m$ and any $j=1,...,k$ consider two minimizing geodesics $\xi_i'$, $\xi_j''$ contained in $S^{n-1} \times \{0\}$ and connecting $y$ respectively  to $\gamma_i'(0)$ and  to $\gamma_j''(0)$ and  call $\tilde{\gamma}_i'$ ($\tilde{\gamma}_j''$ respectively) the path obtained by joining $\gamma_i'$ to the $g_{L,r}$-geodesic connecting $\gamma_i'(1)$ to $\gamma_i'(0)$. By construction the loop $\gamma$ is homotopic to the join of the loops $[(\xi_{i}')^{-1}*\tilde{\gamma}_i'*\xi_i']$'s and $[(\xi_j'')^{-1}*\tilde{\gamma}_j''*\xi_j'']$'s following the order in which the $\gamma_i'$'s, $\gamma_j''$'s appear along $\gamma$. Notice that by construction the $\gamma_i'$'s and the $\gamma_j''$'s appear alternately. Let $p:Y\rightarrow \widehat Y$ be the map defined before and let 
	 $$\Phi^{-1}=\Psi^{-1}\circ p_*:\pi_1(Y,y)\rightarrow\pi_1(Y',y')*\pi_1(Y'',y'')$$ be the inverse of the isomorphism constructed before the Lemma. Then,
	  $$\Phi^{-1}([(\xi_{i}')^{-1}*\tilde{\gamma}_i'*\xi_i'])\in\pi_1(Y',y')\smallsetminus\{1\}
	  \mbox{ and } \Phi^{-1}([(\xi_{j}'')^{-1}*\tilde{\gamma}_j''*\xi_j''])\in\pi_1(Y'',y'')\smallsetminus\{1\}.$$ 
	  Otherwise, if $\Phi^{-1}([(\xi_{i}')^{-1}*\tilde{\gamma}_i'*\xi_i'])\in\pi_1(Y',y')\smallsetminus\{1\}$ were not trivial,  we could replace $\gamma_i'$ with  a minimizing geodesic $\delta_{i}'$ (contained in $S^{n-1}\times\{0\}$) connecting the two endpoints of $\gamma_i'$, thus obtaining a shortest representative of $\gamma$, thus contradicting the minimality of $\gamma$.	\\
	 As a consequence  $|\{\gamma_i'\}|=|\{h_i'\,:\, h_i'\neq 1\}|$ and $|\{\gamma_j''\}|=|\{h_j''\,:\,h_j''\neq 1\}|$. Moreover, by the previous discussion: $$g=\Phi^{-1}[\xi_1'^{-1}*\tilde{\gamma}_1'*\xi_1']\cdot \Phi^{-1}[\xi_1''^{-1}*\tilde{\gamma}_1''*\xi_1'']\cdots  \Phi^{-1}[\xi_n'^{-1}*\tilde{\gamma}_n'*\xi_n']\cdot \Phi^{-1}[\xi_n''^{-1}*\tilde{\gamma}_n''*\xi_n'']$$
	By the unicity of the reduced form of an element in a free product we deduce that: $\Phi^{-1}([(\xi_i')^{-1}*\tilde{\gamma}_i'*\xi_i])=h_i'$, $\Phi^{-1}([(\xi_j'')^{-1}*\tilde{\gamma}_j''*\xi_j''])=h_j''$ for every $i,j=1, \dots n$.\newline
	Recall that our goal is to provide a sufficiently accurate estimate from below to the length of the shortest representative $\gamma$ of $g$ with respect to the Riemannian metric $g_{L,r}$ in terms of $L, r$ and of the lengths of the shortest representatives (with respect to the metrics $g'$, $g''$) of the elements $h_i'$'s, $h_j''$'s of $\pi_1(Y',y')$ and $\pi_1(Y'',y'')$ for every $i,j=1 \dots, n$.\\
	
	\noindent{\bf Claim (0).} Any $\gamma_i'$ (respectively $\gamma_j''$) is such that $|\gamma_i'^{-1}(S^{n-1}\times\{t\})|<+\infty$ for any $t\in(-L,L)$ (respectively, $|\gamma_j''^{-1}(S^{n-1}\times\{t\})|<+\infty$ for any $t\in(-L,L)$).\\
	
	\noindent{\it Proof of Claim (0).} Straightforward from the definition of the metric. \qed\\
	
	\noindent{\bf Claim (1).} Any $\gamma_i'$ (respectively $\gamma_j''$) is such that $|\gamma_i'^{-1}(S^{n-1}\times\{t\})|=2$ for any $t\in(-L,0]$
	(respectively, $|\gamma_j''^{-1}(S^{n-1}\times\{t\})|=2$ for any $t\in[0, L)$).\\
	
	\noindent{\it Proof of Claim (1).} Let us observe first that $|\gamma_i'^{-1}(S^{n-1}\times\{t\})|\ge 2$  for every $t\in[-L, 0]$. Otherwise the loop $\xi_i'^{-1}*\tilde{\gamma}_i'*\xi_i'$ would be trivial, which is not. On the other hand assume that $\gamma_i'$ crosses a sphere $S^{n-1}\times\{t\}$, with $t\in(-L,0]$, for at least three times. Then, we would have that $|\gamma_i'^{-1}(S^{n-1}\times\{-L\})|\ge 4$. Now, let $s_2$ and $s_{2k-1}$ be respectively the first time in which $\gamma_i'$ crosses $S^{n-1}\times\{-L\}$ in direction $S^{n-1}\times\{0\}$ and the last time in which $\gamma_i'$ crosses $S^{2}\times\{-L\}$ in direction $Y'$. Connect $\gamma_i'(s_2)$ with $\gamma_i'(s_{2k-1})$ with a geodesic segment in $(S^{n-1}\times\{-L\}, g_{L,r}|_{S^{n-1}\times\{-L\}})$, call it $\zeta_i'$. Then the path $\gamma_i'|_{[0,s_2]}*\zeta_i'*\gamma_i'|_{[s_{2k-1}, 1]}$ is homotopic with fixed ends to $\gamma_i'$, and  is shorter. This is a contradiction because substituting $\gamma_i'$ with this new path would produce a shorter representative for $g$.\qed\\
	
	\noindent We shall now look at the path $\gamma_i'$  outside the tube $S^{n-1}\times(-L,0]$ (respectively,  $S^{n-1}\times [0,L)$). Let $\gamma_i'(1)=\gamma_i'|_{I_1}$,..., $\gamma_i'(k_i)=\gamma_{i}'|_{I_{k_i}}$  be the subpaths of $\gamma_i'$ contained in $(S^{n-1}\times[-r-L,-L], g_{L,r}|_{S^{n-1}\times[-r-L,-L]})$. Now, look at $p(\gamma_i')$ in restriction to $Y'$ and observe that it is a loop based at $y'$. Moreover, its  homotopy class $[p(\gamma_i')|_{Y'}]=h_i'$. Now, let $\eta_i'(j)=p(\gamma_i'(j))$. The following hold:
	$$\ell_{g_{L,r}}(\gamma_i'(j))\ge \ell_{quot}(\eta_i'(j))\ge r=\ell_{g'}(\zeta_i'(j)))$$
	where we denoted by $\ell_{quot}$ the quotient metric of $(S^{n-1}\times[-L-r,-L], g_{L,r})/\sim$ where $(x,t)\sim (y,t')$ if and only if $t=t'=-L$; and where we introduced the notation $\zeta_i'$ for the $g'$-geodesic segment joining $y'=\eta_i'(j)(0)$ to the other endpoint $\eta_i'(j)(1)$ (see Figure \ref{fig:reduciblemetrics}).
	\begin{figure}
		\centering
		\includegraphics[width=0.8\linewidth]{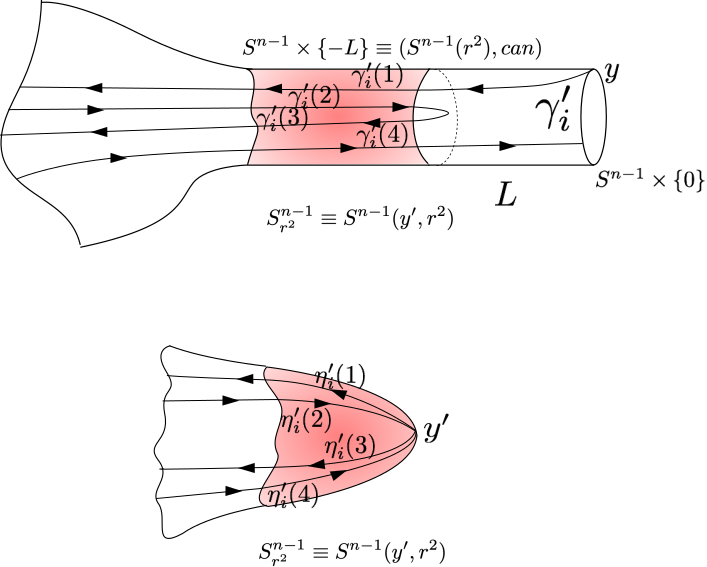}
		\caption[Geodesic loop through the projection map]{The projection $p:(Y,g_{L,r})\rightarrow (\widehat{Y}, \hat d_{L})$, and the images $\eta_i'(j)= p(\gamma_i')$.}
		\label{fig:reduciblemetrics}
	\end{figure}
			Since we are in balls of radius shorter than the injectivity radius of $(Y', g')$, $\eta_i'(j)$ is homotopic with fixed endpoints to $\zeta_i'(j)$. On the other hand, observe that $p$ in restriction to $(Y'\smallsetminus B_{g'}(y',r), g_{L,r})$ is an isometry on its image. Hence, the path obtained from $p(\gamma_i')|_{Y'}$ replacing the $\eta_i'(j)$ with the $\zeta_i'(j)$ is a representative of the same homotopy class and its length is less than the length of $p(\gamma_i')|_{Y'}$. We deduce the following inequality:
	$$\ell_{g_{L,r}}(\gamma_i')\ge |h_i'|_{y'}+2d_{g_{L,r}}(S^{n-1}\times\{-L\}, S^{n-1}\times\{0\})=|h_i'|_{y'}+2L$$
	which is the desired estimate.
		\end{proof}
	
	Now we can prove Theorem \ref{thm_upper_estimate}
	
	\begin{proof} [Proof of Theorem \ref{thm_upper_estimate}]
		Let $g=g_1'g_1'' \cdots g_n'g_n'' \in \pi_1(Y,y)$, where $g_i'= \Phi'(h_i')$, $g_i'' \in \Phi''(h_i'')$, $h_i' \in \pi_1(Y', y')$ and $h_i'' \in \pi_1(Y'',y'')$, with possibly $h_1'=1$, $h_n''=1$.
		Then, by Lemma \ref{lem_length_estimate} we have:
		$$ \lvert g \rvert_{y} \ge \sum_{i=1}^n \left( \lvert h_i' \rvert_{y'} + \lvert h_i''\rvert_{y''} \right) +4nL$$
		Thus we get:
		\small
		$$ e^{-s \lvert g \rvert_{y}} \le e^{-4nL} \cdot e^{-s \lvert h_1'\rvert_{y'}} \cdots e^{-s \lvert h_n' \rvert_{y'}} \cdot e^{-s \lvert h_1''\rvert_{y''}} \cdots e^{-s \lvert h_n'' \rvert_{y''}}$$
		So the Poincaré series of $\pi_1(Y, y_0)$ with respect to the metric $g_{L,r}^{\delta}$ can be estimated as follows:
		$$ \mathcal P\left(\pi_1(Y,y), g_{L,r}\right) =\sum_{g\in\pi_1(Y,y)} e^{-s |g|_y} \le $$ 
		$$ \le 1 +  \sum_{\mathclap{ \substack{n \ge 1 \\ h_i' \in (\pi_1(Y',y'))^{\ast} \\ h_i'' \in (\pi_1(Y'',y''))^{\ast}}}}  \left(  e^{-s|h_1'\cdots h_n' h_n''|_{y}} + e^{-s |h_1' \cdots h_{n-1}''h_n'|_{y}}+ e^{-s |h_1''h_2' \cdots h_n''|_y}+ e^{-s |h_1''h_2' \cdots h_{n-1}''h_n'|_{y}} \right)  {\le}  $$
		$$ {\le}	1+4 \cdot \sum_{ n \ge 1}  \left( \sum_{h' \in (\pi_1(Y',y'))^{\ast}} e^{-s |h'|_{y}} \right)^n  \left(
		\sum_{h''\in (\pi_1(Y'',y''))^{\ast}} e^{-s |h''|_{y}} \right)^n {\le} $$
		$$ {\le} 1+ 4\sum_{n \ge 1}  \left( e^{-s 2L} \sum_{\mathclap{h' \in (\pi_1(Y',y'))^{\ast}}}e^{-s|h'|_{y'}}\right)^n \left( e^{-s 2L}  \sum_{\mathclap{h'' \in (\pi_1(Y'',y''))^{\ast}}} e^{-s |h''|_{y''}} \right)^n{\le}$$
		$$ {\le} 1+4  \sum_{n \ge 1} \left( e^{-4sL} \cdot \mathcal{P}^{\ast}_s(\pi_1(Y',y'), g') \cdot \mathcal{P}^{\ast}_s(\pi_1(Y'',y''),g'')\right)^n $$

		\normalsize
		
		\noindent where the second inequality is consequence of a rearrangement of the summands, the third follows from Lemma \ref{lem_length_estimate} and the last follows from the definition of the Poincaré series, with  $\mathcal{P}^{\ast}$ denoting the Poincaré series without the summand corresponding to the identity.\\
		Now for any arbitrarily chosen $$s>s_{Y', Y''}=\max\{\ent(Y',g'), \ent(Y'', g'')\}$$ the series $\mathcal P_s^*(\pi_1(Y',y'), g')$ and $\mathcal P_s^*(\pi_1(Y'',y''), g'')$ converge to $P_1$, $P_2$.
		Thus the problem of the convergence of $\mathcal{P}_s(\pi_1(Y, y), g_{L, r})$  is reduced to the convergence of the geometric series:
		
		$$ \sum_{\mathclap{n \ge 1}}\left( e^{-s 4L}  \cdot P_1 \cdot P_2\right) ^n$$
		This series converges if the length $L$ of the interpolation cylinder is large enough: $$L>\frac{\log(P_1\cdot P_2)}{4s_{Y', Y''}}$$
		As $s>s_{Y', Y''}=\max\{\ent(Y',g'), \ent(Y'', g'')\}$ this shows that the critical exponent of $\mathcal P(\pi_1(Y,y), g_{L,r})$ is less or equal to $s_{Y',Y''}$ for $L$ sufficiently large. This concludes the proof of Theorem \ref{thm_upper_estimate} (1).\\
		
		\noindent Let us now focus on assertion (2), {\it i.e.} the convergence of the volume. By construction of the metric $g_{L,r}$ we have:
		\small
		$$\Vol(Y'\smallsetminus B_{g'}(y',r), g')+\Vol(Y''\smallsetminus B_{g''}(y'',r), g'')\le \Vol(Y,g_{L,r})$$
		\normalsize
		but also:
		\small
		$$\Vol(Y, g_{L,r})=\Vol(Y'\smallsetminus B_{g'}(y',r), g')+\Vol(Y''\smallsetminus B_{g''}(y'',r), g'')+V(L,r)$$
		\normalsize
		where:
		$$V(L,r)=(S^{n-1}\times[-L-r, L+r], g_{L,r}|_{S^{n-1}\times[-L-r, L+r]})$$
		Now observe that in restriction to $S^{n-1}\times[-L-r, L+r]$ the metric $g_{L,r}$ is equal to the metric $k_{L,r}$ described in Chapter 4, \S\ref{section_metric_connected_sum}. 
		Moreover, denoting:
		$$A'=\max_{r\in[0,\frac{\inj(y')}{2}]}\{\mathcal A_{g'}(S_r'(y'))\}$$ 
		\vspace{-2mm}
		and 
		$$A''=\max_{r\in[0,\frac{\inj(y'')}{2}]}\{\mathcal A_{g''}(S_r''(y''))\}$$
		we observe that these are finite quantities. Call $A=\max\{A', A''\}$; then
		$$V(L,r)\le A\cdot2r+(\omega_{n-1} r^{n-1})\cdot (2L+2r),$$
		where $\omega_{n-1}=\mathcal A(\mathbb S^{n-1})$.
		Hence, for every fixed $L>\bar L$ and every $\varepsilon>0$ there exists $\bar r(L,\varepsilon)>0$ such that for every $r<\bar r(L,\varepsilon)$ we have:
		$$|\Vol(Y,g_{L,r})-\Vol(Y', g')-\Vol(Y'', g'')|<\varepsilon$$
		which proves part (2).
		\end{proof}
	
	\section{Proof of Theorem \ref{thm_upper_estimate_reducible}}
	
	Before going into the proof of Theorem \ref{thm_upper_estimate_reducible} we recall some facts:
	\begin{enumerate}
		\item Let $X$ be a closed, orientable Seifert fibred manifold or a graph-manifold, then we have $\minent(X)=0$ by \cite{anderson2003minimal}.
		\item Let $X$ be a closed, orientable manifold which admits a hyperbolic metric. Then by \cite{besson1995entropies} we know that the hyperbolic metric realizes the Minimal Entropy, that is $\minent(X)=2\cdot \left(\Vol(
		X, hyp)\right)^{1/3}$.
		\item Finally from Theorem \ref{thm_construction_sequence_on_irreducible_Y} we know that on every irreducible, closed, orientable $3$-manifold with non-trivial $JSJ$ decomposition and at least one hyperbolic $JSJ$ component we have $$\minent(X)= 2 \left(\sum_{i=1}^k \Vol(int(X_i),hyp_i)\right)^{1/3}$$ {\it i.e.}, $\minent(Y)^3$ is additive with respect to the connected sum.
	\end{enumerate}

\begin{proof}[Proof of Theorem \ref{thm_upper_estimate_reducible}]
	The theorem follows by iteratively applying Corollary \ref{cor_minent_subadditivity} to the prime summands of the manifold $Y$, keeping into account the estimates (1)---(3) existing for the Minimal Entropy of the prime summands.
\end{proof}

	\newpage

\chapter [The lower bound: irreducible $3$-manifolds]{The lower bound via the barycentre method:\\ irreducible $3$-manifolds} \label{section_a_lower_bound_irreducible}

In this Chapter we shall find the optimal lower bound for the Minimal Entropy of a closed, orientable, irreducible $3$-manifold $Y$.
 We shall handle the reducible case in Chapter 6.\\

 Let $Y$ be a closed, orientable, irreducible $3$-manifold, consider its  $JSJ$ decomposition and denote by $X_1, \dots, X_n$ the hyperbolic $JSJ$ components.  For every $i=1, \dots, n$ the manifold $int(X_i)$ can be endowed with a complete hyperbolic, finite volume metric, unique up to isometries by Mostow Rigidity Theorem; we shall denote it by $hyp_i$.\newline

The main result of this section is the following.

\begin{thm}  \label{thm_lower_estimate_irreducible}
	Let $Y$ be a closed, orientable, irreducible $3$-manifold, and let $X_1, \dots X_k$ be its hyperbolic $JSJ$ components. Then, the following inequality holds:
	$$\minent(Y) \ge 2 \cdot \left( \sum_{i=1} ^k \Vol(int(X_i),hyp_i) \right)^{1/3}.$$
\end{thm}

The last statement, together with Theorem \ref{thm_construction_sequence_on_irreducible_Y}  gives the exact computation of the minimal entropy in the irreducible case. In order to prove this inequality, we shall use the \textit{barycenter method} applied to a $1$-parameter family of $C^2$ Riemannian metrics $\{g_\delta\}$ on $Y$, which are non-positively curved and locally isometric to $\mathbb H^3$ on larger and larger open subsets (as $\delta\rightarrow0$), constructed in Chapter \ref{section_a_conjectural_minimizing_sequence_irreducible}.

\section{The barycentre method}
In this section we shall give a brief introduction to the \textit{barycenter method}. 
Roughly speaking, given  a map $f: (Y,g)\longrightarrow (X,g_0)$ which lifts to  a $\lambda$-equivariant map $\widetilde{f}: (\widetilde{Y},\tilde g) \longrightarrow (\widetilde{X}, \tilde g_0)$ between the universal covers ---here  $\lambda=f_*:\pi_1(Y)\rightarrow\pi_1(X)$ is the homomorphism induced by $f$ between the first homotopy groups---,  the method  consists in:
\begin{itemize}

 \item Immersing $(\widetilde{Y}, \tilde g)$ in $\mathcal{M}(Y,g)$ ---the space of positive and finite Borel measures on $(\widetilde{Y}, \tilde g)$--- via a family of $\pi_1(Y)$-equivariant maps $y \mapsto \mu_{c, y}$.
 \item Taking $\widetilde{f}_{\ast}(\mu_{c, y}) \in \mathcal{M}(X,g_0)$, the push-forward measure with respect to $\widetilde{f}$.
 \item Projecting from $\mathcal{M}(X,g_0)$ on $(\widetilde{X},\tilde g_0)$ by composing with the map which associates to any measure in $\mathcal M(X,g_0)$ its barycenter, that we are going to describe in detail in few lines.

\end{itemize}
This way, we shall obtain, via the composition of the three maps described above, a family of maps $\widetilde{F}_{c}: \widetilde{Y} \longrightarrow \widetilde{X}$ ---namely, $\widetilde{F}_c(y)= \bary [\widetilde{f}_{\ast} \mu_{c, y}] $:

\begin{center}
	\begin{tikzpicture}
	\matrix (m) [matrix of math nodes,row sep=3em,column sep=4em,minimum width=2em]
	{
		\mathcal{M}(\widetilde{Y})& \mathcal{M}(\widetilde{X}) \\
		\widetilde{Y}& \widetilde{X} \\};
	\path[-stealth]
	(m-2-1) edge node [left] {$\mu_{c, y}$} (m-1-1)
	(m-1-1) edge  node [above] {$\widetilde{f}_{\ast}$} (m-1-2)
	(m-1-2) edge node [right] {$\bary$} (m-2-2)
	(m-2-1) edge  node [above] {$\widetilde{F}_c$} (m-2-2);
	\end{tikzpicture}
\end{center}

 By construction, the maps $\widetilde{F}_{c}$ have the following properties:
\begin{itemize}
	\item [(i)] $\widetilde F_c$ is $\lambda$-equivariant, {\it i.e.} $\widetilde F_c(h.\tilde y)=\lambda(h).\widetilde F_c(\tilde y)$; hence, every $\widetilde F_c$ induces a quotient map $F_c:Y\rightarrow X$ such that $(F_c)_*=\lambda$;
	\item[(ii)] the induced maps $F_c$ are homotopic to the map $f$ for every $c$;
\end{itemize}

In our context, we shall consider the following family of immersions in the space $\mathcal M(Y,g)$: 
$$y\mapsto d\mu_{c, y}(y'):= e^{-c \tilde d(y,y')}dv_{\tilde g}(y').$$
where $dv_{\tilde g}$ is the Riemannian measure of the metric $\tilde g$ on $\widetilde Y$ and $c>\ent(g)$.
All these measures are finite by the integral characterization of the entropy in the compact case, see Section \ref{section_integral_characterization}.\\

Now we shall give the formal definition of the barycentre of a measure $\mu$ on a complete, simply connected Riemannian manifold having non-positive curvature. For the proof of the good definition of the barycenter, and of the useful properties provided in Lemmata \ref{lemma_equivariance_f_epsilon}, \ref{lemma_implicit_equations_f_epsilon}, \ref{lemma_derivatives_of_the_implicit_function} we refer to \cite{sambusetti1999minimal}.

\begin{defn}
	Let $(\widetilde{X},\tilde g_0)$ be a complete simply connected Riemannian manifold with non-positive curvature, and $\mu \in \mathcal{M}(\widetilde X,\tilde g_0)$ a positive, finite Borel measure on it. If there exists at least a point $x_0 \in \widetilde{X}$ such that the integral $\int_{\widetilde{X}}  \tilde d_0(x_0,x')^2 d \mu(x')$ converges, then the so-called \textit{Leibniz function}: $$B_{\mu} (x)= \int_{\widetilde{X}}  \tilde d_0(x',x)^2 d\mu(x')$$
	is well defined for any $x \in \widetilde{X}$ ---namely, the integral converges for any $x \in \widetilde{X}$---, is positive and is $C^{\infty}$.
\end{defn}

As a consequence of the strict convexity of the Riemannian distance function $d_0$ on $\widetilde X$, it follows that the Leibniz function associated to any $\mu\in\mathcal M(\widetilde{X},\tilde g_0)$ is a strictly convex function. Moreover, for any fixed $x_0$ we have that $|B_\mu(x)-B_\mu(x_0)|\rightarrow\infty$ as $\tilde d_0(x,x_0)\rightarrow\infty$. We refer to \cite{sambusetti1999minimal} for these facts. As a consequence, it follows that $B_\mu$ admits a unique critical point. It makes thus sense to give the following definition.

\begin{defn}
	Let $(\widetilde X,\tilde g_0)$ and $\mu$ as above.
	We define the {\it barycentre of $\mu$} as the unique critical point of the associated Leibniz function. In particular, in this setting the critical point is a minimum. It will be denoted $\bary [\mu]$. The barycenter is \textit{equivariant}, \textit{i.e.},
for every $\sigma \in \Isom(\widetilde{X},\tilde g_0)$ we have $\bary[\sigma_{\ast} \mu]= \sigma . \bary[\mu]$.
\end{defn}

We shall now list some properties verified by the maps $\widetilde{F}_c$, $F_{c}$. These properties are classical (see for example \cite{sambusetti1999minimal}), nevertheless we shall provide proofs for the sake of completeness.

\begin{lem}[Equivariance of $\widetilde{F}_{c}$]  \label{lemma_equivariance_f_epsilon}
	Using the notation introduced above, the maps $\widetilde{F}_{c}: \widetilde{Y} \longrightarrow \widetilde{X}$ are $\lambda$-equivariant, and thus induce maps between the quotients $F_c:Y\longrightarrow X$.
\end{lem}

\begin{proof}
	 We recall that $\lambda=f_*: \pi_1(Y) \longrightarrow \pi_1(X)$ is the homomorphism induced by $F:Y\rightarrow X$. We need to prove that the $\widetilde{F}_{c}$ are $\lambda$-equivariant, \textit{i.e.}, $\bary[\tilde f_{\ast}\mu_{c, h . y }]= \lambda(h) . \bary [\tilde F _{\ast}\mu_{c,y}]$ for any $h \in \pi_1(Y)$. This will follow from:
	 $$B_{\mu_{c, h.y}}(\lambda(h).x)=B_{\mu_{c, y}}(x)$$
Indeed, the following equalities hold:
\small
$$  B_{\mu_{c, h.y}}(\lambda(h).x)=\int_{\widetilde X} \tilde d_0\left(z,\lambda(h).x\right)^2d(\tilde f_*\mu_{c, h.y})(z)=$$$$=
\int_{\widetilde{Y}} \tilde d_0\left(\tilde f(y'),\lambda(h). x\right)^2\mu_{c, h .y}(y')  =$$
$$= \int_{\widetilde{Y}} \tilde d_0\left(\tilde f(y'),\lambda(h).x\right)^2 e^{- c \tilde d(h.y, y')} dv_{\tilde g}(y')= $$
 $$=\int_{\widetilde{Y}} \tilde d_0\left(\lambda(h).\tilde f(y''),\lambda(h).x\right)^2 e^{- c \tilde d(h.y, h.y'')} dv_{\tilde g}(h.y'') =$$$$= \int_{\widetilde{Y}} \tilde d_0\left(\tilde f(y''),x\right)^2 e^{-c \tilde d(y,y'')}dv_{\tilde g}(y'')=B_{\mu_{c, y}}(x)$$
 \normalsize
  Hence, the minimum of $B_{\mu_{c,h.y}}$ is equal to the image via $\lambda(h)$ of the minimum of $B_{\mu_{c, y}}$.
 \end{proof}

\begin{lem}
	The maps $F_c:Y\longrightarrow X$ are homotopic to $f$.
\end{lem}

\begin{proof}
Since $(X,g_0)$ is non-positively curved we know that $X$ is an aspherical space. Hence, 
the set of homotopy classes of maps $[Y,X]$ from $Y$ to $X$ is in bijection with $\mathrm{Hom}(\pi_1(Y),\pi_1(X))/\mathrm{Inn}(\pi_1(X))$. As a consequence of this bijection, and of the fact that $\lambda=f_*=(F_c)_*:\pi_1(Y)\rightarrow\pi_1(X)$ we deduce that $F_c$ is homotopic to $f$.
\end{proof}

By taking derivatives of the Leibniz function it is easily seen that for every complete, simply connected, non-positively curved $(\widetilde X,\tilde g_0)$ and every positive, finite Borel measure $\mu$ on $(\widetilde{X},\tilde g_0)$ the barycenter $\bary[\mu]$ is uniquely determined by the implicit equation: $\forall v\in T_{\bary[\mu]} \widetilde{X}$,
\small
\begin{equation} \label{eq_implicit_bary}
(dB_{\mu})_{\bary[\mu]}(v)= 2 \int_{\widetilde{X}} (\tilde \rho_0)_{x'}(\bary [\mu]) \cdot (d (\tilde\rho_0)_{x'})_{\bary[\mu]}(v) d\mu(x')=0,
\end{equation}
\normalsize
where $(\tilde\rho_0)_{x'}(\cdot)=\tilde d_0(x',\cdot)$.
As a consequence, we have the following implicit equations for the maps $\widetilde F_	c$ (see \cite{sambusetti1999minimal} or \cite{sambusettiphd} for a detailed proof of the following Lemmata \ref{lemma_implicit_equations_f_epsilon}, \ref{lemma_derivatives_of_the_implicit_function} and \ref{lemma_computation_jacobian}):

\begin{lem}[Implicit equations for $\widetilde F_{c}$] \label{lemma_implicit_equations_f_epsilon}
	The functions $\widetilde F_{c}$ are univocally determined by the vectorial implicit equations $G_{c}(\widetilde F_{c}(y),y)=0$, where $G_{c}=(G_{c}^i): \widetilde{X} \times \widetilde{Y} \longrightarrow \mathbb R^n$ is defined as
	$$ G_{c}^i(x,y)= \frac{1}{2} \int_{\widetilde Y} d_x [( \tilde \rho_{0})_{\tilde f(y')}]^2(E_i)\, e ^{-c \tilde d(y,y')} dv_{\tilde g}(y'),$$
	where $\{E_i\}_{i=1}^n$ is a $\tilde g_0$-orthonormal basis for the tangent space $T_{x} \widetilde{X} $, and $( \tilde \rho_{0})_{\tilde f(y')}(\cdot)= \tilde d_{0}(\tilde f(y'), \cdot)$.
\end{lem}

\begin{proof}
	Since $(\widetilde X, \tilde g_0) $ is non-positively curved, the Riemannian universal cover $\widetilde X$ is diffeomorphic to $\mathbb R^n$; in particular, it has trivial tangent bundle.
	By the definition of the barycentre, $\widetilde F_c(y)= \bary [ f_{\ast}\mu_{c,y}]$ is the unique point of $\widetilde X$ satisfying, for every $i=1, \dots,n$:
	\small
	$$ 0= \left( dB_{\mu_{c, y}}\right)_{\widetilde F_c(y)}(E_i)= \left(  d \int_{\widetilde X} (\tilde \rho_{0})_{x'}(x)^2 \widetilde f_{\ast} \mu_{c, y}(x')  \right)_{\widetilde F_c(y)}(E_i)=$$
	$$= \int_{\widetilde Y} \left(  d [(\tilde \rho_0)_{\widetilde f(z)}]^2  \right)_{\widetilde F_c(y)} (E_i) e^{- \left( c \tilde d(y,z)\right)} dv_{\tilde g}(z)= 2 G_c^i(\widetilde F_c(y),y)$$
	\normalsize
	where we were allowed to pass the derivative through the integral sign thanks to Lebesgue Dominated Convergence Theorem. Indeed, if $x(t) \rightarrow x$, then the following estimates hold:
	\small
	$$ \left \lvert  \frac{\tilde d_0(x(t), \widetilde f(z))^2- \tilde d_0(x, \widetilde f(z))^2}{	t} \right \rvert  e^{-c \tilde d(y,z)} \le$$
	$$ 
\left \lvert \frac{ \tilde d_0(x(t), \widetilde f (z))- \tilde d_0(x, \widetilde f(z))}{t} \right \rvert 
\left( \tilde d_0(x(t), \widetilde f(z) )+ \tilde d _0(x, \tilde f(z)) \right) e^{-c \tilde d(y,z)} \le 
$$
	$$ \le (2+ \delta) \tilde d_0(x(t), \widetilde f(z)) e^{-c\tilde d(y,z)}$$
	\normalsize
	and the last function is integrable.
	\end{proof}
	
	A straightforward computation gives us the following:

\begin{lem} [Derivatives of the implicit function $G_{c}$] \label{lemma_derivatives_of_the_implicit_function}
	Using the notation above, the map  $G_{c}: \widetilde{X} \times \widetilde{Y} \rightarrow \mathbb{R}^n$ is $C^1$ and for $x= \widetilde F_{c}(y)$, the differential $(dG_{c})_{(x,y)}|_{T_x\widetilde{X}}$ is non-singular. Hence, for $x=\widetilde F_c(y)$ the map $\widetilde F_{c}: \widetilde{Y} \rightarrow \widetilde{X}$ is $C^1$ and satisfies:
	\begin{equation} \label{eq_derivative_implicit_function_1}
		 d_y\widetilde F_{c}= -\left( dG_{c}\right)_{(x,y)} |^{-1}_{T_x \widetilde{X} }\circ \left(dG_{c}\right)_{(x,y)}|_{T_y \widetilde{Y}}
		 \end{equation}
	Furthermore, if $u \in T_y \widetilde{Y}$ and $v \in T_x \widetilde{X}$ we have:
	\small
	\begin{equation} \label{eq_derivative_implicit_function_2}
	(dG_{c}^i)_{(x,y)}(u)= - c \int_{\widetilde{Y}}( \tilde \rho_{0})_{\tilde f(y')}(x) \cdot d_x ( \tilde \rho_{0})_{\tilde f(y')}(E_i) \cdot d_y\tilde  \rho_{y'}(u) e^{-c\, \tilde d(y,y')}dv_{\tilde g}(y')
\end{equation}
	\begin{equation} \label{eq_derivative_implicit_function_3}
 (dG_{c}^i)_{(x,y)}(v)= \frac{1}{2} \int_{\widetilde{Y}} Dd_x[(\tilde \rho_{0})_{\tilde f(y')}]^2(v,E_i) e^{-c\, \tilde d(y,y')}dv_{\tilde g}(y')
 \end{equation}
\normalsize	
	\end{lem}

\begin{proof}
We compute the derivative of the function $G_c$ with respect to the variable $x$. Fix $x$ and $y$, and take $v \in T_{x}\widetilde X$ such that $\lVert v \rVert=1$.
Let $x(t)$ be the geodesic such that $x(0)=x$ and $x'(0)=v$. Then:
\small
$$\left(dG_{c}^i \right)_{(x,y)}(v)= \lim_{t \to 0} \int _{\widetilde Y} \left(   \frac{\left(d[(\tilde \rho_0)_{\tilde f (y')}]^2\right)_{x(t)}(E_i)-  \left(  d[(\tilde \rho_0)_{\tilde f (y')}]^2\right)_x (E_i)  }{2t}  \right)  e^{-c \tilde d (y,y')} dv_{\tilde g}(y')$$
\normalsize
Since $X$ is compact, the curvature of $(\widetilde X, \tilde g_0)$ is bounded from below by a constant $k^2$, and Rauch Comparison Theorem implies that the Hessian of the function $(\tilde \rho_{0})_z(x)= \tilde d_0(x,z) $ on $\widetilde X$ is bounded; namely:
\small
$$ \lVert Dd(\tilde \rho_{0})_z \rVert \le \frac{k}{tgh(k \tilde \rho_0)} \le k.$$
\normalsize
Since:
\small$$ \frac{1}{2} Dd \tilde \rho_0^2= \tilde \rho_0 Dd \tilde \rho_0+ d \tilde \rho_0 \otimes d \tilde \rho_0,$$
\normalsize
 it follows:
 \small
$$ \frac{1}{2}\left  \lVert  \left(  Dd[(\tilde \rho_{0})_{\tilde f(y')}]^2\right)_x \right \rVert \le k \tilde d_0(\widetilde f(y'),x)+1.$$
\normalsize
Furthermore, we observe that:
\small
\begin{equation} \label{eq_auxiliary_derivative_G_c}
\left(  \frac{\partial}{\partial v} \frac{\partial}{\partial E_i} [(\tilde\rho_{0})_{\tilde f(y')}]^2  \right)(x)= \left(  Dd [(\tilde\rho_{0})_{\tilde f(y')}]^2 \right)_x (v, E_i)+ \left(d [(\tilde\rho_{0})_{\tilde f(y')}]^2\right)_x (D_v E_i).
\end{equation}
\normalsize
Observe that Lagrange Theorem implies that, for $t$ sufficiently small:
\small
$$\frac{1}{2t} \left \lvert \left( d \tilde \rho_{0, \widetilde f(y')}^2 \right) _{x(t)}(E_i)- \left(   d \tilde \rho_{0, \widetilde f(y')}^2 \right)_x(E_i)  \right \rvert \le$$
$$\le   \frac{\tilde d_0(x(t),x)}{2t} \sup_t  \left 	\lvert   \left(  D d[(\tilde\rho_{0})_{\tilde f(y')}]^2\right)  _x (v, E_i) + \left(  d [(\tilde\rho_{0})_{\tilde f(y')}]^2 \right)_x(D_vE_i) \right \rvert \le$$
$$ \le 1+ \left( k + \lVert D_v E_i \rVert + \delta \right) \tilde d_0(\widetilde f(y'),x)$$
\normalsize
Since the function $1+ \left( k + \lVert D_v E_i \rVert + \delta \right) \tilde d_0(\widetilde f(y'),x)$ is in $L^1(\widetilde Y, dv_{\tilde g})$, then by Lebesgue Dominated Convergence Theorem we can pass the limit through the integral sign.
Moreover, since at $x= \widetilde F_c(y)= \bary [\widetilde f_{\ast} \mu_{c, y}]$ Equality \ref{eq_implicit_bary} implies:
\small
$$ \int_{\widetilde Y} \left(  d [(\tilde\rho_{0})_{\tilde f(y')}]^2\right)_x \left(D_v E_i \right) \mu_{c,y}(y')= \int_{\widetilde X} \left( d[(\tilde\rho_{0})_{\tilde x'}]^2\right)_x \left(D_v E_i\right) \widetilde f_{\ast} \mu_{c, y}(x')=0 \,\,\,\ \forall \, i=1, \dots, n $$
\normalsize
This implies that the integral of the second summand in \ref{eq_auxiliary_derivative_G_c} vanishes. Therefore Equation \ref{eq_derivative_implicit_function_3} holds, and $\partial_v G_c$ is continuous.

Now we compute the derivative of $G_c$ with respect to the variable $y$.
Let $x \in \widetilde X$, $y \in \widetilde Y$ and $u \in T_y \widetilde Y$ whit $\lVert U \rVert  \le 1$, and let $y(t)$ be the geodesic such that $y(0)=y$ and $y'(0)= u$.
Then, the following estimate holds:
\small
$$ \frac{e^{-c\tilde d(y,y(t))}-1}{t} \le \frac{e^{c \tilde d (y,y')-c \tilde d (y(t),y')}-1}{t} \le \frac{e^{c \tilde d (y, y(t))}-1}{t}$$
\normalsize
Therefore, for $t$ sufficiently small:
\small
$$ \left \lvert   \left(  \frac{e^{-c\tilde d (y(t),y')}-e^{-c \tilde d (y,y')}}{t}      \right)  \left(  d [(\tilde\rho_{0})_{\tilde f(y')}]^2  \right) _x (E_i) \right \rvert \le 4 (c+ \delta )e^{-c \tilde d (y,y')} \cdot \tilde d_0( \widetilde f(y'),x)$$
\normalsize
which for for $c > \ent(g)$ is a function in $L^1(\widetilde Y, dv_{\tilde g})$ .
Hence, Lebesgue Dominated Convergence Theorem implies:
\small 
$$ \left(dG_c^i \right)_{x,y}(u)= - c \int _{\widetilde Y} \left(d \tilde \rho_{y'}\right)_y(u) \cdot \left(d (\tilde\rho_{0})_{\tilde f(y')}\right)_x(E_i) \cdot \tilde d_0(\tilde f(y'), x) \mu_{c, y}(y')$$
\normalsize
Since $(\tilde X, \tilde g_0)$ is non-positively curved, the function $(\tilde \rho_{0})_{\tilde f(y')}$ is strictly convex. Therefore, $Dd [(\tilde \rho_{0})_{\tilde f(y')}]^2(v,v) >0$ for every $v \in T_x \widetilde X$, and $(\left(dG_c\right)_{x,y})|_{T_{x} \widetilde X}$
is non-singular. Thus, the Implicit Function Theorem implies that $\widetilde F_c$ is $C^1$, and Equation \ref{eq_derivative_implicit_function_1} holds.
\end{proof}

\begin{lem} [Computation of the Jacobian of $F_{\varepsilon}$ ] \label{lemma_computation_jacobian}
	Let us fix $y \in \widetilde{Y}$,  put $x= \widetilde F_c(y)$ and denote by $\nu_{c,y}(y')= d_0(\tilde f(y),\tilde f(y'))\cdot \mu_{c,y}(y') $. Furthermore, let us introduce the positive quadratic forms $k_{c,y}^X$, $h_{c,y}^X$ and $h_{c,y}^Y$ defined on $T_x \widetilde{X}$, $T_x \widetilde{X}$ and $T_y \widetilde{Y}$ respectively. \\
	
	\noindent $k_{c, y}^X (v,v)= \frac{1}{\nu_{c,y}(\widetilde{Y})} \int_{\widetilde{Y}} \frac{1}{2} Dd_x [(\tilde \rho_0)_{\tilde f(y')}]^2(v,v) d\mu_{c,y}(y'), \;\;\qquad \forall v \in T_x \widetilde{X}$\\
	
	\noindent $h_{c,y}^X(v,v)= \frac{1}{\nu_{c,y}(\widetilde{Y})} \int_{\widetilde{Y}}d_x [(\tilde \rho_{0})_{ \tilde f(y')}]^2(v,v) d\nu_{c,y}(y'),\qquad \qquad\forall v \in T_x \widetilde{X}$\\
	
	\noindent $  h_{c, y}^Y(u,u)=  \frac{1}{\nu_{c,y}(\widetilde{Y})} \int_{\widetilde{Y}} d_y [(\tilde\rho)_{y'}]^2(u,u) d\nu_{c,y}(y'), \;\;\;\;\;\qquad \qquad \forall u \in T_y \widetilde{Y}$\\		
	and let us denote by $K_{c,y}^X$, $H_{c,y}^X$ and $H_{c,y}^Y$ the corresponding endomorphisms of the tangent spaces $T_x \widetilde{X}$, $T_x{\widetilde{X}}$ and $T_y \widetilde{Y}$ respectively, and associated to the scalar products $(\tilde g_0)_x$ on $  T_x \widetilde{X}$ and $(\tilde g)_y$ on $T_y \widetilde{Y}$.
	Then, 
	\begin{equation}
		 \tr_{\tilde g_{0}}H^X_{c, y}=1, \qquad \qquad  \tr_{\tilde g}(H^Y_{c,y})=1
	 	\end{equation}
	and the eigenvalues of $H^X_{c, y}$ and $H^Y_{c,y}$ are pinched between $0$ and $1$.\\
	Moreover, for any $u \in T_y \widetilde{Y}$ and for any $v \in T_x \widetilde{X}$ the following inequality between the quadratic forms holds:
	\begin{equation} \label{eq_jacobian_1}
	k_{c,y}^X(d_y\widetilde F_{c}(u),v) \le c\cdot h^X_{c,y}(v,v) \cdot h^{Y}_{c,y}(u,u),
	\end{equation}
	and 
	\begin{equation} \label{eq_estimate_jacobian_k_h}
	|\jac_y (\widetilde F_{c})| \le \frac{c^n}{n^{\frac{n}{2}}} \cdot \frac{(\det H^X_{c,y})^{\frac{1}{2}}}{\det K^X_{c,y}}
	\end{equation}
\end{lem}

\begin{proof}
	Since $\lVert \nabla (\tilde \rho_{0})_{x'} \rVert=1$, we have $\sum_{i=1}^n d (\tilde \rho_{0})_{\tilde f(y')}(E_i)^2=1$, hence $\tr_{\tilde g_0} H^X_{c,y}=1$. Analogously, we obtain that $\tr_{\tilde g} H^Y_{c,y}=1$.
	Since the two quadratic forms $h^X_{c, y}$ and $h^Y_{c,y}$ are positive, it follows that all the eigenvalues of $H^X_{y,c}$ and $H^Y_{y,c}$ are pinched between $0$ and $1$. Then, by Lemma \ref{lemma_derivatives_of_the_implicit_function} we have that, for every $v \in T_x \widetilde X$:
	\small 
	$$k^X_{c, y} (v, E_i)= \frac{\left(dG_c^i\right)_{(x,y)}(v) }{ \nu_{c,y}(\widetilde Y)}$$
	\normalsize
	Therefore, from Equations \ref{eq_derivative_implicit_function_1} and \ref{eq_derivative_implicit_function_2} we obtain:
	\small
	$$k^X_{c, y} \left( \left(d \widetilde F_c \right)_y(u),v \right) = \frac{c}{\nu_{c,y}(\widetilde Y)} \int _{\widetilde Y} \left(d \tilde \rho _{y'}\right)_y(u) \left( d (\tilde \rho_{0})_{\tilde f(y')}\right)_x (v) \nu_{c, y}(y').$$
	\normalsize
	for every $u \in T_y \widetilde Y$, $v \in T_x \widetilde X$. Then, by the Cauchy-Schwartz Formula we get Inequality \ref{eq_jacobian_1}. In order to prove Inequality \ref{eq_estimate_jacobian_k_h}, we start observing that it is trivially verified if $\jac_y \widetilde F_c=0$, and that $K^X_{c, y}$ is invertible.
	Hence, without loss of generality, we can assume that $y$ is such that $(d \tilde F_c)_y$ is invertible.
	Now we shall derive Inequality \ref{eq_estimate_jacobian_k_h} from Inequality \ref{eq_jacobian_1}, as follows.
	Let $F: (U,g) \rightarrow (V, g_0)$ be a linear isomorphism between Euclidean spaces of the same dimension $n$, and consider positive bilinear forms $h^Y$ on $U$ and $h^X, k^X$ on $V$, represented by the endomorphisms $H^Y$, $H^X$, $K^X$ with respect to $g$ and $g_0$, respectively, and such that:
	\small
	$$ \left \lvert k^X (F(u),v) \right \rvert \le C \cdot h^Y(u,u)^{1/2} \cdot h^X(v,v)^{1/2}$$
	\normalsize
	Since $K^X$ is an isomorphism, once chosen a $g_0$-orthonormal basis $\{v_i\}$ of $V$ diagonalizing $h^X$, we can consider the basis $\{u_i\}$ obtained by orthonormalizing with respect to $g$ the basis $\{  \left(K^X \circ F\right)^{-1}(v_i)\}$. Hence, $K^X \circ F(u_i)$ can be written as a linear combination of kind $\sum_{j \le i} a_{ij}v_j$. Therefore, the matrix representing $K^X \circ F$ with respect to the bases $\{u_i\}$ and $\{v_i\}$ is upper triangular. Thus:
	\small 
	$$ \left \vert  \det K^X\right \rvert \cdot \left \lvert \det F\right \rvert= \prod_{i=1}^n \left \lvert g_0 \left(K^X \circ F(u_i),v_i \right) \right \rvert= \prod_{i=1}^n \left \lvert k^X(F(u_i), v_i)\right \rvert \le$$
	$$ \le C^n \left( \prod_{i=1}^n h^Y(u_i, u_i)^{1/2} \right) \cdot \left(  \prod_{i=1}^n h^X(v_i,v_i)\right)^{1/2} \le C^n \left(  \frac{1}{n} \tr H^Y\right)^{n/2} \cdot \left(\det H^X\right)^{1/2}$$
	\normalsize
\end{proof}

\section{Barycenter method and large hyperbolic balls}

In the classical application of the barycenter method, when the target space $(X,g_0)$ is a hyperbolic manifold or satisfy $\sigma_0\le -1$, a sharp estimate of the Jacobian is obtained by (\ref{eq_estimate_jacobian_k_h}) applying Rauch's Theorem in order to estimate the quadratic form $K_{c,y}^X$ in terms of $\mbox{Id}-H^X_{c,y}$, and then using the following algebraic lemma (see \cite{besson1995entropies}, Appendix D for a detailed proof).

\begin{lem}[Algebraic] \label{lemma_algebraic}
	Let $V=\mathbb{E}^n$ with $n \ge 3$, and denote by $\mathcal{K}$ the  set of all the symmetric endomorphisms $H$ of $V$ whose eigenvalues $h_i$ satisfy $0 \le h_i \le 1$, and $\tr(H)= \sum_{i=1}^{n}h_i=1$.
	Then, for any $ H \in \mathcal{K}$ the following inequality holds:
	$$ \varphi(H)=\frac{(\det H)^{1/2}}{\det(Id-H)} \le \frac{n^{n/2}}{(n-2)^n}.$$
	Furthermore, the value $\frac{n^{n/2}}{(n-2)^n}$ is realized at the unique point of absolute maximum $H =\frac{1}{n}\cdot\id$ for $\varphi$.\\
	
\end{lem}

However, unlike the standard setting where the Hessian of the distance function of the target space can be explicitly computed (the curvature being constantly equal to $-1$), when $\sigma\le 0$ the classical estimate of the Hessian provided by Rauch's Theorem does not provide a useful upper bound for $\varphi(H)$. Nevertheless, in our case the $C^2$ Riemannian metrics $g_\delta$ are hyperbolic on large portions of our Riemannian manifolds, so we shall recover an almost optimal estimate for the Hessian of the distance function,which in combination with (\ref{eq_estimate_jacobian_k_h}) will be sufficient for our purposes.

\subsection{Jacobi fields along hyperbolic directions} In this paragraph and in the next one we shall explain how to obtain almost optimal estimates from below for the Hessian of the distance function on a non-positively curved Riemannian manifold at points which possess sufficiently large hyperbolic neighbourhoods. In particular, this {\it sufficiently large} radius will be made explicit in terms of the closeness to the optimal estimate. We shall first prove the following general:
\begin{prop} \label{proposition_jacobi_fields}
	Let $(M,g)$ be a Riemannian manifold, and assume that $c:[0, \ell] \longrightarrow M$ is the unique minimizing geodesic from $z:=c(0)$ to $c(\ell)$, and that the sectional curvatures $\sigma $ of $M$ along $c$ satisfies:
	\begin{equation*}\label{hypotheses_0}
	\left\lbrace \begin{array}{l} \sigma \le 0 \ \ \ \ \text{ at points }\, c(t) \,\text{ such that }\ \  0 \le t < \ell - R \ ,\\
	\sigma = -1 \ \ \text{ at points }\, c(t) \,\text{ such that }\ \  \ell - R \le t \le \ell  \  . \end{array}\right. 
	\end{equation*}
	Then, for every $t\ge \ell-R$ and every $w \in T_{c(t)} M$  such that $g(w,\dot{c}(t))=0$ we have
	$$ (Dd \rho_{z})_{c(t)}(w,w) \ge (1-2 e^{-2 (t+R-\ell)}) ||w||^2$$
\end{prop}

\begin{proof}
	The unicity of $c$ and the assumption on the sectional curvatures along $c$ tell us that $c$ does not cross the cut-locus of $z=c(0)$.
	Let us denote by $z'=c(\ell-R)$. For every $v\in T_zM$ orthogonal to $\dot c(0)$ define $J_v$ as the Jacobi field along $c$ satisfying the initial conditions $J_v(0)=0$, $J_v'(0)=v$. For every $s\in[0,R]$ let
	$Y_{\ell-R}^0(s)$ and $Y_{\ell-R}^1(s)$ be the parallel transports of $J_v(\ell-R)$, $J_v'(\ell-R)$ respectively from the point $c(\ell-R)$ to the point $c(\ell-R+s)$. Since $\sigma\equiv -1$ along $c$ for any $t\ge \ell-R$ we see that for $s\in[0,R]$ 
	$$ J_v(\ell-R+s)=\cosh(s)\cdot Y_{\ell-R}^0(s)+\sinh(s)\cdot Y_{\ell-R}^1(s)$$
	$$J_v'(\ell-R+s)=\sinh(s)\cdot Y_{\ell-R}^0(s)+\cosh(s)\cdot Y_{\ell-R}^1(s)$$
	Let $\mathrm{II}_t$ denote the second fundamental form at $c(t)$ of the geodesic sphere of radius $t\ge \ell-R$ centered at $z$. Recall that:
	\small
	$$(Dd\rho_{z})_{c(t)}\left(\frac{J_v(t)}{\| J_v(t)\|}, \frac{J_v(t)}{\|J_v(t)\|}\right)=\mathrm{II}_t\left(\frac{J_v(t)}{\| J_v(t)\|}, \frac{J_v(t)}{\|J_v(t)\|}\right)=\frac{g(J_v'(t), J_v(t))}{g(J_v(t), J_v(t))}$$
	\normalsize
	Now we shall prove that the last term is greater or equal to $1-\alpha(t)$ where $\alpha(t)=2\cdot e^{-2(t+R-\ell)}$. Indeed, using the abbreviations:
	\small
	 $$N=\| J_v(\ell-R)\|_{g},\; N'=\| J_v'(\ell-R)\|_{g},\;  L=g(J_v(\ell-R), J_v'(\ell-R))$$ 
	 \normalsize
	 we get:
	\small
	$$\frac{g(J_v' (t ) \,, J_v (t ))}{g(J_v (t ) \,, J_v (t ))} = 
	\frac{\cosh t\cdot \sinh t \left( N^2 + (N')^2\right) + \left( \cosh^2 t + 
		\sinh^2 t\right) L}{ \cosh^2 t \cdot N^2 +
		\sinh^2 t \cdot (N')^2 + 2\, \cosh t\cdot \sinh t \cdot L}=1-A(t)$$
	\normalsize
	where 
	\small
	$$A(t)=\frac{e^{-t}\cosh t\,N^2-e^{-t}\sinh t\, (N')^2-e^{-2t}\,L}{\cosh^2t\, N^2+\sinh^2t\, (N')^2+\sinh 2t\, L}\le$$
	$$\le \frac{e^{-t}\cosh t\, N^2-e^{-t}\sinh t\, (N')^2- e^{-2t} L}{\cosh^2 t\, N^2+\sinh^2t\,(N')^2+\sinh 2t\, L}\le$$$$\le e^{-2t}\cdot\frac{e^{t}\cosh t\, N^2-e^{t}\sinh t\, (N')^2}{\cosh^2 t\, N^2+\sinh^2t\,(N')^2}\le 2\cdot e^{-2t}$$
	\normalsize
	which proves the desired estimate.
\end{proof}

\begin{rmk}[Euristhic interpretation of Proposition \ref{proposition_jacobi_fields}]
	The euristhic idea behind the Proposition is the following: let $c:[0,+\infty)\rightarrow \mathbb H^n$ be a geodesic ray and let $J_{v,w}$ be the Jacobi field along $c$ with initial conditions $J_{v,w}(0)=v$,  $J_{v,w}'(0)=w$. Then $$J_{v,w}(t)=\cosh(t) \,P_{c(t)}(v)+\sinh(t) P_{c(t)}(w)$$ and $$J_{v,w}'(t)=\sinh(t) \,P_{c(t)}(v)+\cosh(t) P_{c(t)}(w),$$ where we denoted by $P_{c(t)}$ the parallel transport along the geodesic $c$ from $T_{c(0)} \mathbb H^n$ to $T_{c(t)}\mathbb H^n$. Taking the difference between $J_{v,w}(t)$ and $J_{v,w}'(t)$ it is easily seen that:
	$$J_{v,w}(t)-J_{v,w}'(t)=e^{-t}\cdot P_{c(t)}(v-w)$$
	Hence, travelling in the hyperbolic space straighten up the Jacobi fields, no matter which initial condition you put. In other words, the Jacobi field $J_{v,w}$ tends to become closer and closer to its first covariant derivative $J_{v,w}'$, thus providing a lower bound for $\frac{g(J_{v,w}(t), J_{v,w}(t))}{g(J_{v,w}(t), J_{v,w}(t))}$ which becomes asymptotically optimal.
	Hence, Proposition \ref{proposition_jacobi_fields} is saying that  given a geodesic which travels in a non-positive direction, we do not care about whatever the Jacobi field $J_v$ does before entering in a hyperbolic direction, but, as far as you travel inside a hyperbolic direction for a sufficiently large amount of time, the Jacobi field $J_v$ will become closer and  closer to its first covariant derivative.
\end{rmk}

\subsection{Large hyperbolic balls and estimate of $\mathrm{Hess}(\rho_z)$
	  }
  
Using Proposition \ref{proposition_jacobi_fields} it is easy to show the following result:

\begin{lem} \label{lem_estimate_Hessian_distance}
	Let $(\widetilde{X}, \tilde g_0)$ be a complete, non-positively curved, simply connected Riemannian manifold. Let us fix $\varepsilon >0$ and $R\ge R_{\varepsilon}=\ln\left(\sqrt{\frac{2}{\varepsilon}}\right)$. Assume that the geodesic ball centred at $x\in\widetilde X$ of radius $R$ is isometric to $B_{\mathbb H^n}(o,R)\subset \mathbb H^n$. Then for every $z\in\widetilde X \smallsetminus B_{\widetilde X}(x,R)$ and for every $v\in T_x\widetilde X$ we have that:
	\small
	\begin{equation}
	\label{eq_estimate_Hessian_distance}
	(1- \varepsilon) (\|v\|^2_{ \tilde g_0}- (d_x( \tilde \rho_0)_z(v))^2) \le (Dd ( \tilde \rho_{0})_z)_x(v,v).
	\end{equation}
	\normalsize
\end{lem}

\begin{proof}
	Since, $z$ is outside $B_{\tilde g_0}(x,R)$  the (unique) geodesic from $z$ to $x$ satisfies the hypotheses of Proposition \ref{proposition_jacobi_fields}. Estimate (\ref{eq_estimate_Hessian_distance}) then follows from the choice made for $R$.
\end{proof}

From Lemma \ref{lem_estimate_Hessian_distance} we get an inequality relating $k_{c, y }^X$ and $ \tilde g_{0}- h_{c, y}^X$,  which will give the required estimate for the jacobian: the following, immediate corollary says that the (almost optimal) inequality between these quadratic forms holds, at least at the points having a sufficiently large hyperbolic neighbourhood:

\begin{cor} [Estimate of $k_{c,y}^X$] \label{cor_estimate_quadratic_forms}
	Let $c >\ent(g)$ and $x= \tilde f(y) \in \widetilde X$. For every $\varepsilon>0$ there exists $R_{\varepsilon}>0$ such that if $\tilde g_0$ is hyperbolic when restricted to the ball $B(x, R_\varepsilon)$, then  for every $z \in \widetilde{X}$ and $v \in T_x\widetilde{X}$, the following inequality holds:
	$$ k_{c, y}^X(v,v) \ge (1- \varepsilon)( \tilde g_{0}(v,v)- h_{c,y}^X(v,v) ).$$
	\end{cor}
	\begin{proof} [Proof of  Corollary \ref{cor_estimate_quadratic_forms}]
		Recall that the probability measure $\nu_{c,y}$ involved in the integral equation defining $h_{c,y}^X$ is obtained from $\mu_{c,y}$ by multiplying for $ ( \tilde \rho_{0})_{\tilde f(y)}(\tilde f(\cdot))$ and normalizing. Furthermore, the following relations hold:
		$$ \frac{1}{2} Dd  \tilde \rho_0^2= \tilde \rho_0Dd \tilde \rho_0 + d \tilde \rho_0 \otimes d \tilde \rho_0 \ge \tilde \rho_0 Dd \rho_0.$$
We distinguish two cases. 
If $x \in \widetilde X \smallsetminus B_{\widetilde X}(x, R)$, then the statement follows by multiplying relation $(\ref{eq_estimate_Hessian_distance})$ by $ (\tilde \rho_{0})_{\tilde f(y)}(\tilde f(y'))$, and integrating with respect to the measure $\mu_{c, y}(y')=e^{-c \tilde d(y,y')}dv_{\tilde g}(y')$.
Otherwise, if $z \in B_{\widetilde X}(x, R)$ the computation of the Hessian of the distance function gives:
$$ \frac{1}{2}Dd\tilde\rho_0^2=  \tilde \rho_0 \cdot \frac{1}{\tanh(\tilde \rho_0)}\cdot \left(g_0- d\tilde \rho_0 \otimes d \tilde \rho_0 \right) \ge $$
$$ \ge  \tilde \rho_0 \cdot \left(g_0-d\tilde \rho_0 \otimes d \tilde \rho_0 \right) \ge (1- \varepsilon)  \tilde \rho_0 \cdot \left(g_0-d\tilde \rho_0 \otimes d \tilde \rho_0 \right)$$
		
	\end{proof}

\medspace

\begin{Souto}
		
	\subsection{An approximated version of Rauch's Comparison Theorem}
Now let us prove the  following  proposition, which gives the desired estimate on the Hessian of the potential. We stress here that this result will be crucial in the proof of our main theorem, as it implies an inequality between the quadratic forms defined in the previous section, which entails itself the required estimate for the jacobians of the natural maps.  

\begin{prop} \label{approximated_Rauch_irreducible}
	For any $\varepsilon >0$ there exists $R_{\varepsilon}>0$  such that, for any $ x  \in \widetilde{X}$ such that  the metric $g_\varepsilon$ is hyperbolic if restricted to the ball $B(x, R_\varepsilon)$, for any $z \in \widetilde{X}$ and $v \in T_x\widetilde{X}$, the following inequality holds:
\small
	\begin{equation} \label{inequality_Rauch_irreducible}
	\!\!\!\!\!\!\!\!\!\!\!\!\!\! \! \!\!\!\!\!\!\! \! \!(1- \varepsilon) \rho^{g_\varepsilon}_{x'}(x)\left(||v||^2_{\varepsilon}- d_x \rho^{g_\varepsilon}_{x'}(v) \otimes d_x \rho^{g_\varepsilon}_{x'}(v) \right) \le \rho^{g_{\varepsilon}}_{x'}(x) \cdot D_x d_x\rho^{g_\varepsilon}_{x'}(v,v) \le D_x d_x(\rho^{g_\varepsilon}_{x'})^2(v,v)
	\end{equation}
	\normalsize

\end{prop}
	\begin{rmk} \label {rem_second_fundamental_form_0}
	Let $M$ be a smooth $3$-manifold, and let $f: M \longrightarrow \mathbb{R}$ be a $C^2$ real function such that $||  \nabla f||=1$, and for any $c \in \mathbb{R}$ let us denote by $H_c$ the $c$-level hypersurface, \textit{i.e.}, $H_c=\{ x \in M | f(x)=c\}$. We shall denote by $\Rmnum 2^{H_c}_x$ the quadratic form on the tangent plane given by the \textit{second fundamental form} of the hypersurface $H_c$. Then $\Rmnum{2}^{H_c}_x(.,.)= Hess(f)_x(.,.) $;  in other words, 
\end{rmk}

\begin{proof}
	\textcolor{red}{We underline here that we shall apply the proposition with $x= F_{\varepsilon(y)}$, which under suitable assumptions on $y$ falls \lq \lq deeply '' in a hyperbolic ball. }\\
	Let $x$ be any point satisfying the hypotheses, and let us assume that the metric $g_\varepsilon$ is hyperbolic in the ball $B(x, R+1)$, for a suitable choice of $R >0$. Since any metric $g_{\varepsilon} $ is non-positively curved, for any $z \in \widetilde{X}$ there exists a unique geodesic $\gamma$ such that $\gamma(0)=x$ and $\gamma (l)=z$. 
	
	Let us consider the hypersurface:
	
	$$ N_{R}= exp \, \{w \in T_{\gamma(R)} \widetilde{X} \, | \, \left\langle w, \gamma'(R) \right\rangle =0, \, ||w|| \le 1 \}.$$
	
Let us consider the following sets, both containing $x$:\\
	\noindent $C_1= \{ x \, | d^{g_\varepsilon}(x, N_R) =R \}$\\
	\noindent  $C_2= \{ x \, | d^{g_\varepsilon}(x,z)= l\}$ \\

We are interested in the mutual position of the $C_i$'s in a neighbourhood  of their common point $x$. \\


\begin{lem} \label{lemma_position_C_i}
	Let $y \in C_1$ be any point. Then, $d^{g_{\varepsilon}}(y, z) \ge l$. 
	In other words, the sets $C_1 $ and $C_2$ appears, in a neighbourhood of $x$, as shown in Figure \ref{figure_Rauch_approximated_irreducible}.
\end{lem}

\noindent \textit{Proof of Lemma \ref{lemma_position_C_i}}
First of all, we shall need the following statement:\\
\begin{claim}
There exists a sufficiently small neighbourhood $U(x)$ of $x$ such that any geodesic radius joining $z$ with a point $x' \in U(x) \cap S(z,l)$ intersects the hypersurface $N_R$.
\end{claim}

\noindent \textit{Proof of the claim}
Since the curvature is non-positive, the distance function is convex. 
Let $x'$ be a point on $S(z,l)$, other than $x$, and denote by $\gamma_{x'}$ the geodesic ray from $z$. Then, going along $\gamma_{x'}$ from $z$ to $x'$ the distance from the geodesic $\gamma$ increases. Roughly speaking, we obtain a \lq \lq cone '' of geodesics starting from $z$ and spanning an arc on $S(z, l')$.
\textcolor{blue}{Dimostrare: la geodetica $\alpha$ è ortogonale a $N_R$, considerare il \lq \lq cono '' di geodetiche da $z $ e usare il fatto che la curvatura è negativa.}\\ \\
Let choose $x' \in U(x) \cap S(z,l)$, and denote by $\gamma_{x'}$ the geodesic ray joining $z$ and $x'$ and by $y'$ the intersection of $\gamma{x'}$ with $N_R$. The point $y'$ separates $\gamma_{x'}$ in two geodesic segments: we shall call $\alpha_{x'}$ the one having $x'$ as an end, and $\beta_x$ the other one. Since $\gamma $ is orthogonal to $N_R$, the point $\gamma(R)$ realizes the distance between $z$ and $N_R$. Hence $d(z, N_R)=h$, and $\ell(\beta_x) \ge h$. Thus, since $\gamma_x= \alpha_x \ast \beta_x$ and $\ell(\gamma_x)=l$, we obtain that $\ell(\alpha_x) \le R$.
In particular $d(x', N_R) \le R$,which is the required inequality.





Now we shall deduce from the  containments an inequality between the Hessians of the two distance functions (from the hypersurface $N_R$ and from the point $z$), defining $C_1$ and $ C_2$ respectively.
To this purpose, we recall ---without proving it--- the following proposition.
\begin{prop}
	Let $i: N \longrightarrow M$ be an isometric immersion, $F: M \longrightarrow \mathbb{R}$ a $C^{\infty}$ function and $f= F_{|N}$ its restriction to $N$. Then, for any $ x \in N$ and $u,v \in T_xN$ we have:
	
	$$ (Hess^N f)(u,v)= (Hess^M F)(u,v)+ \left\langle  \nabla^M F, \Rmnum{2} (u,v) \right\rangle ,$$
	where $\Rmnum{2}$ denotes the second fundamental form.
\end{prop}
Now we are ready to prove the inequalities in \ref{inequality_Rauch_irreducible}. We shall focus only on one of them: the proof of the other one is exactly the same.\\\\
\noindent $ \rho^{g_{\varepsilon}}_{x'}(x) \cdot D_x d_x\rho^{g_\varepsilon}_{x'}(v,v) \le D_x d_x(\rho^{g_\varepsilon}_{x'})^2(v,v)$.\\
We shall compare the curvature between the two distance spheres $C_2=S(z_o, l')$ and $C_1=S(\gamma(R), R$. Denote by $\nu$ the unitary internal normal at $x$, which is common to both the hypersurfaces. Fix $ u \in T_x C_1$, and let $\alpha_u(s) $ be the geodesic on $C_1$ such that $c_u'(0)=u$. Let us consider the function $\rho^{g_\varepsilon}_{z_0}(s)=\rho^{g_\varepsilon}_{z_0}(\alpha_u(s))$, that is the restriction of $\rho^{g_\varepsilon}_{z_0}$ to the geodesic $\alpha_u$. Then, applying the last proposition with $N= \alpha_u$ and $f(s)= \rho^{g_\varepsilon}_{z_0}(s)$ we get:

$$(\rho^{g_\varepsilon}_{z_0})''= (Hess^{\widetilde{X}}\rho^{g_\varepsilon}_{z_0})(u,u)- \left\langle \nabla \rho^{g_\varepsilon}_{z_0}, \Rmnum{2}_{\gamma(R)}(u,u)\right\rangle  .$$

Furthermore, we recall that the Hessian of $\rho^{g_\varepsilon}_{z_0}$ gives the second fundamental form of $S(z_0, l')=C_2$, namely $$Hess(\rho^{g_\varepsilon}_{z_0})_x= \Rmnum{2}^{S(z_0,l')}_x .$$
	Hence we get:
	
	$$(\rho^{g_\varepsilon}_{z_0})''(0)=\Rmnum{2}^{S(z_0,l')}_{z_0}-  \Rmnum{2}^{S(z_0,l')}_{\gamma(R)}.$$
	Moreover, the Claim implies that $\rho^{g_\varepsilon}_{z_0}(s)$ has a maximum at $s=0$ (as $x$ is a maximum of the restriction of $\rho^{g_\varepsilon}_{z_0}$ to $\alpha_u$). Hence we get:
	
	$$\big < D_v d \rho^{g_\varepsilon}_{z_0},v\big > \le \big < D_v d \rho^{g_\varepsilon}_{ \gamma(R)},v \big > . $$
	 
Exactly in the same way, it can be proved that 

$$ \big < D_v d (d^{g_\varepsilon}(., N_R)),v\big > \le \big < D_v d \rho^{g_\varepsilon}_{z_0},v \big >.$$\\
Hence we get the following chain of inequalities:
\begin{equation} \label{Rauch_irreducible_intermediate_1}
\big < D_v d (d^{g_\varepsilon}(., N_R)),v\big > \le \big < D_v d \rho^{g_\varepsilon}_{z_0},v \big > \le  \big < D_v d \rho^{g_\varepsilon}_{ \gamma(R)},v \big >.
\end{equation}\\

By assumption, tha ball  $B(x, R+1)$ is hyperbolic, \textit{i.e.}, there exists an isometric embedding $i: (B(x,R+1), g_{\varepsilon}) \hookrightarrow (\mathbb{H}^n, g_{hyp})$ mapping $N_R$ in a $(n-1)$-ball contained in $\mathbb{H}^{n-1} \subset \mathbb{H}^n$ and $\gamma(R) \in N_R$ in $0 \in \mathbb{H}^n$. Then, there exists $R_{\varepsilon}>0$ such that for any $z \in \mathbb{H}^n$ such that $d_{\mathbb{H}^n}(z, \mathbb{H}^{n-1}) >R_\varepsilon$, for any $w \in T_z \mathbb{H}^n$:

$$\left\langle \nabla _w^{\mathbb{H}^n} \nabla ^{\mathbb{H}^n} d_{\mathbb{H}^n}(., \mathbb{H}^{n-1}), w \right\rangle  ) \ge (1- \varepsilon) \left( ||w||_{\mathbb{H}^n}^2- \left\langle \nabla^{\mathbb{H}^n} d_{\mathbb{H}^n}(., \mathbb{H}^{n-1}),w\right\rangle _{\mathbb{H}^n}\right) 
 $$

Hence, since $i$ is an isometric embedding, for any $R+1 \ge R_{\varepsilon}$ and $v \in T_x \widetilde{X}$ we get:
\begin{equation} \label{Rauch_irreducible_intermediate_2}
 \left\langle  \nabla_v \nabla \rho^{g_\varepsilon}(.,N_R,v) \right\rangle \ge (1- \varepsilon) \left( ||v||^2_{\varepsilon}- d_x\rho^{g_\varepsilon}(v) \otimes d_x\rho^{g_\varepsilon}(v)  \right)  
\end{equation}
It is straightforward to verify the following inequality:

$$ \frac{1}{2} Dd(\rho^{g_\varepsilon})^2= \rho^{g_\varepsilon}Dd\rho^{g_\varepsilon}+ d\rho^{g_\varepsilon} \otimes d\rho^{g_\varepsilon} \ge \rho^{g_\varepsilon}Dd\rho^{g_\varepsilon}.$$

Thus, putting together the estimates above, we get:

\small
$$
\left( 1- \varepsilon \right) \left( \rho^{g_\varepsilon}  ||v||_{\varepsilon}^2- d\rho^{g_\varepsilon}(v) \otimes  d \rho^{g_\varepsilon}(v) \right) 
 \overset{\ref{Rauch_irreducible_intermediate_2}}{\le}  \rho^{g_\varepsilon}Dd(d^{g_\varepsilon})(., N_R)(v,v) \overset{\ref{Rauch_irreducible_intermediate_1}}{\le} 
\rho^{g_\varepsilon} Dd \rho^{g_\varepsilon}_{z}\le  \frac{1}{2} Dd (\rho^{g_\varepsilon}_{z})^2$$
\normalsize
The last chain of inequalities gives, in particular, the required estimate.
	
\end{proof}

\end{Souto}
\subsection{The estimate of the jacobian}

We are now ready to prove the next proposition. Even if we will be interested only in the setting of $3$-manifolds, we shall keep writing the dependence on the dimension $n$, to avoid cumbersome numerical constants:

\begin{prop} \label{proposition_gestimate_jacobian_irreducible}
	With the notation introduced in this chapter,
	for every $\varepsilon >0$ there exists $R_{\varepsilon}>0$ such that, if the ball centred at $\widetilde F_{c}(y)$ of radius $R_\varepsilon$ in $(\widetilde X, \tilde g_0)$ is hyperbolic, the following estimate for the jacobian holds:
	$$ \lvert \jac_y \widetilde F_{c}\rvert \le \frac{1}{\left( 1- \varepsilon\right) ^n} \cdot \left( \frac{c}{n-1} \right)^n. $$
\end{prop}

\begin{proof}
	We have already proved in Lemma \ref{lemma_computation_jacobian} that:
	$$ \lvert  \jac_y \widetilde F_{c}\rvert \le \frac{c^n}{n^{\frac{n}{2}}} \cdot \frac{\det (H^X_{c,y})^{\frac{1}{2}} }{\lvert \det (K^X_{c,y})\rvert}$$
	Furthermore, it follows from Proposition \ref{proposition_jacobi_fields} that, if $\widetilde F_{c}(y)$ admits an hyperbolic ball $B(\widetilde F_{c}(y), R_{\varepsilon})$ then, for any $v \in T_{\widetilde F_{c}(y)}\widetilde X$:
	\small
	$$(1- \varepsilon ) (\tilde \rho_{0})_{\tilde f(y)} \tilde g_{0}(v,v)- d_{\widetilde F_c(y)}(\tilde \rho_0)_{\tilde f(y)}(v) \otimes  d_{\widetilde F_c(y)} (\tilde \rho_0)_{\tilde f(y)}(v) \le$$$$\le  (\tilde\rho_{0})_{\tilde f(y)} (Dd(\tilde\rho_{0})_{\tilde f(y)})_{\widetilde F_c(y)}\le (Dd ((\tilde \rho_{0})_{\tilde f(y)})^2)_{\widetilde F_c(y)}$$
	\normalsize
 Recalling the definition of the quadratic forms $H^X_{c, y}$ and $K^X_{c, y}$, the last inequality is equivalent to:

$$\label{inequality_quadratic_forms}
\left( 1- \varepsilon\right) (\cdot \Id- H^X_{c,y}) \le K^X_{c,y}.  
$$
Thus we get:
\small
$$
 \lvert \jac_y \widetilde F_{c}\rvert \le \frac{c^n}{n^{\frac{n}{2}}} \cdot \frac{ \left(\det H^X_{c,y}\right) ^{\frac{1}{2}}}{(1- \varepsilon)^n\lvert \det (\Id- H^X_{c,y}) \rvert }
   \le  \frac{c^n}{\cancel{n^{\frac{n}{2}}}}\cdot \frac{\cancel{n^{\frac{n}{2}}}}{(n-1)^n}= \frac{1}{(1- \varepsilon)^n}\cdot \left( \frac{c}{n-1}\right)^n ,
$$
\normalsize
where the last inequality follows applying Lemma \ref{lemma_algebraic}, once we recall that the endomorphism $H_{c,y}^X$ is symmetric, has trace equal to $1$ and eigenvalues pinched between $0$ and $1$ (see Lemma \ref{lemma_computation_jacobian}).
\end{proof}

\section{Proof of Theorem \ref{thm_lower_estimate_irreducible}}
\label{section_conclusion_lower_estimate_irrducible}
Let $Y$ be an orientable, closed, irreducible $3$-manifold with non-trivial $JSJ$ decomposition and at least one hyperbolic $JSJ$ component. For the application we have in mind, we shall fix the Riemannian manifold $(Y,g)$ and we shall consider for every sufficiently small $\delta$ as target space $(X,g_0)$ the Riemannian manifold $(Y,g_\delta)$, endowed with the $C^2$ Riemannian metrics which we found in Chapter \ref{section_a_conjectural_minimizing_sequence_irreducible}. The starting map $f$ will be the identity map. Notation will be modified coherently with these choices. Any possible confusion arising from the fact that we are considering Riemannian metrics on the same manifold should be avoided by keeping in mind the general discussion of the previous section.\\
 Recall that we denoted by $X_1,..., X_k$ the hyperbolic $JSJ$ components of $Y$. By the argument used to prove  Theorem \ref{thm_construction_sequence_on_irreducible_Y} and the statement of Theorem \ref{thm_metrics_hyperbolic_pieces} we know that for every $i=1,...,k$ and any $\delta>0$ there exists an open subset $V_\delta^i$ of $(X_i, g_{\delta}|_{X_i})\subseteq(Y,g_\delta)$ such that:
\begin{enumerate}
	\item $(V_\delta^i, g_\delta|_{V_\delta^i})$ is isometric to the complement of a horoball neighbourhood $\mathcal U_\delta^i$ of the cusps of $(int(X_i), hyp_i)$;
	\item $\mathcal U_\delta^i\subseteq \mathcal U_{\delta'}^i$ for every $\delta<\delta'$;
	\item $\Vol(V_{\delta}^i, g_\delta|_{V_\delta^i})/\Vol(X_i, g_\delta|_{X_i})\rightarrow 1$ as $\delta\rightarrow0$.
\end{enumerate}
For every $\varepsilon>0$, $\delta>0$ and $i=1,...,k$ denote by $V_{\delta,\varepsilon}^i\subseteq V_{\delta}^i$ the subsets of those points $p\in V_{\delta}^i$  such that the restriction of $g_\delta$  to $B_{g_\delta}(x,R_\varepsilon)$ is a smooth Riemannian metric of constant curvature $-1$. By construction $\Vol(V_{\varepsilon,\delta}^i)=\Vol(int(X_i)\smallsetminus N_{R_\varepsilon}(\mathcal U_\delta^i), hyp_i)$ where we denoted by $N_r(\cdot)$ the tubular neighbourhood of radius $r$. It is straightforward to check that  properties (1)---(3) of the $V_\delta^i$'s are satisfied by the $V_{\delta,\varepsilon}^i$'s with $\mathcal U_{\delta,\varepsilon}^i=N_{R_\varepsilon}(\mathcal U_\delta^i)$. Moreover, the $V_{\delta,\varepsilon}^i$ satisfy the following additional property:
\begin{itemize}
	\item [(4)] $(N_{R_\varepsilon}(V_{\delta,\varepsilon}^i), g_\delta)$ has curvature constantly equal to $-1$;
\end{itemize}
Now, applying Corollary \ref{cor_estimate_quadratic_forms} to the universal cover, and from the equivariance of the $\widetilde F_c$'s we know that:
$$|\jac_y F_c|\le\frac{1}{ (1-\varepsilon)^3}\left(\frac{c}{2}\right)^3$$
for every $y\in Y$ such that $F_c(y)\in V_{\delta, \varepsilon}^i$. Now observe that, using the coaerea formula:
$$\int_{F_c^{-1}(V_{\delta,\varepsilon}^i)}|\jac_y F_c|dv_g(y)=\int_{V_{\delta,\varepsilon}^i}(\# F_c^{-1}(y))dv_g(y)\ge$$$$\ge \int_{V_{\delta,\varepsilon}^i}1\cdot dv_g(y)=\Vol(V_{\delta,\varepsilon}^i, g_\delta)$$
Now use the estimate for $|\jac_yF_c|$ for $y\in F_c^{-1}(V_{\delta,\varepsilon}^i)$:
$$\Vol(V_{\delta,\varepsilon}^i, g_\delta)\le \int_{F_c^{-1}(V_{\delta,\varepsilon}^i)} |\jac_y F_c|dv_g(y)\le \frac{1
}{(1-\varepsilon)^3}\cdot\left(\frac{c}{2}\right)^3\cdot\Vol(F_c^{-1}(V_{\delta,\varepsilon}^i))$$
Taking the sum over $i=1,...,k$  and observing that $F_{c}^{-1}(V_{\delta,\varepsilon}^i)\cap F_c^{-1}(V_{\delta,\varepsilon}^j)$ has zero measure  for $i\neq j$ we get that:
$$\sum_{i=1}^k\Vol(V_{\delta,\varepsilon}^i, g_\delta)\le\frac{1}{(1-\varepsilon)^3}\cdot\left(\frac{c}{2}\right)^3\cdot\Vol(Y,g).$$
Now for $\delta\rightarrow 0$, $\varepsilon\rightarrow 0$ and $c\rightarrow\ent(g)$ we get:
$$\EntVol(Y,g)\ge 2 \cdot \left(\sum_{i=1}^k\Vol(int(X_i), hyp_i)\right)^{1/3}.$$
Since $g$ is any Riemannian metric on $g$ we conclude that:
$$\minent(Y)=\inf_{g} \EntVol(Y,g)\ge 2 \cdot \left( \sum_{i=1}^k\Vol(int(X_i), hyp_i) \right)^{1/3}.$$
\qed

\chapter [The lower bound: reducible $3$-manifolds]{The lower bound via the barycentre method:\\ reducible $3$-manifolds} \label{section_a_lower_bound_reducible}

 In this Chapter, we shall prove the optimal lower estimate of the minimal entropy in the  case of a reducible $3$-manifold $Y$. To this purpose, we shall generalize the notion of barycentre to a $\mbox{CAT}(0)$ space. A fundamental role in the estimate of the Volume Entropy will be played by the metrics that we have defined in Chapter \ref{section_a_conjectural_minimizing_sequence_reducible} on each irreducible piece, and the characteristics of the particular family of (singular)   metric spaces we are going to use as target spaces for the barycenter method.
 \begin{thm}\label{thm_lower_estimate_reducible}
 	Let $Y$ be a closed, orientable, reducible $3$-manifold, and denote by $Y= Y_1 \# \dots \# Y_m $ a minimal decomposition in prime summands. For every $i=1, \dots, m$ denote by $X_i^1, \dots, X_i^{n_i}$ the $JSJ$ hyperbolic components of $Y_i$. Let $g$ be any Riemannian metric on $Y$. Then:
 	\small
 	$$\VolEnt(Y,g) \ge 2 \left(\sum_{i=1}^m \sum_{j=1}^{n_i} \Vol(int(X_i^j), hyp_i^j)\right)^{1/3}.$$
 	\normalsize
 		As a consequence, 
 		\small
 		$$ \minent(Y) \ge 2 \left( \sum_{i=1}^m \sum_{j=1}^{n_i} \Vol(int(X_i^j), hyp_i^j) \right)^{1/3}.$$
 		\normalsize
 \end{thm}

\section{The family of the target spaces $\{(X,d_\delta)\}$}

Up to reordering the prime decomposition $Y=Y_1\#\cdots \# Y_m$, we may assume that for $i=1, \dots, k$ each $Y_i$ has at least a hyperbolic piece in its $JSJ$ decomposition, while $Y_i$ is purely graph (or Seifert fibred) for $i>k$. We can assume that the connected sum is realized by performing the connected sum of every $Y_i$ to the same copy $Y_0$ of $S^3$.  This can be done by \lq\lq sliding" the gluing spheres along $Y$ via an ambient isotopy. Let $\bar Y_i$, for $i=1,..,k$ be the complement of the ball $B_i^3$ removed from $Y_i$ to realize the connected sum. Now consider the following space: for any $i=1,...,k$ fix a point $y_i\in Y_i$ and let $$X=\bigvee_{y_1=\cdots=y_k} Y_i$$
be the space obtained by shrinking $Y\smallsetminus \bigcup_{i=1}^k int(\bar Y_i)$ to a single point $x_0$. We shall denote by $f:Y\rightarrow X$ the quotient map, see Figure \ref{fig:quotient:space}.\\
\begin{figure}
	\centering
	\includegraphics[width=0.7\linewidth]{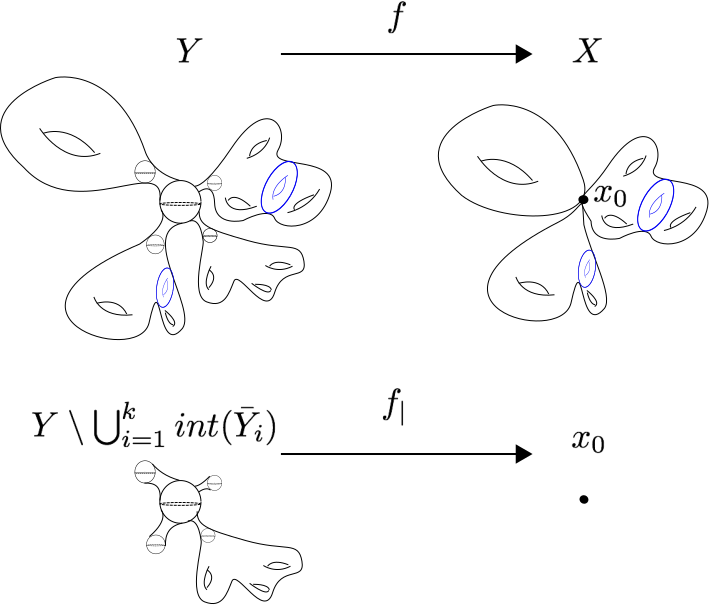}
	\caption[The target spaces]{The map $f:Y \rightarrow X$, where $X= \bigvee_{y_1=\cdots=y_k} Y_i$, collapsing $Y\smallsetminus \bigcup_{i=1}^k int(\bar Y_i)$ to the point $x_0$. }
	\label{fig:quotient:space}
\end{figure}

By construction, $X$ can be presented as $X=\sqcup_{i=1}^k Y_i/\sim$ where $y\sim z$ if and only if $y=y_i$ and $z=y_j$ for some choice of $1\le i, j \le k$. Then $x_0=[y_i]_{\sim}$ (for any choice of $i=1,...,k$). Let $q:\sqcup_{i=1}^k Y_i\rightarrow X$ be the quotient map. Let $d_{\delta}$ be the unique length metric such that $q|_{Y_i}: (Y_i, g_\delta^i)\rightarrow (X, d_\delta)$ is an isometry.  Since the metric space $(X,d_\delta)$ is semi-locally simply connected, it has a universal covering $p:(\widetilde X,\tilde d_\delta)\rightarrow (X, d_\delta)$, which has a Riemannian structure except at those points in $p^{-1}(x_0)=\pi_1(X).\tilde x_0$ where $\tilde x_0$ is a chosen preimage of $x_0$.

\subsection{The universal cover $\widetilde{X}$}$\phantom{a}$
\begin{figure}[h] 
	\centering
	\includegraphics[scale=0.4]{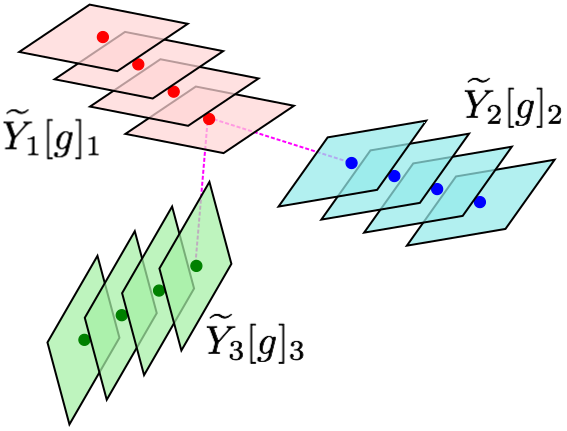}
	\caption[The universal cover $\widetilde X$]{Intersection of $k=3$ leaves in the universal cover.}
	\label{universal_cover}
\end{figure}
\vspace{2mm}

\noindent We now introduce some notation in order to efficiently describe the universal cover of the singular space $X$, and the action by automorphisms of its fundamental group. We remark that, by construction, the fundamental group of $X$ splits as a free product: $\pi_1(X) \simeq \pi_1(Y_1) \ast \dots \ast \pi_1(Y_k)$. Let us denote by $G=\pi_1(X, x_0)$ and by $G_j=(\iota_{j})_*\pi_1(Y_j,y_j)$ where $\iota_j: Y_i\hookrightarrow X$ is the natural inclusion.
Let $\tilde x_0\in\widetilde X$ be a lift of $x_0\in X$. Roughly speaking, $\widetilde{X}$ consists of $G/G_i$ copies of the universal cover $\widetilde{Y}_i $ of each manifold $Y_i$,  glued along the points of the orbit $G.\tilde x_0$.  Let us summarize the structure of $\widetilde X$ in the next proposition:

\begin{prop}[On the structure of $\widetilde X$]
	The universal covering $\widetilde X$ of $X$ is obtained as follows:
	$$\widetilde X=\bigsqcup_{\mathclap{\substack{i=1,...,k \\ [g]_i\in G/G_i}}}\widetilde Y_i[g]_i$$
	where each $\tilde Y_i[g]_i$ is homeomorphic to $\tilde Y_i$.\\ Moreover, $\widetilde Y_i[g]_i\cap\widetilde Y_i[g']_i=\varnothing$ if $[g]_i\neq [g']_i$ and $\widetilde Y_i[g]_i\cap\widetilde Y_j[g']_j\neq \varnothing$ if and only if either $g^{-1}g'\in G_i$ or $(g')^{-1}g\in G_j$ and their intersection is either $\{gh_i.\tilde x_0\}$ for some $h_i\in G_i$ or $\{g'h_j.\tilde x_0\}$ for some $h_j\in G_j$. Finally, notice that for every $g\in G$ the point $g.\tilde x_0$ belongs to $\widetilde Y_1[g]_1, \dots \widetilde Y_k[g]_k$.
\end{prop}

\begin{rmk}
	Let us observe that using the notation introduced above:
	\begin{itemize}
		\item $h. \widetilde Y_i[g]_i=\widetilde Y_i[h\,g]_i$;
		\item for every $i=1,.., k$ we have $p^{-1}(\tilde x_0)\cap\widetilde Y_i[g]_i=\{g\,h_i.\tilde x_0\,|\,h_i\in G_i\}$;
	\end{itemize}
\end{rmk}

\subsection{The distance function $\tilde d_\delta(\cdot, \cdot)$}
First of all remark that, by construction, the distance function $\tilde d_\delta$ on the singular space $\widetilde{X}$ restricts at any leaf $\widetilde{Y_i}[g]$ to the distance $\tilde g_\delta^i$ corresponding to the Riemannian metric on $(\widetilde{Y}_i,\tilde g_\delta^i)$. Moreover,

\begin{lem}\label{lemma_CAT(0)}
	$(\widetilde{X}, d^{\widetilde{X}}_{\varepsilon})$ is a $\mbox{CAT}(0)$ space. 
\end{lem}

The Lemma is a classical result about $CAT(\kappa)$ spaces: indeed, the gluing of two locally $CAT(\kappa)$ spaces along convex subsets is still a locally $CAT(\kappa)$ space (see \cite{bridson2013metric}, Chapter II.11, Theorem 11.1).

	\begin{figure}[h]
		\centering
		\includegraphics[scale=0.35]{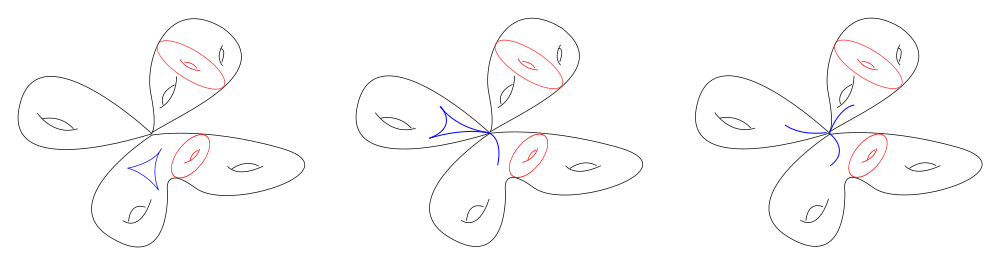}
		\caption[Geodesic triangles in $X$]{The three possible configurations for the vertices of a triangle in $X$.}
		\label{universal_cover}
	\end{figure}

\subsection{Geodesic balls in $(\widetilde{X}, \tilde d_\delta)$}
Let $\tilde x \in  \widetilde{X}$ be any point in the universal cover of $X$, and fix $R \in \mathbb{R}^+$.  It is worth for the sequel to have a good intuition of what the geodesic  ball $B_{\widetilde{X}}(\tilde x, R)$ looks like. Roughly speaking, the ball of radius $R$ centred at any point $\tilde x$ is the union of some metric balls with respect to the metrics $\widetilde{g}_\delta^i$ defined on each leaf $\widetilde{Y_i}[g]_i$.   
More precisely, pick a point $\tilde x \in \widetilde{Y}_i[g]_i$ and fix $R \in \mathbb{R}^+$.  By construction, $\widetilde{Y}_i[g]_i$ is a copy of $\widetilde{Y}_i$, and the restriction of the singular metric $\tilde d_\delta$ coincides with $\tilde{g}^i_{\delta}$.\\
If $p^{-1}(x_0)\cap B_{\widetilde{X}}(\tilde x, R)=\varnothing$, then $B_{\widetilde{X}}(\tilde x, R) \subset \widetilde{Y}_i[g]_i$. Otherwise, the Riemannian ball  $B_{\widetilde Y_i[g]_i}(\tilde x, R)=B_{\widetilde{X}}(\tilde x, R)\cap\widetilde Y_i[g]_i$ ramifies at points $g\,h_{i,\ell} . \tilde x_0$ where $h_{i, \ell}$ is an element of $G_i$ for every $\ell=1,..., \ell_{i,1}$, and for every $j=1,...,k$ the intersection $B_{\widetilde{X}}(\tilde x,R) \cap \widetilde{Y}_j[g\,h_{i,\ell}]_j$ coincides with the Riemannian ball $B_{\widetilde{Y}_j[g\,h_{i,\ell}]_j}(gh_{i,\ell} . x_0, R- \tilde d_\delta(x, g h_{i,\ell}.x_0))$. 
The balls keep ramifying until there are no singular points in its restriction to the leaves. In the end we obtain the following description (see Figure \ref{fig:erikapallemetrichewood} for an evoking image):
$$B_{\widetilde{X}}(\tilde x, R)=B_{\widetilde Y_i[g_i]}(\tilde x, R)\bigcup_{\substack{g'\in G\,:\\ \tilde d_\delta(\tilde x, g'\tilde x_0)<R}}\bigcup_{j=1}^k B_{\widetilde Y_j[g']_j}(g'\tilde x_0, R-\tilde d_\delta(\tilde x,g'\tilde x_0))$$

\begin{figure}
	\centering
	\includegraphics[width=0.5\linewidth]{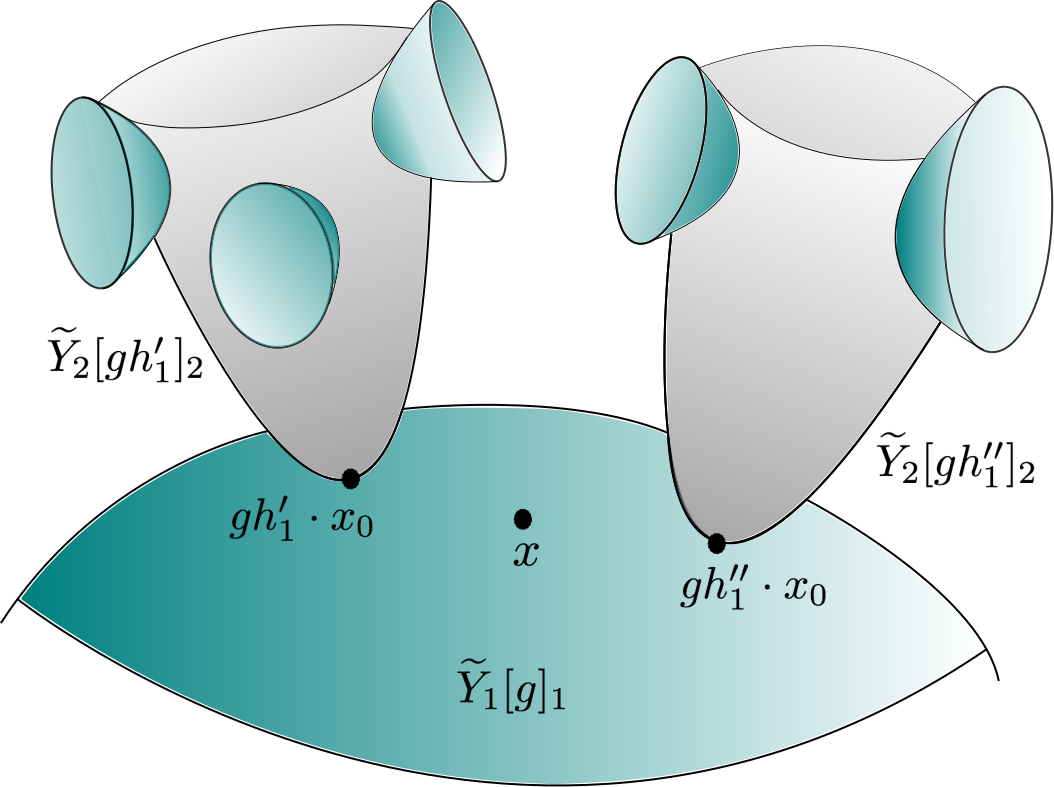}
	\caption[Metric balls in $\widetilde X$]{Metric ball in the universal cover $\widetilde{X}$ of the singular space $X$ associated with a $3$-manifold having two components of hyperbolic type.}
	\label{fig:erikapallemetrichewood}
\end{figure}


%

\section{Barycenter in a CAT(0) space} \label{section_barycenter_CAT(0)}
In this section we shall adapt the notion of barycenter of a measure in the wider context of $CAT(0)$ spaces. The resulting notion will generalize the one in \cite{besson1995entropies} and \cite{sambusetti1999minimal}. To start with, we shall derive an elementary result of Euclidean geometry, which turns out to be fundamental in defining the generalized version of the barycenter. 
First of all, we introduce some notation. We shall denote with capitol letters points in the Euclidean space $\mathbb E^n$ and with tiny letters points in other metric spaces. Moreover, given three points $A,B,C$ we will denote by $[A,B,C]$ the triangle in $\mathbb E^n$ having them as vertices, while given $a,b,c$ we will denote by $[a,b,c]$ the geodesic triangle having them as vertices.
Moreover, in the Euclidean space we will denote by $AB$ the segment with endpoints $A$ and $B$, and with $\overset{\rightarrow}{AB}$ the oriented segment with endpoints $A$ and $B$, parameterized so that $A$ is its source and $B$ is its target. Moreover, we will denote by $[a,b]$ the geodesic segment joining $a$ and $b$. Finally, we denote by $AB^2$ the square length of the segment $AB$.
	\begin{lem} \label{lem_euclidean}
	Let $[A,B,C]$ be a triangle in the $\mathbb E^n$, and take a point $M$ on the side $[B,C]$. Then, the following relation between the lengths of the sides holds:
	$$ AM^2 + BM \cdot CM = AB^2 \cdot \frac{CM}{BC} + AC^2 \cdot \frac{BM}{BC}$$
\end{lem}
\begin{proof}
	Let us parameterize the side $BC$ as the interval $[0,1]$, and choose $t \in [0,1]$ such that $\overset{\rightarrow}{BM}= t \cdot \overset{\rightarrow}{BC}$ and $\overset{\rightarrow}{MC}=(1-t) \cdot \overset{\rightarrow}{BC} $ .
	Let us denote by $\overset{\rightarrow}{AM}^2= \langle \overset{\rightarrow}{AC} , \overset{\rightarrow}{AC}\rangle$. Then we have:
	\begin{equation} \label{eq_euclidean_1}
	\!\!\!\!\!\!\!\!\!\! \overset{\rightarrow}{AM}^2=(\overset{\rightarrow}{AB}+ \overset{\rightarrow}{BM})^2=(\overset{\rightarrow}{AB}+ t  \overset{\rightarrow}{BC})^2= \overset{\rightarrow}{AB}^2+t^2\overset{\rightarrow}{BC}+ 2t \langle \overset{\rightarrow}{AB}, \overset{\rightarrow}{BC} \rangle
	\end{equation}
	Similarly,
	\begin{equation} \label{eq_euclidean_2}
	\!\!\!\!\!\!\!\!\!\! 
	\overset{\rightarrow}{AM}^2=(\overset{\rightarrow}{AC}+ \overset{\rightarrow}{CM})^2=(\overset{\rightarrow}{AC}+ (1-t) \overset{\rightarrow}{BC})^2= \overset{\rightarrow}{AB}^2+(1-t)^2\overset{\rightarrow}{BC}- 2(1-t) \langle \overset{\rightarrow}{AC}, \overset{\rightarrow}{BC} \rangle
	\end{equation}
	Putting together Equations \ref{eq_euclidean_1} and \ref{eq_euclidean_2}, we finally get:
	
	\begin{equation*}
	\overset{\rightarrow}{AM}^2=(1-t) \overset{\rightarrow}{AM}^2 + t \overset{\rightarrow}{AM}^2=
	\end{equation*}
	
	\begin{equation*}
	=(1-t) \overset{\rightarrow}{AB}^2+ t^2 (1-t) \overset{\rightarrow}{BC}^2+ t \overset{\rightarrow}{AC}^2+t(1-t)^2 \overset{\rightarrow}{BC}^2+ 2t (1-y) \langle \overset{\rightarrow}{AB}-\overset{\rightarrow}{AC}, \overset{\rightarrow}{BC} \rangle=
	\end{equation*}
	
	\begin{equation*}
	=(1-t) \overset{\rightarrow}{AB}^2+ t^2 (1-t) \overset{\rightarrow}{BC}^2+ t \overset{\rightarrow}{AC}^2+t(1-t)^2 \overset{\rightarrow}{BC}^2+ 2t (1-y) \langle \overset{\rightarrow}{AB}-\overset{\rightarrow}{AC}, \overset{\rightarrow}{BC} \rangle=
	\end{equation*}
	
	\begin{equation*}
	=(1-t) \overset{\rightarrow}{AB}^2+ t^2 (1-t) \overset{\rightarrow}{BC}^2+ t \overset{\rightarrow}{AC}^2+t(1-t)^2 \overset{\rightarrow}{BC}^2+ 2t (1-y) \langle \overset{\rightarrow}{BC}, \overset{\rightarrow}{BC} \rangle=
	\end{equation*}
	
	Hence, 
	$$ \overset{\rightarrow}{AM}^2= (1-t) \overset{\rightarrow}{AB}^2+ t \overset{\rightarrow}{AC}^2-t(1-t)\overset{\rightarrow}{BC}^2,$$
	and the required equality holds plugging $t=\frac{BM}{BC}$ and $(1-t)=\frac{CM}{BC}$ in the last one.
\end{proof}

\begin{definizioniEteoria}
The previous result of elementary Euclidean geometry easily generalizes to $CAT(0)$ spaces.  Before stating an analogous of Lemma \ref{lem_euclidean} in the context of $CAT(0)$ spaces, we shall introduce some theory.

\begin{lem} [Lemma+ definition: comparison triangles]
	\textcolor{blue}{Definire anche per la curvatura positiva?}\\
	Lat $k$ be a real number, and let $p,q, r$ be three distinct points in the metric space $X$. If  $k>0$, assume furthermore that $d_X(p,q)+d_X(q,r)+d_X(p,r) < D_k$.
	Then, there exist points $\bar{p}, \bar{q}, \bar{r} \in M^2_k$ such that $d_X(p,q)=d(\bar{p},\bar{q})$, $d_X(q, r)= d(\bar{q},\bar{r})$ and $d_X(p,r)=d(\bar{p}, \bar{r})$.	
	If $\Delta(p,q,r)$ is a geodesic triangle with vertices $p,q,r$ then the triangle $\bar{\Delta}(\bar{p}, \bar{q}, \bar{r}) \subseteq M^2_k$ with vertices $p,q,r$ is called a comparison triangle for $\Delta(p,q,r)$.
	
\end{lem}

\end{definizioniEteoria}

\begin{lem}  \label{lemma_comparison_triangle}
	Let $(X,d)$ be a $CAT(0)$ space,
	 let $[a,b,c]$ be any geodesic triangle. Then, for any $m \in [b,c]  $ the following inequality holds:
	\small
	
	$$ d(a,m)^2 + d(m,b) \cdot d(c,m) \le d(a,b)^2 + \frac{d(c,m)}{d(b,c)}+ d(a,c)^2 \cdot \frac{d(b,m)}{d(b,c)}.$$
	\normalsize
\end{lem}

\begin{proof}
	Let us construct the comparison triangle $[A,B,C]$ for $[a,b,c]$ in the Euclidean space. Let $M$ be the point dividing the segment $[B,C]$ in the same proportion as $m$ divides $[b,c]$. By the definition of comparison triangle we have:
	$AB= d(a,b),$ $ AC= d(a,c), BC= d (b,c)$ and by construction $BM= d(b,m)$ and $CM=d(c,m)$. Furthermore, the $CAT(0)$ assumption implies $AM \ge d(a,m)$, and we get:
	\small
	
	\begin{eqnarray} \nonumber
	d(a,m)^2  & \le& AM^2 = AB^2 \cdot \frac{CM}{BC} + AC^2 \cdot \frac{BM}{BC} - BM \cdot CM \\
	\nonumber
	& =& d(a,b)^2 \cdot \frac{d(c,m)}{d(b,c)}+ d(a,c)^2 \cdot \frac{d(b,m)}{d(b,c)}- d(b,m) \cdot d(c,m)
	\end{eqnarray}
	\normalsize

\end{proof}

Now we shall exploit the previous lemma to prove the existence and uniqueness of the barycenter.

\begin{prop} \label{prop_bary_CAT0}
	Let $\mu$ be a positive measure on a $CAT(0)$ space $(X,d)$ such that there exists a point $x_0$ satisfying:
	\begin{equation} \label{eq_auxiliary_bary_CAT0}
	 \int_X (1+ d(x_0,z))^2 d \mu(z) < + \infty 
	 \end{equation}
	and consider the associated Leibniz function
	$$ B_{\mu}(x)= \int _X d^2(x,z) d\mu(z).$$
	This function is continuous on $X$ and $B_{\mu}(x)\rightarrow+\infty$ when $d(x,x_0)\rightarrow+\infty$.
		Moreover, the following conditions hold:
			\begin{itemize}
			\item[i)] the function $B_{\mu}$ has a unique minimum point $\bary[\mu]$, called the \textit{barycenter of $\mu$};
			\item[ii)] for any $x \in X$, we have
			$$ d^2(\bary[\mu],x) \cdot \mu(X) \le B_{\mu}(x)- B_{\mu}(\bary[\mu]).$$
		\end{itemize}
	\end{prop}

\begin{rmk}
	We remark that the hypothesis provided by Inequality \ref{eq_auxiliary_bary_CAT0} simultaneously ensures that $\mu$ is a finite measure on $X$, and the functions $\rho_{x_0}(\cdot)= d (x_0, \cdot)$ and $\rho_{x_0}^2$ are in $L^1(X, \mu)$.
\end{rmk}

\begin{proof} [Proof of Propositon \ref{prop_bary_CAT0}]
Let us show first that the function $B_{\mu}$ is continuous on X.
Let us fix $x_0 \in X$. For $x\in B(x_0,R)$ the following chain of estimates holds:
\small
$$d^2(x,z) \le (d(x,x_0)+ d(x_0,z))^2 \le (R+ d(x_0,z))^2 \le R^2 (1+ d(x_0,z))^2=f_R(z)$$
\normalsize
The function $f_R$ does not depend on $x\in B(x_0,R)$, and is $\mu$-integrable by assumption. Hence for $\{x_k\}_{k\in\mathbb N}\subset B(x_0,R)$ such that $x_k\rightarrow x$ the $\mu$-integrable functions $d^2(x_k, \cdot)$ are converging pointwise to the $\mu$-integrable function $d^2(x_0, \cdot)$ and are smaller than or equal to $f_R$. Thus Lebesgue's Dominated Convergence Theorem implies the continuity of $B_{\mu}$.\\ 

\noindent Now, remark that $B_{\mu}(x) \longrightarrow + \infty$ as $d(x_0,x) \longrightarrow \infty$. Indeed,
recalling the elementary inequality $(a-b)^2 \ge \frac{1}{2}a^2 - b^2$ , we get
\small 
\begin{eqnarray}
\nonumber
B_{\mu}(x) &\ge& \int_X (d(x,x_0)- d(x_0,z))^2 d \mu(z) \ge \frac{1}{2} d(x,x_0)^2 \mu(X)- \int_X d(x_0,z)^2 d\mu(z)\\
\nonumber
& = &\frac{1}{2} d(x,x_0)^2 \mu(X)-B_{\mu}(x_0).
\end{eqnarray}

\normalsize

\noindent We conclude observing that $B_{\mu}(x_0) < \infty$, and by assumption $d(x,x_0) \rightarrow \infty$.\\ 

\noindent \textit{Existence.} The last two properties imply that there exist points $b \in X $ such that $B_{\mu}(b)=\underset{x \in X}{\inf}B_{\mu}(x)$; namely there exists at least a minimum for the function $B_{\mu}$. Now we shall prove that this minimum value is attained at a unique point. \newline

\noindent\textit{Uniqueness.}
\noindent Let $b \in X$ be a point as above, \textit{i.e.}, a point realizing the minimum for the function $B_{\mu}$.
Then Lemma \ref{lemma_comparison_triangle} implies that, for any $x ,z \in X$ and for any $m \in [b,x] $ (where by $[b,x]$ we intend the ---unique--- geodesic segment joining the two points) we have:
\small
$$d(m,z)^2 \le d(b, z)^2 \cdot \frac{d(m,x)}{d(b,x)} + d(x,z)^2 \cdot \frac{d(b,m)}{d(b,x)}- d(m,b) d(m,x).$$
\normalsize
\noindent Integrating the last inequality with respect to the measure $\mu$, we get:

$$ B_{\mu}(b) \le B_{\mu}(m) \le B_{\mu}(b) \cdot \frac{d(m,x)}{d(b,x)} + B_{\mu}(x) \cdot \frac{d(b,m)}{d(b,x)}- d(b,m) d(m,x) \mu(X),$$
where the first inequality holds as $b$ is a minimum for $B_{\mu}$.
Thus we have:

$$ B_{\mu}(b) \cdot \frac{d(b,m)}{d(b,x)} \le B_{\mu}(x) \cdot \frac{d(b,m)}{d(b,x)}- d(b,m) d(m,x) \mu(X).$$

\noindent Multiplying the last relation by $\frac{d(b,x)}{d(m,b)}$ and letting the distance $d(b,m)$ go to zero, we obtain:
$$ B_{\mu}(b) \le B_{\mu}(x) - d^2(b,x) \mu(X),$$
which is clearly equivalent to 
\begin{equation}
\label{eq_unicity_barycentre}
d^2(b,x) \mu(X) \le B_{\mu}(x)-B_{\mu}(b).
\end{equation}
We stress here that Inequality (\ref{eq_unicity_barycentre}) is precisely the point $(ii)$ of the statement. Now assume that $b$ and $b'$ are two points where the function $B_{\mu}$ achieves its minimum. Then Inequality 
(\ref{eq_unicity_barycentre}) implies:
$$ d(b,b')^2 \mu(X) \le B_{\mu}(b')- B_{\mu}(b)=0,$$
whence $d(b,b')=0$ and $b=b'$.
This completes the proof of the uniqueness statement.
\end{proof}

\section{The natural maps}
Let $g$ be any Riemannian metric on $Y$; we shall denote by $d$ the Riemannian distance on  $Y$ and by $\tilde g, \tilde d$ the Riemannian metric and distance  on $\widetilde Y$ respectively. Moreover, for every $i=1, \dots k$ we shall denote  $g^{i}_{\delta}$ the sequence of Riemannian metrics on the irreducible manifold $Y_i$  constructed as in Chapter \ref{section_a_conjectural_minimizing_sequence_irreducible}. Furthermore, we denote by $\rho_{\delta,z}(.)= \tilde d_\delta(z,\,\cdot)$, the distance function from a fixed point $z\in\widetilde X$.
Let us consider the following immersion of $\widetilde Y$ into the space $\mathcal M(Y,g)$ of finite, positive Borel measures on $\widetilde Y$: $$y\mapsto d\mu_{c, y}(y'):= e^{-c\, \tilde d(y,y')}dv_{\tilde g}(y'),$$
for $c>\ent(g)$ and let $$\tilde f: (\widetilde{Y},\tilde g) \longrightarrow (\widetilde{X},\tilde d_\delta)$$ be the lift of the quotient map $f$. Then we compose the latter map with the push-forward via $\tilde f$, thus obtaining the (finite, positive, Borel) measures $\tilde f_*\mu_{c, y}$ in $\mathcal M(\widetilde X,\tilde d)$. Let $B_{c,y}^\delta$ denote the Leibniz function associated with the measure $\tilde f_*\mu_{c,y}$, that is:
\small
$$B_{c,y}^\delta(x)=\int_{\widetilde X}[\rho_{\delta, x}(z)]^2 \, d(\tilde f_{\ast}\mu_{c,y})(z)= \int_{\widetilde Y} (\rho_{\delta, x}(\tilde f(y')))^2 e^{-c \tilde d(y,y')}dv_{\tilde g}(y').$$
\normalsize
Since $(\widetilde X, \tilde d_\delta)$ is a $CAT(0)$ space (Lemma \ref{lemma_CAT(0)}) it follows from Proposition \ref{prop_bary_CAT0} that $B_{c,y}^\delta$ admits a unique minimum point for every $y\in\widetilde Y$. Hence, for every $y\in\widetilde Y$ we define the map
$$\widetilde F_{c,\delta}:(\widetilde Y,\tilde g)\rightarrow (\widetilde X,\tilde d_\delta),\qquad\widetilde F_{c,\delta}(y)=\bary[\tilde f_*\mu_{c,y}]$$
and we shall refer to the maps $\widetilde F_{c,\delta}$ as to the {\it natural maps} on $(\widetilde X, d_{\delta})$.

\section{Properties of the natural maps}

We shall now list the properties of the natural maps $\widetilde F_{c,\delta}$.

\subsection{Equivariance, homotopy type, surjectivity}

To begin with, we shall observe that the natural maps $\widetilde F_{c,\delta}$ are $\lambda$-equivariant, where $\lambda=f_*:\pi_1(Y, y_0)\rightarrow \pi_1(X, x_0)$ is the homomorphism induced by $f$ ---where $y_0\in f^{-1}(x_0)$.

\begin{lem}[Equivariance]\label{lemma_equivariance_sing_space}
	Choose a point $y_0\in f^{-1}(x_0)\subset Y$. The natural maps $\widetilde F_{c,\delta}$ are $\lambda$-equivariant with respect to the surjective homomorphism 
	$\lambda=f_*:\pi_1(Y,y_0)\rightarrow\pi_1(X,x_0)$
\end{lem}

\begin{proof}
	As in the irreducible case we have that $B_{c, h.y}^\delta(\lambda(h). x)=B_{c,y}^\delta(x)$, and thus we obtain the $\lambda$-equivariance of the natural maps.
\end{proof}

Since the maps $\widetilde F_{c,\delta}$ are $\lambda$-equivariant, they induce maps $F_{c,\delta}:Y\rightarrow X$ between the quotients. Now observe that:

\begin{lem}[Homotopy type]\label{lemma_homotopy_sing_space}
	For every $\delta>0$ and every $c>0$ the natural map $F_{c,\delta}$ is homotopic to $f$.
\end{lem}

\begin{proof}
The proof goes as in the irreducible case, once observed that $X$ is an aspherical $CW$-complex.
\end{proof}

Finally, we observe that:

\begin{lem}[Surjectivity]
	For every $\delta>0, c>\ent(g)$ the natural map $F_{c,\delta}$ is surjective.
\end{lem}

\begin{proof}
Observe that $H_3(X,\mathbb Z)\cong\mathbb Z^k=(\iota_1)_*[Y_1]\mathbb Z\oplus\cdots\oplus(\iota_k)_*[Y_k]\,\mathbb Z$ where $\iota_i: Y_i\hookrightarrow X$ is the natural inclusion.  For every $i=1,..,k$ let us denote by $p_i: X\rightarrow Y_i$ the map sending $Y_i$ identically into $Y_i$ and $X\smallsetminus Y_i$ into $y_i$ (the point on $Y_i$ identified to $x_0$). Let us consider $\alpha_X=(\iota_1)_*[Y_1]\oplus\cdots\oplus(\iota_k)_*[Y_k]\in H_3(X,\mathbb Z)$, then $(p_i)_*(\alpha_X)=[p_i(X)]=[Y_i]\in H_3(Y_i,\mathbb Z)$. Moreover, we have that $f_*[Y]=\alpha_X$. On the other hand, for every $i=1,...,k$ the map $p_i\circ F_{c,\delta}$ is homotopic to the map $p_i\circ f$, hence $(p_i\circ F_{c,\delta})_*[Y]=[Y_i]$ ---{\it i.e.} $p_i\circ F_{c,\delta}$ has degree $1$, which implies the surjectivity of the map $p_i\circ F_{c,\delta}$. Since this holds for every $i=1,...,k$ and the space $X$ is obtained by identifying for $i=1, \dots, k$ the points $\{y_i\}$ to $x_0$ we conclude that $F_{c,\delta}$ is surjective.
\end{proof}

\subsection{Regularity}
In this paragraph we shall be concerned with derivatives of the $\widetilde F_{c,\delta}$'s wherever this makes sense. Notice that $\widetilde X$ has a Riemannian manifold structure outside the set $S_{x_0}=p^{-1}(x_0)$. Hence, it makes sense to consider derivatives outside this set. Notice that, since we are dealing with a singular space, we have to face some difficulties concerning the regularity of the distance function. In this section we shall explain how this issue can be circumvented.
We start proving the following lemma.
\begin{lem}[Regularity] \label{lemma_regularity_barycentre_sing}
	The function $B_{c,y}^\delta(x)$ is $C^{2}$ on $\widetilde{X} \setminus S_{x_0}$.
\end{lem}
Since $B_{c,y}^\delta$ is obtained by integrating $(\tilde\rho_{\delta,x})^2$,  Lemma \ref{lemma_regularity_barycentre_sing} is a consequence of the next Lemma, concerning the regularity of the function $\tilde{\rho}_{\delta, z}=\tilde d_\delta(z,\cdot)$, and of Lebesgue's Dominated Convergence Theorem.
\begin{lem}
	Let $z \in \widetilde{X}$ be any point. The function $\tilde \rho_{\delta,z}: \widetilde{X} \longrightarrow \mathbb{R}^+$ is $C^{2}$ on $\widetilde{X}\smallsetminus S_{x_0}.$
\end{lem}
\begin{proof}
	Let $x\in\widetilde{X} \smallsetminus S_{x_0}$. Assume that $z \in \widetilde{Y}_i[g]_i$ and $x \in \widetilde{Y}_j[g']_j$. Since $\widetilde Y_i[g]_i$, $\widetilde Y_j[g']_j$ are convex subsets of $(\widetilde X,\tilde d_{\delta})$ we can define $P_{i,g}$, $P_{j,g'}$ as the projections on these subsets. By construction it is readily seen that $P_{i,g}(x)= (g\cdot h_i). \tilde x_0$  and $P_{j, g'}(z)=(g'\cdot h_j).\tilde x_0$, for suitably chosen $h_i$, $h_j$ (depending only on $\widetilde Y_i[g]_i$, $\widetilde Y_j[g']_j$). Then,
	$$\tilde\rho_{\delta,z}(x)=\tilde d_{\delta}(x,z)=\tilde d_\delta(z, P_{i,g}(x))+\tilde d_{\delta}(P_{i,g}(x), P_{j,g'}(z))+\tilde d_{\delta}(P_{j,g'}(z), x).$$
	Notice that, once $z$ has been fixed, the first summand $\tilde d_\delta(z, P_{i,g}(x))$ and the second summand $\tilde d_{\delta}(P_{i,g}(x), P_{j,g'}(z))$  are constant as long as $x$ varies in $\widetilde Y_j[g']_j$. Moreover, by construction, the restriction of $\tilde d_\delta$ to $\widetilde Y_{i}[g]_i$ is equal to the Riemannian distance function on $(\widetilde Y_i, \tilde g_\delta^i)$ (and similarly the restriction of $\tilde d_{\delta}$ to $\widetilde Y_j[g']_j$). 
	Hence, there exists a constant $K=K(z, g, g')$ such that
	\begin{equation}\label{eq_tigger}
	\tilde\rho_{\delta,z}(x)=K+\tilde d_{\delta}(x, P_{j,g'}(z))
	\end{equation}
	for every point $x\in\widetilde Y_j[g']_j$. By this remark, it is clear that we can take derivatives up to the second order (recall that the Riemannian metrics $\tilde g_{\delta}^i$'s are $C^2$) of $\tilde\rho_{\delta,z}$ at every point $x\in \widetilde X\smallsetminus S_{x_0}$, using the Levi-Civita connection with respect to the metric $\tilde g_\delta^j$ of the appropriate leaf $\widetilde Y_j[g']_j$ to which $x$ belongs.
\end{proof}

Since we arleady know that for every $y\in\widetilde Y$ the minimum of $B_{c,y}^\delta$ is attained at unique point, {\it i.e.} the barycentre $\widetilde F_{c,\delta }(y)$ of the measure $\tilde f_*\mu_{c,y}$,  Lemma \ref{lemma_regularity_barycentre_sing} allows us to give the following characterization of $\widetilde F_{c,\delta}(y)$ when $\widetilde F_{c,\delta}(y)\in\widetilde X\smallsetminus S_{x_0}$.
The proofs of the following Lemmata \ref{lem_implicit_equation_manifold_point_sing}, \ref{lemma_derivatives_of_the_implicit_function_sing} and \ref{lemma_Jacobian_sing} follow exactly the same line as those of Lemmata \ref{lemma_implicit_equations_f_epsilon}, \ref{lemma_derivatives_of_the_implicit_function} and \ref{lemma_computation_jacobian} in Chapter \ref{section_a_lower_bound_irreducible}.

\begin{lem}[The implicit equation at manifold points] \label{lem_implicit_equation_manifold_point_sing}
	For every $\delta>0$, $c>\ent(g)$ and every $y\in\widetilde Y$ such that $x=\widetilde F_{c,\delta}(y)\in \widetilde X\smallsetminus S_{x_0}$, the point $x$ is characterized by the equation:
	\small
	\begin{equation}\label{eq:char_barycenter_at_manifold_points}
	 \int_{\widetilde X}\tilde\rho_{\delta, z}(x)d(\tilde\rho_{\delta, z})_x(v) d(\tilde f_*\mu_{c,y})(z)=0,\qquad\forall v\in T_x\widetilde X
	\end{equation}
	\normalsize
	In other words, the map $\widetilde F_{c,\delta}$ is univocally determined on the complement of the preimage of the singular set $S_{x_0}$ by the vectorial implicit equation $G_{c,\delta}(\widetilde F_{c,\delta}(y),y)=0$, where $G_{c,\delta}=(G_{c,\delta}^i): (\widetilde{X}\smallsetminus S_{x_0}) \times \widetilde{Y} \longrightarrow \mathbb R^n$ is defined as
	\small
	$$ G_{c,\delta}^i(x,y)= \frac{1}{2} \int_{\widetilde Y} d_x [(\tilde\rho_{\delta, \tilde f(y')}]^2(E_i) e ^{-c \tilde d(y,y')} dv_{\tilde g}(y'),$$
	\normalsize
	where $\{E_i\}_{i=1}^n$ is a global $\tilde g_\delta^j$-orthonormal basis for the tangent space $T_{x} \widetilde{X}$ where $x\in\widetilde Y_j[g']_j$ for some $j$, and $\rho_{\delta,\tilde f(y')}(\cdot)= \tilde d_{\delta}(\tilde f(y'), \cdot)$.
\end{lem}

Since $x\in\widetilde X\smallsetminus S_{x_0}$ we are allowed to take derivatives (up to the first order) of the function $G_c(x,y)$. Hence, for those $y\in \widetilde Y$ such that $x=\widetilde F_{c,\delta}(y)\in\widetilde X\smallsetminus S_{x_0}$ we have:

\begin{lem} [Derivatives of the implicit function $G_{c,\delta}$] \label{lemma_derivatives_of_the_implicit_function_sing}
	Using the notation introduced above, the map  $G_{c,\delta}: (\widetilde{X}\smallsetminus S_{x_0})\times \widetilde{Y} \rightarrow \mathbb{R}^n$ is $C^1$ and for $x=\widetilde F_{c,\delta}(y)\in\widetilde X\smallsetminus S_{x_0}$, the differential $(dG_{c,\delta})_{(x,y)}|_{T_x\widetilde{X}}$ is non-singular. Hence, for $x=\widetilde F_{c,\delta}(y)\in\widetilde X\smallsetminus S_{x_0}$ the map $\widetilde F_{c,\delta}: \widetilde{Y} \rightarrow \widetilde{X}$ is $C^1$ and satisfies:
	$$ d_y\widetilde F_{c,\delta}= -\left( dG_{c,\delta}\right)_{(x,y)} |^{-1}_{T_x \widetilde{X} }\circ \left(dG_{c,\delta}\right)_{(x,y)}|_{T_y \widetilde{Y}}.$$
	Furthermore, if $u \in T_y \widetilde{Y}$ and $v \in T_x \widetilde{X}$ we have:
	\small
	$$
	(dG_{c,\delta}^i)_{(x,y)}(u)= - c \int_{\widetilde{Y}}\tilde \rho_{\delta,\tilde f(y')}(x) \cdot d_x ( \tilde \rho_{\delta,\tilde f(y')}(E_i))\cdot d_y \tilde\rho_{y'}(u) e^{-c\,\tilde d(y,y')}dv_{\tilde g}(y')
	$$
	$$ (dG_{c}^i)_{(x,y)}(v)= \frac{1}{2} \int_{\widetilde{Y}} Dd_x[\tilde \rho_{\delta,\tilde f(y')}]^2(v,E_i) e^{-c\, \tilde d(y,y')}dv_{\tilde g}(y')$$
	\normalsize	
\end{lem}

Finally, we have the following estimate of the Jacobian of the map $\widetilde F_{c,\delta}$ at those points $y\in \widetilde Y$ such that $\widetilde F_{c,\delta}(y)\in\widetilde X\smallsetminus S_{x_0}$:

\begin{lem} [Computation of the Jacobian of $\widetilde F_{c,\delta}$ ] \label{lemma_Jacobian_sing}
	Let us fix $y \in \widetilde{Y}$ such that  $x= \widetilde F_{c,\delta}(y)\in\widetilde X\smallsetminus S_{x_0}$ and define  $$\nu_{c,y}^\delta(y')= \tilde d_\delta(\tilde f(y),\tilde f(y'))\cdot \mu_{c,y}(y') .$$ Furthermore, let us introduce the positive quadratic forms $k_{c,\delta,y}^X$, $h_{c,\delta, y}^X$ and $h_{c,\delta,y}^Y$ defined on $T_x \widetilde{X}$, $T_x \widetilde{X}$ and $T_y \widetilde{Y}$ respectively. \\
	
	\noindent $k_{c,\delta, y}^X (v,v)= \frac{1}{\nu_{c,y}^\delta(\widetilde{Y})} \int_{\widetilde{Y}} \frac{1}{2} Dd_x [\tilde\rho_{\delta, \tilde f(y')}]^2(v,v) d\mu_{c,y}(y'), \;\;\qquad \forall v \in T_x \widetilde{X}$\\
	
	\noindent $h_{c,\delta,y}^X(v,v)= \frac{1}{\nu_{c,y}^\delta(\widetilde{Y})} \int_{\widetilde{Y}}d_x [\tilde\rho_{\delta,\tilde f(y')}]^2(v,v) d\nu_{c,y}^\delta(y'),\qquad \qquad\forall v \in T_x \widetilde{X}$\\
	
	\noindent $  h_{c,\delta, y}^Y(u,u)=  \frac{1}{\nu_{c,y}^\delta(\widetilde{Y})} \int_{\widetilde{Y}} d_y [(\tilde\rho)_{y'}]^2(u,u) d\nu_{c,y}(y'), \;\;\;\;\;\qquad \qquad \forall u \in T_y \widetilde{Y}$\\		
	and let us denote by $K_{c,\delta,y}^X$, $H_{c,\delta,y}^X$ and $H_{c,\delta,y}^Y$ the corresponding endomorphisms of the tangent spaces $T_x \widetilde{X}$, $T_x{\widetilde{X}}$ and $T_y \widetilde{Y}$ respectively, and associated to the scalar products $(\tilde g_\delta^j)_x$ on $  T_x \widetilde{X}=T_x\widetilde Y_j[g']_j$ (where we are assuming $x\in\widetilde Y_j[g']_j$) and $(\tilde g)_y$ on $T_y \widetilde{Y}$.
	Then, 
	\begin{equation}
	\tr_{\delta}H^X_{c,\delta, y}=1, \qquad \qquad  \tr_{\tilde g}(H^Y_{c,\delta, y})=1
	\end{equation}
	and the eigenvalues of $H^X_{c,\delta y}$ and $H^Y_{c,\delta,y}$ are pinched between $0$ and $1$. Notice that we denoted for simplicity $\tr_{\delta}$ the trace with respect to the Riemannian metric on $\widetilde X\smallsetminus S_{x_0}$ given by the restriction of $d_\delta$.\\
	Moreover, for any $u \in T_y \widetilde{Y}$ and for any $v \in T_x \widetilde{X}$ the following inequality between the quadratic forms holds:
	\begin{equation}
	k_{c,\delta, y}^X(d_y\widetilde F_{c,\delta}(u),v) \le c\cdot h^X_{c,\delta, y}(v,v) \cdot h^{Y}_{c,\delta,y}(u,u),
	\end{equation}
	and 
	\begin{equation} \label{estimate_jacobian_k_h_sing}
	|\jac_y (\widetilde F_{c,\delta})| \le \frac{c^n}{n^{\frac{n}{2}}} \cdot \frac{(\det H^X_{c,\delta,y})^{\frac{1}{2}}}{\det K^X_{c,\delta,y}}.
	\end{equation}
\end{lem}

In the next paragraph we shall obtain an estimate for the Hessian of $\tilde d_{\delta}$ outside the singular set, which together with Lemma \ref{lemma_Jacobian_sing} will give the lower bound of the Minimal Entropy in the reducible case.

\section{The Hessian of $\tilde \rho_{\delta, z}$ on $\widetilde X\smallsetminus S_{x_0}$} By Lemma \ref{lemma_Jacobian_sing}, in order to estimate $\jac \widetilde F_{c,\delta}$ we need to provide a good upper bound for the quadratic form $k_{c,\delta,y}^X$. To this purpose we shall need an estimate of the Hessian of $\tilde{\rho}_{\delta, z}$ where $z\in\widetilde X\smallsetminus S_{x_0}$. This will be done by looking at $Hess_x(\tilde \rho_{\delta,z})$ as  the second fundamental form ${\rm II}_{z,R}(x)$ of the distance sphere $S_{\delta}(z,R)$ with center $z$ and radius $R=\tilde d_{\delta}(z,x)$, computed at $x$.

\begin{lem}[Hessian of $\tilde{\rho}_{\delta,z}$]\label{lemma_regularity_hessian}
	Let us fix  a point $x\in \widetilde{X} \smallsetminus S_{x_0}$ and a point $z\in \widetilde X$. Then $\mathrm{II}_{z,R}(x)= Hess_x(\tilde\rho_{\delta,z})$.
	
\end{lem}

\begin{proof}
	Since $x\in\widetilde X\smallsetminus S_{x_0}$ we know that there exists a unique $i\in\{1,..,k\}$ and a unique $[g]_i\in G/G_i$ such that $x\in\widetilde Y_i[g]_i$. By construction of $\tilde d_{\delta}$ the intersection of the geodesic sphere $S_{\delta}(z,R)$ passing through $x$ with the leaf $\widetilde Y_i[g]_i$ is equal to $S_\delta(P_{i,g}(z), R-r)$ where $r=\tilde d_{\delta}(P_{i,g}(z),z)$.  But $S_{\delta}(P_{i,g}(z), R-r)$ is isometric to a suitable $\tilde g_\delta^i\,$--geodesic sphere of the same radius in $(\widetilde Y_i, \tilde g_\delta^i)$. Thus:
	$$\mathrm{II}_{z,R}(x)=\mathrm{II}_{P_{i,g}(z), R-r}(x)=Hess_{x}(\tilde\rho_{\delta, P_{i,g}(z)}^i)=Hess_{x}(\tilde{\rho}_{\delta,z}),$$
	where in the second equality we denoted by $\tilde{\rho}_{\delta, P_{i,g}(z)}^i$ the Riemannian distance function $\tilde d_{\delta}|_{\widetilde Y_i[g]_i}$ from the point $P_{i,g}(z)$; the third equality comes from Equation \ref{eq_tigger}.
\end{proof}

\section{An approximated version of Rauch's Comparison Theorem}
From now on, we proceed following the same lines of the proof of the irreducible case, paying some extra care due to the presence of the singular points. Recall that we denoted $y_i\in Y_i$ for $i= 1, \dots, k$ the points identified to the singular point $x_0\in X$.

\begin{prop}\label{prop_approximated_Rauch_reducible}
For every $\varepsilon >0$ define $R_{\varepsilon}=\ln\left(\sqrt{\frac{2}{\varepsilon}}\right)$. Assume that $x \in \widetilde X\smallsetminus S_{x_0}$ is such that, for any $i=1, \dots, k$ and $g\in \pi_1(X,x_0)$, the intersection $B_{\widetilde X}(x, R_{\varepsilon}) \cap \widetilde{Y}_i[g]_i$ is either empty or is a hyperbolic ball; then for every $z \in \widetilde X $ and $v \in T_{x}\widetilde{X},$ the following inequality holds:
\begin{equation} \label{inequality_Rauch_reducible}
(1 - \varepsilon)\tilde\rho_{\delta,z}\left(||v||_{\delta}^2- (d_x \tilde\rho_{\delta,z} \otimes d_x \tilde\rho_{\delta,z})(v,v)\right) \le \tilde\rho_{\delta,z} D_xd\tilde\rho_{\delta,z}(v,v).
\end{equation}

\end{prop}

\begin{proof}
	We may assume that $x\in \widetilde Y_i[g]_i$. Recall that we denoted by $P_{g,i}$ the projection onto the convex subspace $\widetilde Y_i[g]_i$. In order to prove Proposition \ref{prop_approximated_Rauch_reducible} we need to distinguish two cases: the case where $P_{i,g}(z)\in B_{\widetilde Y_i[g]_i}(x,R_\varepsilon)$ and the case where $\tilde d_{\delta}(P_{i,g}(z), x)>R_\varepsilon$. In the first case, observe that the second fundamental form $\mathrm{II}_{P_{i,g}(z),r}(x)$ of the ball of radius $r=\tilde d_{\delta,i}(x,P_{i,g}(z))$ ---where we denote with $\tilde d_{\delta, i}$ the Riemannian distance function on $\widetilde Y_i[g]_i$--- coincides with the second fundamental form in the hyperbolic case. Hence, for $v\in T_x\widetilde X$
	$$Hess_x(\tilde{\rho}_{\delta,z})(v,v)=Hess_x(\tilde{\rho}^i_{\delta, P_{i,g}(z)})(v,v)\ge\tilde g_{\delta}^i(v,v)-(d_x\tilde{\rho}^i_{P_{i,g}(z)}(v))^2,$$
	where we used the notation $\tilde \rho_{\delta, w}^i$ for the Riemannian distance function from $w$ in $(\widetilde Y_i[g]_i, \tilde g_\delta^i)$.
	Otherwise, the distance between $P_{i,g}(z)$ and $x$ is greater than $R_{\varepsilon}$, and we can apply Lemma \ref{lem_estimate_Hessian_distance} in restriction to $\widetilde Y_i[g]_i$, thus obtaining:
	$$Hess_x(\tilde{\rho}_{\delta,z})(v,v)\ge (1-\varepsilon)\left(\tilde g_\delta^i(v,v)-(d_x\tilde{\rho}_{\delta,z}(v))^2\right).$$
\end{proof}

\section{The Jacobian estimate} \label{section_jacobian_estimate_reducible}
We are now ready to apply the inequalities found in Lemma \ref{lemma_Jacobian_sing} and Proposition \ref{prop_approximated_Rauch_reducible}, to derive an almost optimal inequality for the Jacobian of the maps $\widetilde F_{c,\delta}$ valid on an appropriately chosen subset of $Y$.

\begin{prop} [Jacobian Estimate, reducible case] \label{proposition_gestimate_jacobian_reducible}
	For every $\varepsilon >0$ there exists $R_{\varepsilon}>0$ such that, if the ball centred at the non-singular point $x=\widetilde F_{c,\delta}(y)\in\widetilde Y_i[g]_i\subset\widetilde X$ of radius $R_\varepsilon$ is such that $B_{\widetilde X}(x, R_{\varepsilon})\cap \widetilde Y_i[g_i]$ is hyperbolic for every $g\in\pi_1(X,x_0)$, and every $i=1,...,k$, the following estimate for the Jacobian holds:
	\begin{equation}
	\label{eq_jacobian_estimate_reducible}
	\lvert \jac \widetilde F_{c,\delta}(y)\rvert \le \frac{1}{\left( 1- \varepsilon\right) ^n} \cdot \left( \frac{c}{n-1} \right)^n .
	\end{equation}
\end{prop}

\begin{proof}
	Let $x=\widetilde F_{c,\delta}(y)$ and assume that $B_{\widetilde X}(x, R_\varepsilon)$ satisfies the assumptions. Then by Proposition \ref{prop_approximated_Rauch_reducible} we know that, at the point $x=\widetilde F_{c,\delta}(y)$ the following inequality holds:
	$$(1 - \varepsilon)\tilde\rho_{\delta,z}\left(||v||_{\delta}^2- (d_x \tilde\rho_{\delta,z} \otimes d_x \tilde\rho_{\delta,z})(v,v)\right) \le \tilde\rho_{\delta,z} D_xd\tilde\rho_{\delta,z}(v,v)$$
	for every $z\in\widetilde X$.
	Integrating in $d(\tilde f_*\mu_{c,y})$ and normalizing by $\nu_{c,y}^\delta(\widetilde Y)$ we get the following inequality between quadratic forms
	$$(1-\varepsilon)\cdot(\tilde g_\delta^i-h_{c,\delta, y}^X)\le k_{c,\delta, y}^X$$
	and hence we obtain the inequality between the symmetric endomorphisms:
	$$(1-\varepsilon)\cdot\left(\Id- H_{c,\delta, y}^X\right) \le K_{c,\delta, y}^X$$
	Now, using this inequality, Lemma \ref{lemma_Jacobian_sing} and Lemma \ref{lemma_algebraic} we obtain:
	$$	\lvert \jac \widetilde F_{c,\delta}(y)\rvert \le\frac{c^n}{n^{\frac{n}{2}}}\cdot\frac{|\det(H_{c,\delta, y}^X)|^{\frac{1}{2}}}{\det(K_{c,\delta,y}^X)}\le$$$$ \le\frac{c^n}{n^{\frac{n}{2}} (1-\varepsilon)^n}\frac{|\det(H_{c,\delta,y}^X)|^{\frac{1}{2}}}{\det(\Id- H_{c,\delta, y}^X)}\le\frac{1}{(1-\varepsilon)^n}\cdot\left(\frac{c}{n-1}\right)^n$$
	which is the desired inequality.
\end{proof}

\section{Proof of Theorem \ref{thm_lower_estimate_reducible}}

The proof of Theorem \ref{thm_lower_estimate_reducible} partly relies on the proof of Theorem \ref{thm_lower_estimate_irreducible}. For every $i=1,...,k$ and every $j=1,..., n_i$ we shall denote by $\{V_{i, \delta, \varepsilon}^{j}\}_{j=1}^{n_i}$ the collection of open subsets of $(Y_i, g_\delta^i)$ constructed in the proof of Theorem \ref{thm_lower_estimate_reducible}, when $Y_i$ does not admit a complete hyperbolic metric. Otherwise we shall choose $V_{i,\delta,\varepsilon}^1=Y_i$. By construction of the metrics $g_\delta^i$ (where $g_\delta^i$ is the ---unique--- hyperbolic metric for every $Y_i$ which is of hyperbolic type) and by Proposition \ref{proposition_gestimate_jacobian_reducible} we have that:
$$|\jac F_{c,\delta}(y)|\le\frac{1}{(1-\varepsilon)^3}\cdot\left(\frac{c}{2}\right)^3$$
for every $y\in\left(\bigcup_{i=1}^k\bigcup_{j=1}^{n_i}F_{c,\delta}^{-1}(V_{i,\delta,\varepsilon}^j)\right)\smallsetminus\{x_0\}$. On the other hand we have that for every fixed $\varepsilon>0$ the volume $\Vol(V_{i,\delta,\varepsilon}^j)$ converges as $\delta\rightarrow 0$ to $\Vol(int(X_{i}^j), hyp_i^j)$. Now observe that:
$$\int_{\left(\bigcup_{i=1}^k\bigcup_{j=1}^{n_i}F_{c,\delta}^{-1}(V_{i,\delta,\varepsilon}^j)\right)\smallsetminus F_{c,\delta}^{-1}\{x_0\}}|\jac F_{c,\delta}(y)|dv_g(y)=$$$$=\int_{\left(\bigcup_{i=1}^k\bigcup_{j=1}^{n_i}V_{i,\delta,\varepsilon}^j\right)\smallsetminus\{x_0\}}\#(F_{c,\delta}^{-1}(x))d\mu_{\delta}(x)=$$
$$=\sum_{i=1}^k\int_{\bigcup_{j=1}^{n_i} V_{i,\delta,\varepsilon}^j}\#(F_{c,\delta}^{-1}(x))dv_{g_{\delta}^i}(x)\ge \sum_{i=1}^k\sum_{j=1}^{n_i}\Vol\left( V_{i,\delta,\varepsilon}^j,g_\delta^i\right)$$
where the first equality comes from the coarea formula and the second from the fact that the $V_{i, \delta, \varepsilon}^j$'s are disjoint. Plugging the inequality for the Jacobian into the formula, we see that:
$$\sum_{i=1}^k\sum_{j=1}^{n_i}\Vol\left( V_{i,\delta,\varepsilon}^j, g_{\delta}^i\right)\le\frac{1}{(1-\varepsilon)^3}\cdot\left(\frac{c}{2}\right)^3\cdot\Vol(Y,g)$$
Now for every fixed $\varepsilon>0$ the left-hand side of the latter inequality converges to $\sum_{i=1}^k\sum_{j=1}^{n_i}\Vol(int(X_{i}^j), hyp_{i}^j)$ as $\delta\rightarrow 0$. Since this holds for every choice of $\varepsilon>0$ it is easily seen that:
\begin{equation}\label{equation_entvol_reducible}
\EntVol(Y,g)\ge 2 \cdot \left(\sum_{i=1}^n\sum_{j=1}^{n_i}\Vol(int(X_{i}^j), hyp_{i}^j)\right)^{1/3}
\end{equation}
As  Inequality (\ref{equation_entvol_reducible}) holds for every Riemannian metric $g$ on $Y$ we have the following estimate from below for $\minent(Y)$:
$$\minent(Y)\ge2\cdot\left(\sum_{i=1}^n\sum_{j=1}^{n_i}\Vol(int(X_{i}^j), hyp_{i}^j)\right)^{\frac{1}{3}}.$$

\newpage 
\appendix
\chapter{Smooth and Riemannian orbifolds}\label{Appendix_orbifolds}
This appendix is devoted to the introduction of smooth and Riemannian orbifolds. For further details, we recommend \cite{boileau2003three} and \cite{kleinerlottorbifolds}.

\subsection{Definition and basic notions}
\begin{defn}
A smooth $n$-orbifold is a metrizable topological space $\mathcal{O}$ endowed with a collection $\left \lbrace  \left(U_i, \tilde{U}_i, \phi_i, \Gamma_i \right) \right \rbrace_i $, called an \textit{atlas}, where for every $i$ $U_i$ is an open subset of $\mathcal{O}$, $\tilde{U}_i$ is an open subset of $\mathbb{R}^{n-1} \times \left[0, \infty \right)$, $\phi_i: \tilde{U}_i \rightarrow U_i$ is a continuous map (called a \textit{chart}) and $\Gamma_i$ is a finite group of diffeomorphisms of $\tilde{U}_i$ satisfying the following conditions:

\begin{enumerate}
	\item The $U_i$'s cover $\mathcal{O}$.
	\item Each $\phi_i$ factors through a homeomorphism between $\tilde{U}_i / \Gamma_i$ and $U_i$.
	\item The charts are  compatible in the following sense: for every $x \in \tilde{U}_i$ and $y \in \tilde{U}_j$ with $\phi_i(x)= \phi_j(y)$, there is a diffeomorphism $\psi$ between a neighbourhood $V$ of $x$ and a neighbourhood $W$ of $y$ such that $\phi_j(\psi(z))= \phi_i(z)$ for all $z \in V$.
\end{enumerate}

\end{defn}

For convenience, we shall assume that the atlas is maximal.

We remark that an orbifold $\mathcal{O}$ is a manifold if all the groups $\Gamma_i$ are trivial. 
Given an orbifold $\mathcal{O}$, we denote by $\lvert \mathcal{O} \rvert$ its underlying space, \textit{i.e.} the topological space obtained by neglecting the orbifold structure.\\

Let $x \in \mathcal{O}$ be any point; then the local group of $\mathcal{O}$ at $x$ is the group $\Gamma_x$ defined as follows: pick any chart $(U, \tilde{U}_i, \phi, \Gamma)$, where $x \in U$.  Then $\Gamma_x$ is the stabilizer of any point of $\phi^{-1}(x)$ under the action of $\Gamma$. It turns out that the group $\Gamma_x$ is well defined up to isomorphism. 
We say that the point $x$ is \textit{regular} if $\Gamma_x$ is trivial, otherwise $x$ is \textit{singular}.
Notice that the set of singular points is empty if and only if $\mathcal{O}$ is a manifold.\\

There are several notions that can be defined in the orbifold context, extending in a straightforward manner the analogous notion given for manifolds.

The \textit{boundary} of $\mathcal{O}$, denoted $\partial \mathcal{O}$, is the set of points $x \in \mathcal{O}$ such that there exists a chart at $x$, $\phi_i: \tilde{U}_i \rightarrow U_i$ such that $\phi_i^{-1}(x) \subset \mathbb{R}^{n-1} \times \left \lbrace0 \right \rbrace$.
The orbifold $\mathcal{O} \setminus \partial\mathcal{O}$ is called the \textit{interior} of $\mathcal{O}$.
We call $\mathcal{O}$ \textit{orientable} if it has an atlas such that each $\phi_i$ and every element of each $\Gamma_i$ are orientation preserving.

\begin{oss}
Let $\mathcal{O}$ be an orbifold with boundary, and assume that all the singular points are conical points, \textit{i.e.} the group $\Gamma_i \simeq \mathbb{Z}_{k}$ acts by rotation.
Let $x$ be a boundary point. Then, there exists a chart $(\tilde{U}_i, U_i. \phi_i, \Gamma_i)$ at $x$ such that $\phi_i^{-1}(x) \subset \mathbb{R}^{n-1} \times \{0\}$, and the group $\Gamma_i$ is trivial. 
Thus, there is a neighbourhood of the boundary such that the local groups are trivial.

\end{oss}

\subsection{Maps between orbifolds}
A map between two orbifolds $\mathcal{O}$ and $\mathcal{O}'$ is a continuous map $f: \lvert \mathcal{O} \rvert \rightarrow \lvert \mathcal{O}' \rvert$ such that for every $x \in \mathcal{O}$ there are charts $\phi_i: \tilde{U}_i \rightarrow U_i \ni x$ and $\phi_j': \tilde{U}'_j \rightarrow U_j'$ such that $f(U_i) \subset U_j'$ and the restriction $f|_{U_i}$ can be lifted to a smooth map $\tilde{f}: \tilde{U}_i \rightarrow \tilde{U}'_j$ which is equivariant with respect to some homomorphism $\Gamma_i \rightarrow \Gamma'_j$.
A map $f: \mathcal{O} \rightarrow \mathcal{O}'$ is an \textit{immersion} (resp. a \textit{submersion}) if the lifts $\tilde{f}$ are immersions (resp. submersions).
An \textit{embedding} is an immersion whose underlying map is a homeomorphism with the image.
A \textit{diffeomorphism} is a surjective embedding.

\subsection{Orbifold covering}
A \textit{covering}  of an orbifold $\mathcal{O}$ is an orbifold $\tilde{\mathcal{O}}$ with a continuous map $p: \lvert \tilde{\mathcal{O}} \rvert \rightarrow \lvert \mathcal{O} \rvert$, called a \textit{covering map}, such that every point $x \in \mathcal{O}$ has a neighbourhood $U$ with the following property: for each component $V$ of $p^{-1}(U)$ there is a chart $\phi: \tilde{V} \rightarrow V$ such that $p \circ \phi$ is a chart.\\

We say that an orbifold is \textit{good} if it is covered by a manifold, \textit{bad} otherwise. If an orbifold is finitely covered by a manifold we say that it is \textit{very good}.\\

The notions of deck transformation group and universal covering extend to the context of orbifolds, see \cite{thurston1997three} and \cite{boileau2003three}.\\
The \textit{fundamental group} of $\mathcal{O}$, denoted by $\piorb(\mathcal{O})$, is the deck transformation group of its universal covering.\\

There is an interpretation of the fundamental group in terms of loops, keeping track of the passage through the singular locus.
Furthermore, a continuous map between two orbifolds gives rise to a map between the corresponding orbifold fundamental groups.

\subsection{Riemannian orbifolds}

A Riemannin metric on an orbifold $\mathcal{O}$ is given by an atlas for $\mathcal{O}$ along with a collection of Riemannian metric on the $\tilde{U}_i$'s so that:
\begin{enumerate}
\item $\Gamma_i$ acts isometrically on $\tilde{U}_i$ for every $i$;

\item  The diffeomorphism $\psi$ is an isometry.
\end{enumerate}

We say that a map $f: \mathcal{O} \rightarrow \mathcal{O}'$ between Riemannian orbifolds is a Riemannian submersion (resp. a Riemannian immersion) if the lift $\tilde{f}$ is a \textit{Riemannian submersion} (resp. a \textit{Riemannian immersion}) equivariant with respect to the isometric actions of the local groups.

\begin{rmk}
Let $\mathcal{O}$ and $\mathcal{O}'$ be two orbifolds that are global quotients of two Riemannian manifolds $M$ and $M'$ with respect to groups of isometries $\Gamma$ and $\Gamma'$ acting properly.
Then, a Riemannian submersion between $\mathcal{O}$ and $\mathcal{O}'$ is equivalent to a Riemannian submersion between $M$ and $M'$ which is equivariant with respect to the action of $\Gamma$ and $\Gamma'$.
\end{rmk}

\chapter{Proof of Proposition \ref{prop_interpolation}} \label{appendix_proposition}

\begin{lem} \label{Lemma_1_interpolation} 
Let $\Phi \in \mathcal{C}^1 \left( -\infty,0 \right] \cap \mathcal{C}^1 \left[ 0, + \infty\right) \cup \mathcal{C}^0(\mathbb{R})$ such that:
\begin{enumerate}[label=\roman{*}., ref=(\roman{*})]
	\item $ \Phi'(t) \ge 0$ for every $t \not =0$;
	\item $|\Phi(0)| \le M;$
	\item $0 \le \Phi'(0)^- < \Phi'(0)^+ $
\end{enumerate}
Then there exists $\varepsilon >0$, $\varepsilon$ arbitrarily small, such that there exists a function $\Phi_{\varepsilon} \in \mathcal{C}^1(\mathbb{R})$ satisfying the following properties:
\begin{enumerate}
	\item [a.] $\Phi_{\varepsilon}(t)= \Phi(t)$ for $t \in \mathbb{R} \setminus \left(- \varepsilon, \varepsilon \right)$ and $\Phi \le \Phi_{\varepsilon} \le \Phi(\varepsilon)$ for $ t \in \left( -\varepsilon, \varepsilon\right)$
	\item [b.] $\Phi'(-\varepsilon) \le \Phi'_{\varepsilon}(t) \le \Phi'(\varepsilon)$ for every $t \in \left(-\varepsilon, \varepsilon\right)$
	\item  [c.] $\lvert \int_{-\varepsilon}^{\varepsilon} \Phi_{\varepsilon}(s) ds \rvert \le 2M\varepsilon$
\end{enumerate}
\end{lem}
\begin{proof}

Property $iii$ implies that there exists $\varepsilon_1$ such that, for $\varepsilon < \varepsilon_1$, inequality $\Phi'(-\varepsilon) < \Phi'(\varepsilon)$ holds. Let us choose $\varepsilon < \varepsilon_1$; then it si possible to construct a function $\Phi_{\varepsilon}$ such that $\Phi_{\varepsilon}(\pm \varepsilon)= \Phi(\pm \varepsilon)$, $\Phi_{\varepsilon}'(\pm \varepsilon)= \Phi'(\pm \varepsilon)$ and $\Phi'(-\varepsilon) \le \Phi'(t) \le \Phi'(\varepsilon)$.
Furthermore, we can choose $\Phi_{\varepsilon}$ such that $\Phi(-\varepsilon) \le \Phi_{\varepsilon}(t) \le \Phi(\varepsilon)$.
Moreover, since $|\Phi(0)| <M$ there exists $\varepsilon_2$ such that for $\varepsilon < \varepsilon_2$ inequality $|\Phi(t)| \le M$ holds for $t \in (-\varepsilon, \varepsilon)$. Thus, provided $\varepsilon < \min (\varepsilon_1,\varepsilon_2)$:
$$ \left \lvert \int_{-\varepsilon}^{\varepsilon} \Phi_{\varepsilon}(s) ds \right \rvert \le \int_{-\varepsilon}^{\varepsilon} \lvert  \Phi_{\varepsilon} (s)\rvert ds \le 2M \varepsilon$$
\end{proof}

\begin{lem} \label{Lemma_2_interpolation} 
Let $\varphi \in \mathcal{C}^2\left(- \infty, 0 \right] \cap \mathcal{C}^2 \left[0, + \infty \right) \cap \mathcal{C}^1(\mathbb{R})$ with $\varphi(0)  \ge 0$ such that:
\begin{enumerate}[label=\roman{*}., ref=(\roman{*})]
	\item $\varphi$ convex;
	\item $|\varphi'(0)| \le M$;
	\item $0 \le \varphi''(0)^- < \varphi''(0)^+$
\end{enumerate}
Then there exists $\varepsilon >0$
arbitrarily small and $\varphi_{\varepsilon} \in \mathcal{C}^2(\mathbb{R})$ such that:
\begin{enumerate}
	\item [A.] $\varphi'(t) \le \varphi'_{\varepsilon}(t) \le \varphi'(\varepsilon)$ for every $t \in \left(-\varepsilon, \varepsilon \right)$ ;
	\item [B.] $\varphi''(-\varepsilon) \le \varphi_{\varepsilon}''(t) \le \varphi''(\varepsilon)$ for $t \in (-\varepsilon, \varepsilon)$;
	\item [C.]
	$	\left \lbrace \begin{array}{l}
	\varphi_{\varepsilon}= \varepsilon \quad t\le -\varepsilon \\
	\varphi_{\varepsilon}= \varepsilon+c \quad t\ge \varepsilon, |c|<2M\varepsilon\\
	|\varphi_{\varepsilon}- \varphi| < 3M \varepsilon
	\end{array}  \right.$
\end{enumerate}
\end{lem}
\begin{proof}
Denote by $\Phi= \varphi'$. Then, $\Phi$ satisfies the hypotheses of Lemma \ref{Lemma_1_interpolation} and applying it we get a function $\Phi_{\varepsilon}$ such that properties \textit{a,b,c} of Lemma \ref{Lemma_1_interpolation} hold.
Let us define: $$\varphi_{\varepsilon}(t)= \int_{-\varepsilon}^t \Phi_{\varepsilon}(s) ds + \varphi(-\varepsilon)$$
We show that $\varphi_{\varepsilon}$ satisfies property $C$.\\ \medskip
For $t \le -\varepsilon$,
\small
$$\varphi_{\varepsilon}(t)= \int_{-\varepsilon}^t \Phi_{\varepsilon}(s) ds + \varphi(-\varepsilon)= - \int_{t}^{-\varepsilon} \Phi_{\varepsilon}(s) ds+ \varphi(-\varepsilon)=\cancel{ -\varphi(-\varepsilon)}+ \varphi(t)+ \cancel{\varphi(-\varepsilon)}= \varphi(t)$$
\normalsize
Thus $\varphi_{\varepsilon} = \varphi$ for $t \le -\varepsilon$.\\ 
\medskip
For $t \ge \varepsilon:$
\begin{multline*}
\varphi_{\varepsilon}(t)= \int_{-\varepsilon}^t \Phi_{\varepsilon}(s) ds + \varphi(-\varepsilon)= \int_{-\varepsilon}^{\varepsilon} \Phi_{\varepsilon}(s) ds + \int_{\varepsilon}^{t} \Phi(s) ds + \varphi(-\varepsilon) =\\
= \int_{-\varepsilon}^{\varepsilon} \Phi_{\varepsilon}(s) ds + \int_{\varepsilon}^{t} \varphi'(s)ds + \varphi(-\varepsilon)= 
\int_{-\varepsilon}^{\varepsilon} \Phi_{\varepsilon}(s) ds + \phi(t)-\phi(\varepsilon)+\varphi(-\varepsilon)=\\ = \phi(t)+c_1+c_2
\end{multline*}

where $c_1= \int_{-\varepsilon}^{\varepsilon} \Phi_{\varepsilon}(s) ds$, $c_2= \varphi(-\varepsilon)- \varphi(\varepsilon)$.
\bigskip

Furthermore,
$$\lvert c_1+c_2 \rvert \le  \left \lvert c_1  \right \rvert + \left \lvert c_2 \right \rvert \le \left \lvert \int_{-\varepsilon}^{\varepsilon} \Phi_{\varepsilon}(s)ds \right \rvert + \left \lvert \varphi(-\varepsilon)- \varphi(\varepsilon) \right \rvert \le 4M\varepsilon$$
where Inequality $\int_{-\varepsilon}^{\varepsilon} \Phi_{\varepsilon}(s) ds \le 2M \varepsilon$ follows by Lemma \ref{Lemma_1_interpolation}.
\bigskip 

For $t \in \left(-\varepsilon, \varepsilon \right)$, the same computation just held gives:
$$ \lvert \varphi_{\varepsilon}(t)- \varphi(t )\rvert \le \left \lvert \int_{-\varepsilon}^t \Phi_{\varepsilon}(s) ds \right \rvert + \left \lvert  \varphi(-\varepsilon)- \varphi(t) \right \rvert \le 4M \varepsilon$$

Finally, we observe that property \textit{a} of Lemma \ref{Lemma_1_interpolation} implies the thesis \textit{A}, and part \textit{B} of the thesis follows from property \textit{b}.

\end{proof}
\begin{prop}
For every fixed $\ell, \delta >0$ there exists $\varepsilon >0$ arbitrarily small and $\varphi_{\varepsilon} \in \mathcal{C}^2(\mathbb{R})$ such that:
\begin{enumerate}
	\item $\varphi_{\varepsilon}$ is not increasing and convex, such that
	
	$ \left \lbrace \begin{array}{l}
	\varphi_{\varepsilon}(t)= \varphi_0(t)= \ell e^{-t} \, \, \mbox{ for } 	 t \le -\varepsilon \\
	\varphi_{\varepsilon}(t)= \ell' \, \, \mbox{ for } 	 t \ge t_{\delta} + \varepsilon
	\end{array} \right.$\\
	where:
	$$ \ell \frac{\sqrt{\delta (1 + \delta)}}{1+ 2 \delta} \le \ell' \le 4\ell \frac{\sqrt{\delta(1+ \delta)}}{1+2\delta}$$

	\item $\frac{(\varphi_{\varepsilon}')^2}{\varphi_{\varepsilon}^2} < (1+ 2 \delta)^2 \frac{(\varphi_0')^2}{\varphi_0^2}$ for every $t \in \mathbb{R}$
	\item $\frac{\varphi_{\varepsilon}''}{\varphi_{\varepsilon}} \le (1+2\delta)^2$ for every $t \in (-\varepsilon, t_{\delta}+ \varepsilon)$
\end{enumerate}
where $t_{\delta}= \frac{1}{2(1+2 \delta)} \ln \left(1+ \frac{1}{\delta}\right)$

\end{prop}
\begin{proof}
The function $\varphi_0(t)= \ell e^{-t}$ satisfies the differential equation $\frac{\varphi_0''}{\varphi}=1$, with initial conditions
$$ \left \lbrace \begin{array}{l}
\varphi_0(0)= \ell \\
\varphi_0'(0)= -\ell
\end{array} \right.$$

Let us denote by $\varphi_{\delta}$ the solution of the differential equation $\frac{\varphi_{\delta}''}{\varphi_{\delta}}= \left(1+ 2 \delta\right)^2$ with the same initial conditions.
Namely,

$$ \varphi_{\delta}(t)= \ell \left( \frac{1+ \delta}{1+ 2 \delta }\right) e^{-(1+2 \delta)t} + \ell \frac{\delta}{1+ 2 \delta} e^{(1+2\delta)t}$$

Straightforward computations show that, for $\delta >0$, the function $\varphi_{\delta}$ is more convex than $\varphi_0$, and $\varphi_{\delta}$ has a unique stationary point, which turns to be a minimum, at $t_{\delta}= \frac{1}{2(1+ 2 \delta)} \ln \left(1+ \frac{1}{\delta}\right)$. We denote by $\ell_{\delta}$ the corresponding minimal value.

We define the following function:

$$ \varphi(t)= \left \lbrace
\begin{array}{ll}
\varphi_0(t) & \text{for } t<0 \\
\varphi_{\delta}(t) & \text{for } t\in \left[0,\right] \\
\ell_{\delta} & \text{for } t > t_{\delta}
\end{array}
\right.$$
We observe that $\varphi(t) \in \mathcal{C}^1(\mathbb{R}) \cap \mathcal{C}^2\left(\left(- \infty,0\right] \right)\cap \mathcal{C}^2(\left[0, t_{\delta}\right]) \cap \mathcal{C}^2 ( \left[ t_{\delta}, + \infty\right))$, but is not $\mathcal{C}^2(\mathbb{R})$. Furthermore, $\varphi$ is not increasing, convex and satisfies condition $(1)$.
Moreover, we observe that, for every $\delta >0$:
\begin{equation} \label{equation_interpolation_first_derivative} \frac{\varphi_{\delta}'^2}{\varphi_{\delta}^2} \le \frac{\varphi_0'^2}{\varphi_0^2}.
\end{equation}
Indeed, on the interval $\left[0, t_{\delta}\right]$ the following inequality holds:
$$ \frac{\varphi_{\delta}''}{\varphi_{\delta}} \ge \frac{\varphi_0''}{\varphi_0}$$
and $\varphi_{\delta}(0)= \varphi_0(0)$, $\varphi_{\delta}'(0)= \varphi'_{\delta}(0)$.
Then, by Sturm-Liouville Theorem, $0 > \varphi_{\delta}' \ge \varphi_0'$ and $\varphi_{\delta} \ge \varphi_0 >0$ and
$$ \left \lvert  \frac{\varphi_{\delta}'}{\varphi} \right\rvert \le  \left\lvert  \frac{\varphi_0'}{\varphi} \right\rvert.$$
The latter Inequality implies Equation \ref{equation_interpolation_first_derivative}.
In order to obtain a $\mathcal{C}^2$ function satisfying the required properties we will apply (twice) Lemma \ref{Lemma_2_interpolation}, as follows.\\ \medskip

Firstly, we observe that $\varphi$ satisfies the hypotheses of Lemma \ref{Lemma_2_interpolation}; hence, there exists a function $\varphi_{\varepsilon}$, enjoying properties $(A), (B), (C)$. In particular, $\varphi_{\varepsilon}$ is convex, not increasing, and
$$ \left \lbrace 
\begin{array}{ll}
\varphi_{\varepsilon}= \varphi & \mbox{for } t\le \varepsilon\\
\varphi_{\varepsilon}= \varphi+c(\varepsilon) & \mbox{for } t \ge \varepsilon \\
\lvert	\varphi_{\varepsilon}- \varphi \rvert \le 3M \varepsilon & \mbox{for } t \in \left[-\varepsilon, \varepsilon\right]
\end{array}
\right.$$
We claim that $\varphi_{\varepsilon}$ satisfies properties $2)$ and $3)$.
Indeed, Lemma \ref{Lemma_2_interpolation} implies:
$$ \varphi'(t) \le \varphi'_{\varepsilon}(t) \le \varphi'(\varepsilon)$$
for $t \in \left(-\varepsilon, \varepsilon\right) $.
Thus:
\small
$$ \left(  \frac{\varphi_{\varepsilon}'(t)}{\varphi_{\varepsilon}(t)}\right)^2 \le \left( \frac{\varphi'(t)}{\varphi_{\varepsilon}(t)}\right)^2 \le \left(\frac{\varphi'(t)}{\varphi(t)-3M\varepsilon}\right)^2 $$
\normalsize
Now choose $\varepsilon \le \min \left \lbrace \frac{1}{2}, \frac{\delta \ell}{6M \sqrt{e}}\right \rbrace$ so that, in particular, $\varepsilon \le \frac{\delta \varphi(\varepsilon)}{6M}$; then we have
\small
$$ \left(\frac{\varphi'(t)}{\varphi(t)-3M\varepsilon}\right)^2 \le \left(\frac{\varphi'(t)}{\left(1- \frac{\delta}{2}\right) \varphi(t)}\right) ^2$$
\normalsize
Thus,
\small $$ \left( \frac{\varphi_{\varepsilon}'(t)}{\varepsilon_{\varepsilon}(T)}\right)^2 \le  \left(\frac{\varphi'(t)}{\varphi(t)-3M\varepsilon}\right)^2 \le \left(\frac{\varphi'(t)}{\left(1- \frac{\delta}{2}\right) \varphi(t)}\right) ^2 \le \frac{1}{\left(1-\frac{\delta}{2}\right)^2}\le \left( 1+4 \delta \right)$$
\normalsize
where the latter inequality holds provided that $\delta < \frac{1}{2}$.
Thus, $\varphi_{\varepsilon}$ satisfies Property $2$.\\ \medskip
We claim that $\varphi_{\varepsilon}$ also satisfies Property $3$. Lemma \ref{Lemma_2_interpolation} implies:
$$ \varphi''(-\varepsilon) \le \varphi_{\varepsilon}''(t) \le \varphi''(\varepsilon)$$
Hence we have: 
\small 
$$ \frac{\varphi_{\varepsilon}''(t)}{\varphi_{\varepsilon}(t)} \le \frac{\varphi''(\varepsilon)}{\varphi(t)-3M\varepsilon} \le \frac{\varphi''(\varepsilon)}{\varphi(\varepsilon)-3M\varepsilon}$$
\normalsize

We choose $\varepsilon \le \min \left \lbrace \frac{1}{2} , \frac{\delta^2 \ell}{2M \sqrt{e}}\right \rbrace$, so that $3M\varepsilon \le \delta^2 \varphi(\varepsilon)$. With this choice, 
\small 
$$ \frac{\varphi''(\varepsilon)}{\varphi(\varepsilon)-3M \varepsilon} \le \frac{\varphi''(\varepsilon)}{\varphi(\varepsilon)- \delta^2 \varphi(\varepsilon)}
$$
\normalsize
and so
\small
$$ \frac{\varphi_{\varepsilon}''(t)}{\varphi_{\varepsilon}(t)} \le \frac{\varphi''(\varepsilon)}{\varphi(\varepsilon)-3M \varepsilon} \le \frac{\varphi''(\varepsilon)}{\varphi(\varepsilon)- \delta^2 \varphi(\varepsilon)}= \frac{\varphi'_{\delta}(\varepsilon)}{\varphi_{\delta}(\varepsilon)} \frac{1}{(1- \delta)^2}= \left( \frac{1+2 \delta}{ 1-\delta  } \right)^2 \le \left(1+ 4 \delta\right)^2$$
\normalsize
where the latter inequality holds provided that $\delta \le \frac{1}{2}$.\\ \bigskip

\noindent Let us define $\psi_{\varepsilon}(t)= \varphi_{\varepsilon}(t_{\delta}-t)$. We observe that $\psi_{\varepsilon}(0)= \varphi_{\varepsilon}(t_{\delta})$, $\psi_{\varepsilon}(\varepsilon)= \varphi_{\varepsilon}(t_{\delta}- \varepsilon)$ and $\psi_{\varepsilon}(-\varepsilon)= \varphi_{\varepsilon}(t_{\delta}+ \varepsilon)$.
Furthermore, $\psi_{\varepsilon}$ satisfies the hypothesis of Lemma \ref{Lemma_2_interpolation}. Then there exists $\varepsilon'$ sufficiently small, and a convex function $\psi_{\varepsilon, \varepsilon'}(t)$ such that:
\small
$$ \left \lbrace 
\begin{array}{ll}
\psi_{\varepsilon, \varepsilon'}= \psi_{\varepsilon} & \mbox{for } t\le \varepsilon'\\
\psi_{\varepsilon, \varepsilon'}= \psi_{\varepsilon}+c(\varepsilon') & \mbox{for } t \ge \varepsilon' \\
\lvert	\psi_{\varepsilon, \varepsilon'}- \psi_{\varepsilon}\rvert \le 3M \varepsilon' & \mbox{for } t \in \left[-\varepsilon,' \varepsilon'\right]
\end{array}
\right.$$
\normalsize
Furthermore, $\psi_{\varepsilon, \varepsilon'}$ satisfies:
\small
\begin{equation} \label{eq_1_lemma_2_at_tdelta}
\psi_{\varepsilon}'(t) \le \psi_{\varepsilon, \varepsilon'}(t) \le \psi_{\varepsilon}'(\varepsilon') \Longleftrightarrow - \varphi_{\varepsilon}'(t_{\delta}-t) \le \psi_{\varepsilon, \varepsilon'}(t) \le - \varphi_{\varepsilon}'(t_{\delta}-\varepsilon')
\end{equation}
\begin{equation} \label{eq_2_lemma_2_at_tdelta}
\psi_{\varepsilon}''(- \varepsilon') \le \psi_{\varepsilon, \varepsilon'}(t) \le \psi_{\varepsilon}''(\varepsilon') \Longleftrightarrow \varphi_{\varepsilon}''(t+\varepsilon') \le \psi_{\varepsilon, \varepsilon'}''(t) \le \varphi_{\varepsilon}''(t_{\delta}-t)
\end{equation}
\normalsize
In order to prove that $\psi_{\varepsilon, \varepsilon'}$ satisfies Property $2$, we perform the following estimates, exploiting Equation \ref{eq_1_lemma_2_at_tdelta}:
\small
$$ \left( \frac{\psi_{\varepsilon, \varepsilon'}'(t)}{\psi_{\varepsilon, \varepsilon'}} \right)^2 \le \left( \frac{\psi_{\varepsilon}'(\varepsilon')}{\psi_{\varepsilon, \varepsilon'}(t)}\right)^2 \le \left(  \frac{ \psi'_{\varepsilon}(\varepsilon')}{ \varphi(t_{\delta}-t)+ c_{\varepsilon}-3M \varepsilon} \right)^2 \le \left( \frac{\psi'_{\varepsilon'}}{\psi(t_{\delta}-t)-\frac{\varphi_{\delta}(t_{\delta})}{100}-3M \varepsilon'}\right)^2$$
\normalsize
where the latter inequality holds provided that we have chosen $\varepsilon$ sufficiently small so that $\lvert c_{\varepsilon}\rvert \le \frac{\delta \varphi_{\delta}(t_{\delta})}{100}$.

We remark that $\min_{t \in \left( t_{\delta}-\varepsilon', t_{\delta}+ \varepsilon'\right)}= \varphi(t_{\delta})$. 
We choose $\varepsilon'$ sufficiently small so that $3M \varepsilon' \le \frac{\delta \varphi(t_{\delta})}{100}$; thus we have:
\small
$$ \varphi(t_{\delta}-t) + c_{\varepsilon}-3M \varepsilon'  \ge \varphi(t_{\delta})- \frac{\delta \varphi (t_{\delta})}{100}-3M \varepsilon' \ge \varphi(t_{\delta}) \left(1- \frac{\delta}{50}\right) $$
\normalsize
Hence, recalling Inequality \ref{eq_1_lemma_2_at_tdelta}
\small
$$ \left( \frac{\psi_{\varepsilon, \varepsilon'}(t)}{\psi_{\varepsilon, \varepsilon'}(t)}\right)^2 \le \left( \frac{\psi'_{\varepsilon'}}{\psi(t_{\delta}-t)-\frac{\varphi_{\delta}(t_{\delta})}{100}-3M \varepsilon'}\right)^2 \le \left(  \frac{-\varphi_{\delta}'(t_{\delta}- \varepsilon')}{\left(1- \frac{\delta}{ro}\right) \varphi(t_{\delta})}\right)^2 \le$$ $$ \le  \frac{1}{4 \left(1-\frac{\delta}{50}\right)^2} \left \lvert  \frac{\varphi_0'}{\varphi_0} \right \rvert \le  \frac{1}{4 \left(1- \frac{\delta}{50}\right)^2} \le 1$$
\normalsize
where the second inequality holds if we choose $\varepsilon'$ small enough so that $$- \varphi_{\delta}'(t_{\delta}- \varepsilon') < \frac{1}{2} \varphi(t_{\delta})$$

Finally, we check that $\psi_{\varepsilon, \varepsilon'}$ satisfies Property $3$. Exploiting Equation \ref{eq_2_lemma_2_at_tdelta} we obtain:
\small
$$ \frac{\psi_{\varepsilon, \varepsilon'}''(t)}{\psi_{\varepsilon, \varepsilon'}(t)} \le \frac{\varphi_{\varepsilon}''(t_{\delta}- \varepsilon')}{\varphi_{\varepsilon}(t_{\delta}-t)-3M \varepsilon'} \le \frac{\varphi_{\varepsilon}''(t_{\delta}- \varepsilon')}{\varphi(t_{\delta}+ \varepsilon')-3M\varepsilon'- \lvert c_{\varepsilon} \rvert} $$
\normalsize
We choose $\varepsilon'$ sufficiently small so that $$\lvert c_{\varepsilon} \rvert +3M \varepsilon' \le \frac{\delta}{10} \varphi(t_{\delta)}$$ Then:
\small 
$$\frac{\psi_{\varepsilon, \varepsilon'}''(t)}{\psi_{\varepsilon, \varepsilon'}(t)} \le \frac{\varphi_{\varepsilon}''(t_{\delta}- \varepsilon')}{\varphi(t_{\delta}+ \varepsilon')-3M\varepsilon'- \lvert c_{\varepsilon} \rvert} \le \frac{\varphi''(t_{\delta}- \varepsilon')}{\varphi(t_{\delta}) \left( 1- \frac{\delta}{10}\right)} \le \frac{(1+2 \delta)^2}{1- \frac{\delta}{10}}
$$
\normalsize
where the latter inequality holds provided that $\varepsilon'$ is sufficiently small so that $$\varphi(t_{\delta}- \varepsilon') \left(1 - \frac{\delta}{10}\right) \le \varphi(t_{\delta})$$\\ \bigskip
Finally, we call 
$$ \widetilde{\psi}_{\varepsilon, \varepsilon'}(t)= \psi_{\varepsilon, \varepsilon'}-c_{\varepsilon'}$$
so that $\widetilde{\psi}_{\varepsilon, \varepsilon'}(t)= e^{t_{\delta}-t}$ for $t \le t_{\delta}- \varepsilon'$, and we define
$$ \psi_{\varepsilon, \varepsilon'}(t)= \widetilde{\psi}_{\varepsilon, \varepsilon'}(t- t_{\delta})$$
By construction, this function satisfies the required properties.

\end{proof}

\addcontentsline{toc}{chapter}{Bibliography}
\bibliographystyle{alpha}
\bibliography{erika1}

\end{document}